\documentclass[11pt,a4paper,reqno]{amsart}
\usepackage[english]{babel}
\usepackage[T1]{fontenc}
\usepackage{verbatim}
\usepackage{palatino}
\usepackage{amsmath}
\usepackage{mathabx}
\usepackage{amssymb}
\usepackage{amsthm}
\usepackage{amsfonts}
\usepackage{graphicx}
\usepackage{esint}
\usepackage{color}
\usepackage{mathtools}
\usepackage{overpic}

\usepackage[colorlinks = true, citecolor = black]{hyperref}
\author{Tuomas Orponen}
\title{On arithmetic sums of Ahlfors-regular sets}
\address{Department of Mathematics and Statistics\\ University of Jyv\"askyl\"a,
P.O. Box 35 (MaD)\\
FI-40014 University of Jyv\"askyl\"a\\
Finland}
\email{tuomas.t.orponen@jyu.fi}
\date{\today}
\subjclass[2010]{11B30 (primary) 28A80 (secondary)}
\keywords{Ahlfors-regular sets, sum-product problem, Hausdorff dimension}
\thanks{T.O. is supported by the Academy of Finland via the projects \emph{Quantitative rectifiability in Euclidean and non-Euclidean spaces} and \emph{Incidences on Fractals}, grant Nos. 309365, 314172, 321896. T.O. is also supported by the University of Helsinki via the project \emph{Quantitative rectifiability of sets and measures in Euclidean spaces and Heisenberg groups}, project No. 7516125.}

\newcommand{\R}{\mathbb{R}}

\newcommand{\N}{\mathbb{N}}

\newcommand{\Z}{\mathbb{Z}}
\newcommand{\tn}{\mathbb{P}}

\newcommand{\calD}{\mathcal{D}}
\newcommand{\calH}{\mathcal{H}}

\newcommand{\calS}{\mathcal{S}}

\newcommand{\spt}{\operatorname{spt}}
\newcommand{\Hd}{\dim_{\mathrm{H}}}

\newcommand{\diam}{\operatorname{diam}}

\newcommand{\dist}{\operatorname{dist}}

\newcommand{\m}{\mathfrak{m}}
\newcommand{\M}{\mathfrak{M}}

\def\Barint_#1{\mathchoice
          {\mathop{\vrule width 6pt height 3 pt depth -2.5pt
                  \kern -8pt \intop}\nolimits_{#1}}%
          {\mathop{\vrule width 5pt height 3 pt depth -2.6pt
                  \kern -6pt \intop}\nolimits_{#1}}%
          {\mathop{\vrule width 5pt height 3 pt depth -2.6pt
                  \kern -6pt \intop}\nolimits_{#1}}%
          {\mathop{\vrule width 5pt height 3 pt depth -2.6pt
                  \kern -6pt \intop}\nolimits_{#1}}}

\numberwithin{equation}{section}

\theoremstyle{plain}
\newtheorem{thm}[equation]{Theorem}
\newtheorem*{"thm"}{"Theorem"}
\newtheorem{conjecture}[equation]{Conjecture}
\newtheorem{lemma}[equation]{Lemma}

\newtheorem{ex}[equation]{Example}

\newtheorem{proposition}[equation]{Proposition}
\newtheorem{question}{Question}

\theoremstyle{definition}

\newtheorem{definition}[equation]{Definition}

\newtheorem{notation}[equation]{Notation}

\theoremstyle{remark}
\newtheorem{remark}[equation]{Remark}

\newtheorem{warning}[equation]{Warning}

\newcommand{\nref}[1]{(\hyperref[#1]{#1})}

\DeclareMathSymbol{\intop}  {\mathop}{mathx}{"B3}

\begin{document} 

\begin{abstract} Let $A,B \subset \R$ be closed Ahlfors-regular sets with dimensions $\Hd A =: \alpha$ and $\Hd B =: \beta$. I prove that 
\begin{displaymath} \Hd [A + \theta B] \geq \alpha + \beta \cdot \tfrac{1 - \alpha}{2 - \alpha} \end{displaymath}
for all $\theta \in \R \, \setminus \, E$, where $\Hd E = 0$.
\end{abstract}

\maketitle

\tableofcontents

\section{Introduction}

This paper contains the following sum-product result for Ahlfors-regular subsets of $\R$: 
\begin{thm}\label{mainNonTec} Let $A,B \subset \R$ be closed Ahlfors-regular sets with $\Hd A = \alpha \in [0,1]$ and $\Hd B = \beta \in [0,1]$. Then 
\begin{displaymath} \Hd [A + \theta B] \geq \alpha + \beta \cdot \tfrac{1 - \alpha}{2 - \alpha} \end{displaymath}
for all $\theta \in \R \, \setminus \, E$, where $\Hd E = 0$. \end{thm}

The paper also contains a $\delta$-discretised version of Theorem \ref{mainNonTec}, see Theorem \ref{main}, and Remark \ref{rem:singleScale}. I will next review other statements of similar nature in previous literature. Three close relatives are Bourgain's discretised sum-product theorem \cite{Bo2}, Dyatlov and Zahl's estimate \cite{MR3558305} for the additive energy of Ahlfors-regular sets, and Hochman's projection theorem \cite{Ho} for self-similar sets.

Bourgain in \cite{Bo2}, see also his earlier paper \cite{Bo1}, proved lower bounds for the dimension of $A + \theta B$ without assuming that $A,B$ are Ahlfors-regular. He showed that if $\alpha = \beta \in (0,1)$, and $E \subset \R$ has dimension $\Hd E \geq \kappa > 0$, then 
\begin{equation}\label{bourgain} \Hd [A + \theta B] \geq \alpha + \epsilon \end{equation}
for some $\theta \in E$, where $\epsilon > 0$ only depends on $\alpha = \beta$, and $\kappa > 0$. Without the Ahlfors-regularity assumption, the dependence of "$\epsilon$" on "$\kappa$" in \eqref{bourgain} is inevitable, see the next example. As Theorem \ref{mainNonTec} shows, the dependence disappears for Ahlfors-regular sets. 

\begin{ex}\label{ex:arithmetic} Fix $n \in \N$ and $\kappa \in (0,\tfrac{1}{2}]$ such that $n^{2\kappa} \in \N$. Let $A_{n} := \{k/n : 1 \leq k \leq n\}$ and $E_{n} := \{k/n^{2\kappa} : 1 \leq k \leq n^{2\kappa}\}$. Then $A_{n} + E_{n}A_{n} \subset \{k/n^{1 + 2\kappa} : 1 \leq k \leq 2n^{1 + 2\kappa}\}$. Thus, if $\delta \in (0,1)$, and $N_{\delta}(A)$ is the $\delta$-covering number of $A$, we have 
\begin{equation}\label{form146} N_{\delta}(A_{n} + E_{n}A_{n}) \leq 2n^{1 + 2\kappa}, \qquad c \in E_{n}. \end{equation}
For $0 < \delta \leq n^{-1 - 2\kappa}$, the same inequality remains true, up to an absolute constant, if we replace $A_{n},E_{n}$ by the $\delta$-neighbourhoods $A_{n}(\delta)$ and $E_{n}(\delta)$. Now, choose $\delta := n^{-2}$. Then, if $\mathcal{H}^{s}_{\infty}$ stands for $s$-dimensional Hausdorff content, it is easy to check that
\begin{displaymath} \mathcal{H}^{1/2}_{\infty}(A_{n}(n^{-2})) \gtrsim 1 \quad \text{and} \quad \mathcal{H}^{\kappa}_{\infty}(E_{n}(n^{-2})) \gtrsim 1. \end{displaymath}
Therefore, informally speaking, $A_{n}(n^{-2})$ is the $n^{-2}$-neighbourhood of a $\tfrac{1}{2}$-dimensional set, and $E_{n}(n^{-2})$ is the $n^{-2}$-neighbourhood of a $\kappa$-dimensional set. Noting that $n^{1 + 2\kappa} = (n^{-2})^{1/2 + \kappa}$, the inequality \eqref{form146} roughly says that $A_{n}(n^{-2}) + E_{n}(n^{-2})A_{n}(n^{-2})$ is at most $(\tfrac{1}{2} + \kappa)$-dimensional. 

The construction described above can be "iterated" to produce compact, Cantor type, non-Ahlfors-regular sets $A,E_{\kappa} \subset [0,1]$ such that $\Hd A = \tfrac{1}{2}$, and $\Hd E_{\kappa} = \kappa$, and $\Hd (A + E_{\kappa}A) \leq \tfrac{1}{2} + \kappa$. In particular, $\lim_{\kappa \searrow 0} \dim (A + E_{\kappa}A) = \dim A$. This shows that the Ahlfors-regularity assumption in Theorem \ref{mainNonTec} is necessary, even when $A = B$. The "iteration" procedure is explained further in \cite[Example 4.1]{MR3503710} in a slightly different, but very similar, context.
 \end{ex}
 
We have now seen that the size of "$\epsilon$" in \eqref{bourgain} depends on "$\kappa$", and this dependence vanishes for Ahlfors-regular sets (as long as $\kappa > 0$). Theorem \ref{mainNonTec} has another key difference compared to Bourgain's result: Bourgain initially proved \eqref{bourgain} only in the case $\alpha = \Hd A = \Hd B = \beta$, while Theorem \ref{mainNonTec} makes no assumptions on the relative sizes of $\alpha$ and $\beta$. A variant of Bourgain's result \eqref{bourgain} (for general sets) is simply false in the total absence of such assumptions: it is not difficult to construct (non-Ahlfors-regular) sets $A,B \subset [0,1]$ of dimensions $\Hd A = \tfrac{1}{2}$ and $\Hd B = \tfrac{1}{4}$ such that $\Hd [A + \theta B] = \tfrac{1}{2}$ for all $\theta \in E$, where $\Hd E = \tfrac{1}{4}$. The construction is based on arithmetic progressions, just like Example \ref{ex:arithmetic}. If $A_{n} = \{k/\sqrt{n} : 1 \leq k \leq \sqrt{n}\}$ and $B_{n} = E_{n} = \{k/\sqrt[4]{n} : 1 \leq k \leq \sqrt[4]{n}\}$ for some $n \in \N$ such that $\sqrt[4]{n} \in \N$, then 
\begin{itemize}
\item $\mathcal{H}^{1/2}_{\infty}(A_{n}(n^{-1})) \sim 1$,
\item $\mathcal{H}^{1/4}_{\infty}(B_{n}(n^{-1})) = \mathcal{H}^{1/4}_{\infty}(E_{n}(n^{-1})) \sim 1$, and
\item $N_{n^{-1}}(A_{n}(n^{-1}) + E_{n}(n^{-1})B_{n}(n^{-1})) \sim \sqrt{n} = N_{n^{-1}}(A_{n}(n^{-1}))$.
\end{itemize}
The last point is the scale $n^{-1}$ analogue of the equation $\Hd (A + EB) = \Hd A$. Again, an iterative construction as in \cite[Example 4.1]{MR3503710} is required to pass from the approximating sets $A_{n},E_{n},B_{n}$ to $A,E,B$. 

While Bourgain's result \eqref{bourgain} was initially proven only in the case $\alpha = \beta$, the case $\alpha \neq \beta$ has been recently studied in \cite{2021arXiv211002779O}. Based on evidence in finite fields \cite{2018arXiv180109591O}, the following conjecture seems plausible for general sets:
\begin{conjecture} Let $\kappa > 0$, and let $A,B,E \subset [0,1]$ be compact sets with $\Hd B \geq \kappa$, $\Hd B + \Hd E \geq \Hd A + \kappa$, and $\Hd A \leq 1 - \kappa$. Then \eqref{bourgain} holds for some $\theta \in E$, and for some $\epsilon > 0$ depending only on $\kappa$. \end{conjecture}

To understand why the lower bound in Theorem \ref{mainNonTec} might be plausible -- while a similar statement for non-Ahlfors-regular sets fails in multiple ways -- one should next compare Theorem \ref{mainNonTec} to the bound of Dyatlov and Zahl \cite{MR3558305} on the additive energy of Ahlfors-regular sets. In \cite[Theorem 6]{MR3558305}, it was shown that if $A \subset [0,1]$ is an Ahlfors-regular set with $\Hd A = \alpha \in (0,1)$, then there exists a constant $\epsilon > 0$, which depends on both "$\alpha$" and the Ahlfors-regularity constant of $A$, such that
\begin{displaymath} (\mathcal{H}^{\alpha})^{4}(\{(x_{1},x_{2},x_{3},x_{4}) \in A^{4} : |(x_{1} + x_{2}) - (x_{3} + x_{4})| < \delta\}) \lesssim \delta^{\alpha + \epsilon}, \quad 0 < \delta < 1. \end{displaymath}
For more recent improvements and generalisations of the work of Dyatlov and Zahl, see the papers \cite{MR4163999} by Rossi and Shmerkin, and \cite{2020arXiv201202747C} by Cladek and Tao. It follows from the estimate above that $\Hd [A + A] \geq \alpha + \epsilon$. More generally, one can show that $\Hd [A + \theta A] \geq \alpha + \epsilon$ for \textbf{every} $\theta \in [\tfrac{1}{2},1]$, where $\epsilon$ only depends on $\alpha \in (0,1)$, and the regularity constant of $A$. Does this imply that the zero-dimensional exceptional set mentioned in Theorem \ref{mainNonTec} is not really necessary? No, because the constant "$\epsilon$" in the uniform lower bound essentially depends on the regularity constant of "$A$". 

\begin{ex}\label{ex:An} Consider a $\tfrac{1}{2}$-dimensional self-similar set $A_{N}$ on $[0,1]$ with "$N$" generating intervals of common length $N^{-2}$, placed in arithmetic progression. Then, $A_{N}$ is $\tfrac{1}{2}$-Ahlfors-regular for every $N \geq 1$, but the constant increases when $N \to \infty$, and also $\Hd [A_{N} + A_{N}] \to \tfrac{1}{2}$ as $N \to \infty$. Regardless, Theorem \ref{mainNonTec} shows that $\Hd [A_{N} + \theta A_{N}] \geq \tfrac{2}{3}$ for every $\theta \in [0,1] \, \setminus \, E$, where $\Hd E = 0$. \end{ex}

Even though the largest "$\epsilon$" in $\Hd [A + A] \geq \alpha + \epsilon$ depends on the Ahlfors-regularity constant of $A$, it is known that $A + A + \ldots + A$ contains an interval, and in particular $\Hd [A + \ldots + A] = 1$, if the number of copies of $A$ in the sum is sufficiently large, depending on the dimension and Ahlfors-regularity constant of $A$. This is due to Astels \cite{MR1491854}, and a recent higher-dimensional generalisation is due to Feng and Wu \cite{MR4313239} (Ahlfors-regularity is not sufficient in $\R^{d}$, but one rather needs to assume that $A$ has \emph{positive thickness}, which is implied by Ahlfors-regularity for $A \subset \R$.) While these results of Astels or Feng-Wu are not used in the proof of Theorem \ref{mainNonTec}, an idea of "iteration" is also present, see Section \ref{s:outline}, and enables one to upgrade $\epsilon$-improvements to more substantial ones, at the cost of throwing away a zero-dimensional set of exceptional directions $E \subset \R$. 

The examples above indicate that Theorem \ref{mainNonTec} illustrates a phenomenon slightly distinct from those discussed by Bourgain, and Dyatlov-Zahl. In some ways, a closer relative is Hochman's projection theorem \cite{Ho} for self-similar sets: a special case of his result shows that if $A,B \subset \R$ are self-similar sets with common contraction ratios, then 
\begin{displaymath} \Hd [A + \theta B] = \min\{\Hd A + \Hd B,1\} \end{displaymath}
for all $\theta \in \R \, \setminus \, E$, where $\Hd E = 0$. In fact, even the packing dimension of $E$ is zero. The assumption about common contraction ratios is needed to ensure that $A \times B$ is also a self-similar set: this is relevant, because Hochman's projection theorem states, generally, that if $K \subset \R^{2}$ is a self-similar set without rotations, then $\Hd \pi_{\theta}(K) = \min\{\Hd K,1\}$ for all $\theta \in \R \, \setminus \, E$, where $\dim E = 0$, and $\pi_{\theta}(x,y) := x + \theta y$.

I briefly mention that the case with (dense) rotations can also be handled with different methods, see \cite{MR2470633}. More generally, the product of self-similar sets is self-affine, and there are numerous papers containing strong projection theorems for self-affine sets and measures, see \cite{MR3955707,MR3558544,MR2912701,MR2470633}.

The starting point of the current paper was an attempt to generalise Hochman's projection theorem from self-similar sets $K \subset \R^{2}$ to Ahlfors-regular sets $K \subset \R^{2}$. The attempt failed in various ways, and Theorem \ref{mainNonTec} was the best outcome I could recover. The initial goal still seems plausible, however, so I pose it as a question:
\begin{question}\label{mainQ} Let $K \subset \R^{2}$ be a closed Ahlfors-regular set. Is it true that $\Hd \pi_{\theta}(K) = \min\{\Hd K,1\}$ for all $\theta \in \R \, \setminus \, E$, where $\Hd E = 0$? \end{question} 
An affirmative answer to Question \ref{mainQ} would be an improvement over Marstrand's classical projection theorem \cite{Mar} for Ahlfors-regular sets. Partial evidence for Question \ref{mainQ} can be found in \cite{MR4218963}, where an analogous result is proved for Assouad dimension in place of Hausdorff dimension. (The result for Assouad dimension does not require "$K$" to be Ahlfors-regular, which is essentially due to the possibility to shift attention from "$K$" to its more regular tangents.)

\subsection{Proof outline}\label{s:outline} The uniform lower bound $\Hd [A + A] \geq \Hd A + \epsilon$ for Ahlfors-regular sets, discussed above, can be deduced from certain \emph{inverse theorems} in additive combinatorics; there are several of these, for example by Hochman \cite{Ho}, Sanders \cite{MR2994996}, and most recently by Shmerkin \cite{Sh}. Similar ideas are also present in the works of Bourgain \cite{Bo1,Bo2}. To be precise, Dyatlov-Zahl \cite{MR3558305} applied Sanders's theorem \cite{MR2994996}, while Rossi-Shmerkin used Shmerkin's inverse theorem \cite{Sh}.  I deliberately wrote "can" in the opening of the section, since Cladek-Tao \cite{2020arXiv201202747C} avoid the use of inverse theorems. 

The idea of inverse theorems is, briefly, the following: if $\Hd [A + B] \leq \Hd A + \epsilon$ for a sufficiently small constant $\epsilon > 0$, then $A$ needs to have \emph{$1$-dimensional branching} on scales where $B$ has \emph{positive-dimensional branching}. If $\Hd B > 0$, then $B$ necessarily has positive-dimensional branching on many scales, and hence $A$ has $1$-dimensional branching on many scales. However, if $A$ is $\alpha$-Ahlfors-regular with $\alpha < 1$, then $A$ cannot have $1$-dimensional branching on \textbf{any} scales. If follows that $\Hd [A + B] \geq \Hd A + \epsilon$.

As we have seen in Example \ref{ex:An}, there is an inevitable dependence between the Ahlfors-regularity constant of $A$, and the number "$\epsilon$". This dependence is hidden in the argument above: the statement that "$A$ cannot have $1$-dimensional branching on any scales" is only true, if the scales are chosen appropriately: the \emph{leaps} between consecutive scales need to be big enough, depending on the regularity constant of $A$. The \emph{leap size} is one of the parameters in the inverse theorem (of either Hochman or Shmerkin), and the theorems only become applicable if the threshold "$\epsilon$" in the inequality $\Hd [A + B] \leq \Hd A + \epsilon$ is small enough, depending on the leap size. 

To prove Theorem \ref{mainNonTec}, a bootstrapping scheme is required. The following outline is not in $1$-to-$1$ correspondence with what really happens in the proof, but I hope that some ideas are transmitted. The main idea is to assume "inductively" that one has already established a weaker version of Theorem \ref{mainNonTec}: one has already managed to show that 
\begin{equation}\label{form260} \Hd [A + \theta B] \geq \alpha + \chi \end{equation}
for a certain parameter $\chi \geq 0$, for \textbf{all} pairs of Ahlfors-regular sets $A,B$ with fixed dimensions $\alpha,\beta$, and regularity constants $C_{A},C_{B}$, and for all $\theta \in [0,1]$ outside a tiny exceptional set $E \subset [0,1]$. The set $E$ may naturally depend on the choices of $A,B$. In the "base case" $\chi = 0$ (or alternatively $\chi = \epsilon$, depending on $C_{A}$), one simply has $E = \emptyset$. To quantify the "tininess" of $E$ in general, one formulates a $\delta$-discretised version of the theorem, so $E = E_{\delta}$. Then, one shows that $\mathcal{H}^{\tau}_{\infty}(E) \leq \delta^{\eta}$, where $\tau > 0$ can be chosen arbitrarily small (one eventually wants to let $\tau \to 0$ to establish the $0$-dimensionality of exceptions), and $\eta > 0$ is an auxiliary parameter. For the details of how to set this up, see Theorem \ref{main}.

Next, one makes a counter assumption: there exists a pair of Ahlfors-regular sets $A,B$, with exactly the same parameters $\alpha,\beta,C_{A},C_{B}$ as above, such that 
\begin{equation}\label{form258} \Hd [A + \theta B] \leq \alpha + \chi + \epsilon \quad \text{with} \quad \chi + \epsilon < \beta \cdot \tfrac{1 - \alpha}{2 - \alpha} \end{equation}
for all $\theta \in E \subset [0,1]$, where, roughly speaking, $\mathcal{H}^{\tau}_{\infty}(E) \geq \delta^{\eta/2}$. This means that $E$ is allowed to be fairly tiny, but is assumed to be substantially larger than the "tininess" of $E$ in the inductive hypothesis. Hence, combining \eqref{form260} and \eqref{form258}, one finds that
\begin{equation}\label{form259} \alpha + \chi \leq \Hd [A + \theta B] \leq \alpha + \chi + \epsilon \end{equation}
for a typical direction $\theta \in E$. Now, fix $\theta_{0} \in E$ such that \eqref{form259} holds, and view $A \times B$ as a subset of $[A + \theta_{0}B] \times \R =: A' \times \R$ (after a change of coordinates). Then $A'$ has the following two properties:
\begin{enumerate}
\item $\alpha + \chi \leq \Hd A' \leq \alpha + \chi + \epsilon$,
\item $A'$ is a union of translates of $A$. This implies that $A'$ has $\geq \alpha$-dimensional branching on all scales. For a clarification of what this means, see the statement of Theorem \ref{shmerkin}, and the remark after it. 
\end{enumerate}
Recalling that $A \times B \subset A' \times \R$, up to a change of coordinates, one more precisely rewrites $A \times B = A' \times B'$, where $B'$ represents the intersection of $A \times B$ with the "typical fibre" under the map $\pi_{\theta_{0}}(x,y) = x + \theta_{0}y$. From the upper bound in \eqref{form259}, one infers that $\Hd B' \geq \Hd (A \times B) - (\alpha + \chi + \epsilon) = \beta - \chi - \epsilon$. Recall from \eqref{form258} that $\chi + \epsilon < \beta(1 - \alpha)/(2 - \alpha)$, so in fact $\Hd B' \geq \beta/(2 - \alpha)$.

Next, fix another direction $\theta_{1} \in E$ such that $|\theta_{0} - \theta_{1}| \sim 1$. Then, writing $\theta := \theta_{1} - \theta_{0}$, one roughly observes that $A + \theta_{1}B = A' + \theta B'$, so in particular 
\begin{equation}\label{form152} \Hd (A' + \theta B') = \Hd (A + \theta_{1}B) \stackrel{\eqref{form259}}{\leq} \alpha + \chi + \epsilon \stackrel{(1)}{\leq} \Hd A' + \epsilon. \end{equation}
This places us in a position to use Shmerkin's inverse theorem, as described in the second paragraph of this section. See Theorem \ref{shmerkin} for the precise statement of Shmerkin's result. The leap size (hence the choice of $\epsilon$) will only depend on $C_{A}$. The conclusion is that if $\epsilon > 0$ is small enough, then $A'$ has $1$-dimensional branching on all the scales where $B'$ has positive-dimensional branching. The fraction of scales such that $B'$ has positive-dimensional branching is at least $\Hd B' \geq \beta - \chi - \epsilon$. Consequently $A'$ has $1$-dimensional branching on at least that many scales. 

On the other hand, property (2) of the set $A'$ says that $A'$ has $\geq \alpha$-dimensional branching on \textbf{all} scales (this is where the dependence between the leap size and the Ahlfors-regularity of $A$ comes in). Ignoring "$\epsilon$", this allows us to compute the following lower bound on the dimension of $A'$:
\begin{displaymath} \alpha + \chi \stackrel{(1)}{\geq} \Hd A' \geq 1 \cdot \Hd B' + \alpha \cdot (1 - \Hd B') \geq (\beta - \chi) + \alpha \cdot (1 - \beta + \chi). \end{displaymath}
Rearranging, we find that $\chi \geq \beta(1 - \alpha)/(2 - \alpha)$, which contradicts \eqref{form258}. Hence the counter assumption in \eqref{form258} must be false, and in fact $\Hd [A + \theta B] \geq \alpha + \chi + \epsilon$ for some $\theta \in E$. Repeating this argument $\sim \epsilon^{-1}$ times will eventually conclude the proof of Theorem \ref{mainNonTec}.

\subsection{Notation}\label{s:notation} Open balls in either $\R^{d}$ will be denoted $B(x,r)$; in this paper always $d \in \{1,2\}$. If $f,g$ are real-valued non-negative functions of some parameter $x \in X$, the notation $f \lesssim g$ means that there exists an absolute constant $C \geq 1$ such that $f(x) \leq Cg(x)$ for all $x \in X$. If $A \subset \R^{d}$ is a bounded set, and $\delta > 0$, then $N_{\delta}(A)$ refers to the smallest number of balls $B(x,\delta) \subset \R^{d}$ required to cover $A$. If $A$ is a finite set, its cardinality is denoted $|A|$. For $\delta > 0$, the open Euclidean $\delta$-neighbourhood of a set $A \subset \R^{d}$ is denoted $A(\delta)$.  For $\delta > 0$, a \emph{tube of width $\delta > 0$} is a set of the form $\pi_{\theta}^{-1}(I)$, where $\pi_{\theta}(x,y) = x + \theta y$, $\theta \in [0,1]$, and $I \subset \R$ is an interval of length $\delta$. Note that if $T = \pi_{\theta}^{-1}[x - \delta/2,x + \delta/2] \subset \R^{2}$ is a tube of width $\delta$, where $\theta \in [0,1]$, then $\ell(\delta/4) \subset T \subset \ell(\delta/2)$ with $\ell := \pi_{\theta}^{-1}\{x\}$.

\subsection{Acknowledgements} I would like to thank the referee for a very careful reading of the manuscript, and for making a large number of helpful suggestions. 

\section{Some definitions} We start by recalling the notion of entropy. If $\mu$ is a probability measure on some space $\Omega$, and $\mathcal{F}$ is a $\mu$ measurable partition of $\Omega$, the $\mathcal{F}$-entropy of $\mu$ is defined by
\begin{displaymath} H(\mu,\mathcal{F}) := \sum_{F \in \mathcal{F}} \mu(F)\log \tfrac{1}{\mu(F)}. \end{displaymath}
Here $0 \cdot \log \tfrac{1}{0} := 0$, and "$\log$" refers to logarithm in base $2$. Typically, the measures of interest are probability measures $\R$, and $\mathcal{F}$ is the partition into dyadic intervals of length $\delta > 0$; this partition is denoted by $\mathcal{D}_{\delta} := \{[k\delta,(k + 1)\delta) : k \in \Z\}$. In addition to $\mathcal{F}$-entropy, we will also need the following \emph{conditional entropy of $\mu$}:
\begin{displaymath} H(\mu,\mathcal{F} \mid \mathcal{E}) := \sum_{E \in \mathcal{E}} \mu(E)H(\mu_{E},\mathcal{F}). \end{displaymath}
Here $\mu_{E} := \mu(E)^{-1} \cdot \mu|_{E}$, and $\mathcal{E},\mathcal{F}$ are $\mu$ measurable partitions of $\Omega$. In the applications below, every $E \in \mathcal{E}$ is a finite union of certain sets in $\mathcal{F}$. In this special case, the conditional entropy can be rewritten as
\begin{equation}\label{entropy} H(\mu,\mathcal{F} \mid \mathcal{E}) = H(\mu,\mathcal{F}) - H(\mu,\mathcal{E}), \end{equation}
as an easy calculation shows (or see \cite[Proposition 3.3]{MR3590535}). The deepest fact we will need to know about entropy is the estimate $H(\mu,\mathcal{F}) \leq \log |\mathcal{F}|$, which is an immediate consequence of Jensen's inequality. 

We then proceed to other definitions.
 
\begin{definition}[Projections] For $\theta \in \R$, we define the maps $\pi_{\theta} \colon \R^{2} \to \R$ by 
\begin{displaymath} \pi_{\theta}(x,y) := x + \theta y, \qquad (x,y) \in \R^{2}. \end{displaymath} 
\end{definition}

\begin{definition}\label{def:mult} Let $K \subset \R^{2}$, let $0 < r \leq R \leq \infty$, and let $x \in K$. For $\theta \in [0,1]$, we define the following \emph{multiplicity number}:
\begin{displaymath} \m_{K,\theta}(x \mid [r,R]) := N_{r}(B(x,R) \cap K_{r} \cap \pi_{\theta}^{-1}\{\pi_{\theta}(x)\}).  \end{displaymath} 
Here $K_{r}$ refers to the $r$-neighbourhood of $K$. Thus, $\m_{K,\theta}(x \mid [r,R])$ keeps track of the (smallest) number of $r$-balls needed to cover the intersection between $B(x,R) \cap K_{r}$ and the line $\pi_{\theta}^{-1}\{\pi_{\theta}(x)\}$. Often the set "$K$" is clear from the context, and we abbreviate $\m_{K,\theta} =: \m_{\theta}$. We also allow for the case $R = \infty$: then $B(x,R) := \R^{2}$. \end{definition}

\begin{definition}[High multiplicity sets]\label{def:highMult} Let $0 < r \leq R \leq \infty$, $M > 0$, and let $\theta \in [0,1]$. For $K \subset \R^{2}$, we define the \emph{high multiplicity set}
\begin{displaymath} H_{\theta}(K,M, [r,R]) := \{x \in K : \m_{K,\theta}(x \mid [r,R]) \geq M\}. \end{displaymath}  \end{definition}

Note that the same latter "$H$" will stand for both high-multiplicity sets, and entropy; the correct interpretation should always be clear from context. We will next verify some elementary but useful facts about the multiplicity numbers and high multiplicity sets. The first observation is that 
\begin{equation}\label{form62} M' \geq M > 0 \quad \Longrightarrow \quad H_{\theta}(K,M',[r,R]) \subset H_{\theta}(K,M,[r,R]) \end{equation}
for all $\theta,K,r \leq R$, since if $\m_{K,\theta}(x \mid [r,R]) \geq M'$, then also $\m_{K,\theta}(x \mid [r,R]) \geq M$. The next lemma answers the questions: what happens to the multiplicity numbers and high multiplicity sets if we change the radii $r$ and $R$?
\begin{lemma}\label{lemma6} Let $K \subset \R^{2}$, $x \in K$, $\theta \in S^{1}$, $C \geq 1$, and assume that $Cr \leq R$. Then,
\begin{equation}\label{form32} \m_{K,\theta}(x \mid [r,R]) \leq C \cdot \m_{K,\theta}(x \mid [Cr,R]) \quad \text{and} \quad \m_{K,\theta}(x \mid [r,R]) \leq \m_{K,\theta}(x \mid [r,CR]). \end{equation} 
In particular,
\begin{equation}\label{form55} H_{\theta}(M,[r,R]) \subset H_{\theta}(\tfrac{M}{C},[Cr,R]) \quad \text{and} \quad H_{\theta}(M,[r,R]) \subset H_{\theta}(M,[r,CR]) \end{equation}
for all $M > 0$. Here we abbreviate $H_{\theta}(\ldots) := H_{\theta}(K,\ldots)$.
\end{lemma}

\begin{proof} Clearly $K_{r} \subset K_{Cr}$, so 
\begin{displaymath} \m_{K,\theta}(x \mid [r,R]) \leq N_{r}(B(x,R) \cap K_{Cr} \cap \pi_{\theta}^{-1}\{\pi_{\theta}(x)\}). \end{displaymath}
The right hand side above only differs from $\m_{K,\theta}(x \mid [Cr,R])$ in the appearance of "$N_{r}$" instead of "$N_{Cr}$". But evidently
\begin{displaymath} N_{r}(A \cap \pi_{\theta}^{-1}\{t\}) \leq C \cdot N_{Cr}(A \cap \pi_{\theta}^{-1}\{t\}), \qquad A \subset \R^{2}, \end{displaymath} 
and we have proven the first inequality in \eqref{form32}. The second inequality is a restatement of $N_{r}(B(x,R) \cap K_{r} \cap \pi_{\theta}^{-1}\{\pi_{\theta}(x)\}) \leq N_{r}(B(x,CR) \cap K_{r} \cap \pi_{\theta}^{-1}\{\pi_{\theta}(x)\})$. The inclusions in \eqref{form55} are direct consequences of the inequalities in \eqref{form32}.  \end{proof}

\begin{warning} The first inequality in \eqref{form32} is not reversible. Writing out the definition,
\begin{displaymath} \m_{K,\theta}(x \mid [Cr,R]) = N_{Cr}(B(x,R) \cap K_{Cr} \cap \pi_{\theta}^{-1}\{\pi_{\theta}(x)\}) \end{displaymath}
counts the intersections between $\pi_{\theta}^{-1}\{\pi_{\theta}(x)\}$ and $K_{Cr}$, and there might be many (many!) more of those than intersections between $\pi_{\theta}^{-1}\{\pi_{\theta}(x)\}$ and $K_{r}$. In other words, it is possible that $\m_{K,\theta}(x \mid [Cr,R]) \gg \m_{K,\theta}(x \mid [r,R])$ for individual points $x \in K$. \end{warning}

\begin{definition}[Rescaling map] Let $z_{0} \in \R$ or $z_{0} \in \R^{2}$, and $r > 0$. By definition, the \emph{rescaling map} with parameters $z_{0},r_{0}$ is the map $T_{z_{0},r_{0}}(z) := (z - z_{0})/r_{0}$, which sends the ball $B(z_{0},r_{0})$ to the unit ball $B(1) = B(0,1)$. 
\end{definition}

We next record how the high multiplicity sets interact with rescaling maps.

\begin{lemma}\label{lemma7} Let $K \subset \R^{2}$ be arbitrary, let $0 < r \leq R \leq \infty$, $M > 0$, and $\theta \in [0,1]$. Then, 
\begin{displaymath} T_{z_{0},r_{0}}(H_{\theta}(K,M,[r,R])) = H_{\theta}(T_{z_{0},r_{0}}(K),M,[\tfrac{r}{r_{0}},\tfrac{R}{r_{0}}]), \qquad z_{0} \in \R^{2}, \, r_{0} > 0. \end{displaymath}
\end{lemma}
\begin{proof} Let $y = T_{z_{0},r_{0}}(z) \in T_{z_{0},r_{0}}(H_{\theta}(K,M,[r,R]))$. Thus $z \in H_{\theta}(K,M,[r,R])$, that is,
\begin{equation}\label{form251} N_{r}(B(z,R) \cap K_{r} \cap \pi_{\theta}^{-1}\{\pi_{\theta}(z)\}) \geq M. \end{equation}
We are supposed to prove that $y \in H_{\theta}(T_{z_{0},r_{0}}(K),M,[\tfrac{r}{r_{0}},\tfrac{R}{r_{0}}])$, or in other words
\begin{equation}\label{form250} N_{\tfrac{r}{r_{0}}}(B(y,\tfrac{R}{r_{0}}) \cap (T_{z_{0},r_{0}}K)_{\tfrac{r}{r_{0}}} \cap \pi_{\theta}^{-1}\{\pi_{\theta}(y)\}) \geq M. \end{equation}
To see this, one first needs to spend a moment to check that
\begin{displaymath} B(y,\tfrac{R}{r_{0}}) \cap (T_{z_{0},r_{0}}K)_{\tfrac{r}{r_{0}}} \cap \pi_{\theta}^{-1}\{\pi_{\theta}(y)\} = T_{z_{0},r_{0}}(B(z,R) \cap K_{r} \cap \pi_{\theta}^{-1}\{\pi_{\theta}(z)\}). \end{displaymath}
I omit the details. After this, one can apply the equation
\begin{displaymath} N_{\tfrac{r}{r_{0}}}(T_{z_{0},r_{0}}(A)) = N_{r}(A), \qquad A \subset \R^{2}, \end{displaymath}
to the set $A = B(z,R) \cap K_{r} \cap \pi_{\theta}^{-1}\{\pi_{\theta}(z)\}$ to infer \eqref{form250} from \eqref{form251}. This proves the inclusion "$\subset$" of the lemma, and the other inclusion follows in a similar fashion. \end{proof}

\section{Technical version of the main theorem}

In this section, we reduce the proof of the main result, Theorem \ref{mainNonTec}, to a more technical statement. We begin with a few additional definitions. 

\begin{definition}[$(\tau,C)$-Frostman measure] Let $\tau > 0$ and $C \geq 1$. A Borel measure $\nu$ on $\R$ is called a $(\tau,C)$-Frostman measure if $\nu(B(x,r)) \leq C r^{\tau}$ for all $x \in \R$ and $r > 0$.
\end{definition}

 \begin{definition}[$(\alpha,C)$-regular measure] Let $\alpha > 0$ and $C \geq 1$. A Borel measure $\mu$ with $K := \spt \mu \subset \R$ is called $(\alpha,C)$-regular if 
 \begin{enumerate}
 \item $\mu$ is an $(\alpha,C)$-Frostman measure, and
 \item $N_{r}(K \cap B(x,R)) \leq C (R/r)^{\alpha}$ for all $x \in \R$ and $0 < r \leq R < \infty$.
 \end{enumerate}
\end{definition}
\begin{definition}[Ahlfors-regular set]\label{def:ahlforsRegular} Let $0 \leq \alpha \leq 1$. A closed set $A \subset \R$ is called \emph{$\alpha$-regular} if $\calH^{\alpha}|_{A}$ is $(\alpha,C)$-regular for some $C \geq 1$. \end{definition}

\begin{remark} A more common definition of Ahlfors-regularity is the following: for $\alpha \geq 0$, an $\mathcal{H}^{\alpha}$ measurable set $A \subset \R$ is $\alpha$-regular if there exists a constant $C > 0$ such that 
\begin{equation}\label{form147} C^{-1}r^{\alpha} \leq \mathcal{H}^{\alpha}(A \cap B(x,r)) \leq Cr^{\alpha}, \qquad x \in A,\, 0 < r < \diam(A). \end{equation}
A set $A \subset \R$ satisfying \eqref{form147} is also $\alpha$-regular according to Definition \ref{def:ahlforsRegular}. Indeed, first note that $\mathcal{H}^{s}|_{A}$ is clearly $(\alpha,C)$-Frostman. Second, if $x \in \R$, $0 < r \leq R < \infty$, and $N := N_{r}(A \cap B(x,R))$, then there exists an $(r/2)$-separated set $\{x_{1},\ldots,x_{N}\} \subset A \cap B(x,R)$ of cardinality $N$. Consequently, if also $4R < \diam(A)$, we may use \eqref{form147} to infer that
\begin{displaymath} C^{-1}N(r/2)^{\alpha} \leq \sum_{j = 1}^{N} \mathcal{H}^{s}(A \cap B(x_{j},\tfrac{r_{j}}{2})) \leq \mathcal{H}^{s}(A \cap B(x_{1},4R)) \leq 4^{\alpha}CR^{\alpha}. \end{displaymath}
This yields $N \leq 8^{\alpha}C^{2}(R/r)^{\alpha}$. If $4R \geq \diam(A)$, then we should first find a point $\bar{x} \in \R$ such that $N \sim N_{r}(A \cap B(\bar{x},\diam(A)/4))$. Then we may use the previous argument to deduce that $N \lesssim 8^{\alpha}C^{2}(\diam(A)/r)^{\alpha} \leq 32^{\alpha}C^{2}(R/r)^{\alpha}$. Hence, we conclude that $A$ is $\alpha$-regular according to Definition \ref{def:ahlforsRegular}. The converse implication is also true, but this this is not so relevant, so we leave the details to the reader. Our definition of Ahlfors-regularity has the useful property that the constants remain precisely unchanged under rescaling. \end{remark}

With these definitions in hand, we can state the main technical result of the paper:

\begin{thm}\label{main} Let $\alpha,\beta,\tau \in (0,1]$,  $C_{\alpha},C_{\beta} > 0$, and let $\sigma > \beta/(2 - \alpha)$.
\begin{itemize}
\item[(A1) \phantomsection \label{A1}] Let $\mu_{\alpha}$ be an $(\alpha,C_{\alpha})$-regular measure with $K_{\alpha} := \spt \mu_{\alpha} \subset \R$.
\item[(A2) \phantomsection \label{A2}] Let $\mu_{\beta}$ be a $(\beta,C_{\beta})$-regular measure with $K_{\beta} := \spt \mu_{\beta} \subset \R$.
\end{itemize}
Write $\mu := \mu_{\alpha} \times \mu_{\beta}$ and $K := \spt \mu = K_{\alpha} \times K_{\beta} \subset \R^{2}$. Then,
\begin{equation}\label{mainEq} \mathcal{H}^{\tau}_{\infty}(\{\theta \in [0,1] : \mu(B(1) \cap H_{\theta}(K,\delta^{-\sigma},[\delta,1])) \geq \delta^{\eta}\}) \leq \delta^{\eta} \end{equation}
for all $0 < \delta \leq \delta_{0}(\alpha,\beta,\sigma,\tau,C_{\alpha},C_{\beta})$ and $0 < \eta \leq \eta_{0}(\alpha,\beta,\sigma,\tau,C_{\alpha},C_{\beta})$.
\end{thm}

\begin{remark}\label{rem:singleScale} Theorem \ref{main} easily implies a "single-scale" version of Theorem \ref{mainNonTec}. It reads as follows, using the same notation as in Theorem \ref{main}, and assuming additionally that $K \subset B(\tfrac{1}{2})$. Let $0 < \delta \leq \delta_{0}$ and $0 < \eta \leq \eta_{0}$, and let $E \subset [0,1]$ be an arbitrary set with $\mathcal{H}^{\tau}_{\infty}(E) > \delta^{\eta}$. \emph{Then, there exists $\theta \in E$ such that}
\begin{equation}\label{form144} N_{\delta}(\pi_{\theta}(K')) \geq \delta^{\sigma + \eta - \alpha - \beta} \end{equation} 
\emph{for all subsets $K' \subset K$ with $\mu(K') \geq \delta^{\eta}$}. I will sketch the argument for deducing \eqref{form144} from Theorem \ref{main}: the full details are very similar to the deduction of Theorem \ref{mainNonTec}, and these details follow up next.

Pick $\theta \in E$ such that 
\begin{equation}\label{form145} \mu(B(1) \cap H_{\theta}(K,(C_{1}\delta)^{-\sigma},[C_{2}\delta,1])) < \tfrac{1}{2}\delta^{\eta}, \end{equation}
where $C_{1},C_{2} \sim_{C_{\alpha},C_{\beta}} 1$ are suitable constants to be determined soon. This is possible by \eqref{mainEq} (adjusting $\delta,\eta$ slightly to take $C_{1},C_{2}$ into account), and the assumption $\mathcal{H}^{\tau}_{\infty}(E) > \delta^{\eta}$. Now, assume that \eqref{form144} fails for this particular $\theta$, and some set $K' \subset K$ with $\mu(K') \geq \delta^{\eta}$. By the regularity of $\mu$, we have $N_{\delta}(K') \gtrsim_{C_{\alpha},C_{\beta}} \delta^{\eta - \alpha - \beta}$. Since \eqref{form144} fails, the set $K'$ can be covered by a collection $\mathcal{T}$ of $|\mathcal{T}| \leq \delta^{\sigma + \eta - \alpha - \beta}$ tubes of the form $T = \pi_{\theta}^{-1}(I)$ for some interval $I \subset \R$ of length $|I| = \delta$. Informally speaking, the "typical" tube in $\mathcal{T}$ now intersects $\gtrsim N_{\delta}(K')/|\mathcal{T}| \gtrsim_{C_{\alpha},C_{\beta}} \delta^{- \sigma}$ discs in the minimal $\delta$-cover of $K'$. To be more precise, the sub-family $\mathcal{T}_{\mathrm{light}}$ of tubes in $\mathcal{T}$ failing this property (say $N_{\delta}(K' \cap T) \leq c \cdot \delta^{-\sigma}$ for $T \in \mathcal{T}_{\mathrm{light}}$), can cover at most a fraction of $\tfrac{1}{2}\mu(K')$ of the $\mu$ measure of $K'$. The reason is that $\mu(K' \cap T) \lesssim c \cdot \delta^{-\sigma + \alpha + \beta}$ for all $T \in \mathcal{T}_{\mathrm{light}}$, and hence
\begin{displaymath} \sum_{T \in \mathcal{T}_{\mathrm{light}}} \mu(K' \cap T) \lesssim c \cdot \delta^{-\sigma + \alpha + \beta} \cdot |\mathcal{T}| \leq c \cdot \delta^{\eta} \leq c \cdot \mu(K'). \end{displaymath}
If $c \sim_{C_{\alpha},C_{\beta}} 1$ is chosen appropriately, we have now shown that every tube in $\mathcal{T} \, \setminus \, \mathcal{T}_{\mathrm{light}}$ intersects $\gtrsim_{C_{\alpha},C_{\beta}} \delta^{-\sigma}$ discs in the minimal $\delta$-cover of $K'$ (hence also $K$), and these tubes cover, altogether, a subset $K_{0}' \subset K'$ of measure $\mu(K_{0}') \geq \tfrac{1}{2}\mu(K')$. If $C_{1},C_{2} \sim_{C_{\alpha},C_{\beta}} 1$ are chosen appropriately, it follows from the definition of high multiplicity sets that 
\begin{displaymath} K_{0}' \subset H_{\theta}(K,(C_{1}\delta)^{-\sigma},[C_{2}\delta,1]) \quad \Longrightarrow \quad \mu(B(1) \cap H_{\theta}(K,(C_{1}\delta)^{-\sigma},[C_{2}\delta,1])) \geq \mu(K_{0}'). \end{displaymath}
Since $\mu(K_{0}') \geq \tfrac{1}{2}\delta^{\eta}$, this contradicts \eqref{form145} and completes the proof of \eqref{form144} for $K'$. \end{remark}

We now prove Theorem \ref{mainNonTec}, assuming Theorem \ref{main}.

\begin{proof}[Proof of Theorem \ref{mainNonTec}, assuming Theorem \ref{main}] Let $A,B \subset \R$ be closed non-empty Ahlfors-regular sets, with $\alpha := \Hd A \in (0,1)$ and $\beta := \Hd B \in (0,1]$. Let $\mu_{\alpha} := \mathcal{H}^{\alpha}|_{A}$ and $\mu_{\beta} := \mathcal{H}^{\beta}|_{B}$. Then $\mu_{\alpha}$ is $(C_{\alpha},\alpha)$-regular for some $C_{\alpha} \geq 1$, and $\mu_{\beta}$ is $(C_{\beta},\beta)$-regular for some $C_{\beta} \geq 1$. Let us first reduce matters to proving
\begin{equation}\label{form148} \Hd \{\theta \in [0,1] : \Hd [A + \theta B] < \alpha + \beta \cdot \tfrac{1 -\alpha}{2 - \alpha}\} = 0. \end{equation}
Indeed, if
\begin{displaymath} \Hd \{\theta \in \R : \Hd [A + \theta B] < \alpha + \beta \cdot \tfrac{1 -\alpha}{2 - \alpha}\} > 0, \end{displaymath} 
then the same holds with "$\R$" replaced by either "$[-r,0]$" or "$[0,r]$" for some $r > 0$. In the former case we define $B_{r} := -rB$, and in the latter case we define $B_{r} := rB$. In both cases $B_{r}$ is $\beta$-regular, although with different constants. Moreover, $B_{r}$ fails \eqref{form148}, and this completes our reduction.

Assume then that \eqref{form148} fails for $A,B$. In other words, there exist $\tau,\epsilon > 0$, a subset $E \subset [0,1]$ with $\mathcal{H}^{\tau}_{\infty}(E) > 0$, and a constant 
\begin{equation}\label{eq:zeta} 0 < \zeta < \alpha + \beta \cdot \tfrac{1 - \alpha}{2 - \alpha} - \epsilon \end{equation}
such that
\begin{displaymath} \Hd [A + \theta B] < \zeta, \qquad \theta \in E. \end{displaymath}
We may assume with no loss of generality that $A,B \subset [0,\tfrac{1}{4}]$. Fix $0 < \eta < \epsilon/2$. We claim that there exist arbitrarily small scales $\delta > 0$ such that the following holds: there exists a set $E_{\delta} \subset E$ with $\mathcal{H}^{\tau}_{\infty}(E_{\delta}) \geq \delta^{\eta}$, and for every $\theta \in E_{\delta}$ a collection of dyadic intervals $\mathcal{I}_{\theta} \subset \mathcal{D}_{\delta}$ with cardinality $|\mathcal{I}_{\theta}| \leq \delta^{-\zeta}$, and the following property:
\begin{equation}\label{form252} (\mu_{\alpha} \times \mu_{\beta})(\{(x,y) \in A \times B : \pi_{\theta}(x,y) \in \cup \mathcal{I}_{\theta}\}) \geq 2 \cdot \delta^{\eta}, \qquad \theta \in E_{\delta}. \end{equation} 
The proof is entirely standard, but let us give the details nevertheless. For all $\theta \in E$, we assumed that $\Hd \pi_{\theta}(A \times B) < \zeta$. Consequently, for any $\delta_{0} > 0$, there exists a countable collection of dyadic intervals $\mathcal{J}_{\theta}$ of lengths $\leq \delta_{0}$ such that $\pi_{\theta}(A \times B) \subset \cup \mathcal{J}_{\theta}$, and
\begin{equation}\label{form262} \sum_{J \in \mathcal{J}_{\theta}}  |J|^{\zeta} \leq 1. \end{equation}
Since the push-forward measure $\mu_{\theta} := \pi_{\theta}(\mu_{\alpha} \times \mu_{\beta})$ is supported on $\pi_{\theta}(A \times B) \subset \cup \mathcal{J}_{\theta}$, we have $\mu_{\theta}(\cup \mathcal{J}_{\theta}) \gtrsim (C_{\alpha}C_{\beta})^{-1}$. Consequently, by the pigeonhole principle, there exists a dyadic scale $\delta = 2^{-j} \cdot \delta_{0}$, with $j = j(\theta) \geq 0$, such that
\begin{displaymath} \mu_{\theta}(\cup \mathcal{J}_{j,\theta}) \gtrsim (C_{\alpha}C_{\beta})^{-1} \cdot (1 + j)^{-2}, \end{displaymath}
where $\mathcal{J}_{j,\theta} = \{J \in \mathcal{J}_{\theta} : |J| = 2^{-j} \cdot \delta_{0}\}$. If the largest scale "$\delta_{0} > 0$" was chosen small enough, depending only on $C_{\alpha},C_{\beta},\eta > 0$, then $(C_{\alpha}C_{\beta})^{-1}(1 + j)^{-2} \gg \delta^{\eta}$, and consequently $\mu_{\theta}(\cup \mathcal{J}_{j,\theta}) \geq 2 \cdot \delta^{\eta}$. Evidently $|\mathcal{J}_{j,\theta}| \leq \delta^{-\zeta}$ by \eqref{form262}. Now we can set $\mathcal{I}_{\theta} := \mathcal{J}_{j,\theta}$, and the requirements $|\mathcal{I}_{\theta}| \leq \delta^{-\zeta}$ and \eqref{form252} are satisfied. The only problem is that the choice of $j(\theta)$, and hence $\delta = 2^{j} \cdot \delta_{0}$ depends on the choice of $\theta \in E$. To fix this, and find the subset $E_{\delta} \subset E$ with $\mathcal{H}^{\tau}_{\infty}(E_{\delta}) \geq \delta^{\eta}$, another application of the pigeonhole principle is needed: one considers the sets $E_{j} := \{\theta \in E : j(\theta) = j\}$, for $j \geq 0$, and uses the sub-additivity of $\mathcal{H}^{\tau}_{\infty}$ to pick $j_{0} \geq 0$ with the property $\mathcal{H}^{\tau}_{\infty}(E_{j_{0}}) \gtrsim \mathcal{H}^{\tau}_{\infty}(E)/(1 + j_{0})^{2}$. If $\delta_{0} > 0$ was initially chosen small enough, now depending on $\mathcal{H}^{\tau}_{\infty}(E) > 0$ and $\eta$, the right hand side is substantially larger than $\delta^{\eta}$. Finally, we set $\delta := 2^{-j_{0}} \cdot \delta_{0}$ and $E_{\delta} := E_{j_{0}}$. Then $\mathcal{H}^{\tau}_{\infty}(E_{\delta}) \geq \delta^{\eta}$, and \eqref{form252} is satisfied for all $\theta \in E_{\delta}$.

Now, fix $\delta > 0$, as above, and let $\mathcal{T}_{\theta} := \{\pi_{\theta}^{-1}\{I\} : I \in \mathcal{I}_{\theta}\}$, $\theta \in E_{\delta}$. These are collections of tubes of width $\delta$ (recall how we defined this in Section \ref{s:notation}), and the sets $H_{\theta} := \{(x,y) \in A \times B : \pi_{\theta}(x,y) \in \cup \mathcal{I}_{\theta}\}$ appearing in \eqref{form252} are covered, each, by the union $\cup \mathcal{T}_{\theta}$. Writing $\mu := \mu_{\alpha} \times \mu_{\beta}$, a tube $T \in \mathcal{T}_{\theta}$ is called \emph{light} if 
\begin{displaymath} \mu(T) \leq \delta^{\zeta + \eta}. \end{displaymath}
Then the union of light tubes has $\mu$-measure no larger than $\delta^{\zeta + \eta} \cdot |\mathcal{T}_{\theta}| \leq \delta^{\eta}$. Consequently, by \eqref{form252}, at least half of the $\mu$ measure of $H_{\theta}$ is covered by non-light tubes. Fix a non-light tube $T \in \mathcal{T}_{\theta}$, and write $\gamma := \alpha + \beta$, so that $\mu$ is $(\gamma,C_{\gamma})$-regular for some $C_{\gamma} \sim C_{\alpha}C_{\beta}$. Then, with $K := A \times B \subset [0,\tfrac{1}{4}]^{2}$, we may infer from the non-lightness of $T$ that
\begin{displaymath} N_{\delta}(K \cap T) \gtrsim C_{\gamma}^{-1} \cdot \delta^{-\gamma + \zeta + \eta}. \end{displaymath} 
Let $\sigma < \gamma - \zeta - \eta$ (the right hand side here is positive, see the calculation in \eqref{form136}), and take $\delta > 0$, above, so small that $N_{\delta}(K \cap T) \geq 100 \cdot (3\delta)^{-\sigma}$ for every non-light tube $T \in \mathcal{T}_{\theta}$. Then, we claim that
\begin{equation}\label{form254} K \cap T \subset H_{\theta}(K,(3\delta)^{-\sigma},[3\delta,1]), \qquad T \in \mathcal{T}_{\theta} \text{ non-light}. \end{equation}
To see this, fix a non-light tube $T \in \mathcal{T}_{\theta}$, and $x \in K \cap T$. Then, since $K = A \times B \subset [0,\tfrac{1}{4}]^{2}$, we have $K \subset B(x,\tfrac{3}{4})$, and therefore
\begin{displaymath} \m_{K,\theta}(x \mid [3\delta,1]) = N_{3\delta}(B(x,1) \cap K_{3\delta} \cap \pi_{\theta}^{-1}\{\pi_{\theta}(x)\}) \geq \tfrac{1}{100} \cdot N_{\delta}(K \cap T) \geq (3\delta)^{-\sigma}. \end{displaymath}
This proves \eqref{form254}. Since the total $\mu$ measure of non-light tubes exceeds $\delta^{\eta}$, by \eqref{form252}, we conclude that
\begin{displaymath} \mu(B(1) \cap H_{\theta}(K,(3\delta)^{-\sigma},[3\delta,1])) \geq \delta^{\eta}, \qquad \theta \in E_{\delta}, \end{displaymath}
and consequently
\begin{equation}\label{form253} \mathcal{H}^{\tau}_{\infty}(\{\theta \in [0,1] : \mu(B(1) \cap H_{\theta}(K,(3\delta)^{-\sigma},[3\delta,1])) \geq \delta^{\eta}\}) \geq \mathcal{H}^{\tau}_{\infty}(E_{\delta}) \geq \delta^{\eta}. \end{equation}
Here $\sigma < \gamma - \zeta - \eta$ can be chosen arbitrarily close to the value
\begin{equation}\label{form136} \gamma - \zeta - \eta \stackrel{\eqref{eq:zeta}}{>} (\alpha + \beta) - (\alpha + \beta \cdot \tfrac{1 - \alpha}{2 - \alpha} - \epsilon) - \eta = \tfrac{\beta}{2 - \alpha} + \tfrac{\epsilon}{2}. \end{equation}
In particular, we may take $\sigma > \beta/(2 -\alpha)$. But once we do this, \eqref{form253} contradicts the statement of Theorem \ref{main} for $\delta > 0$ and $\eta > 0$ small enough. This completes the proof of Theorem \ref{mainNonTec}. \end{proof}

\subsection{Proof of Theorem \ref{main}} In this section, we reduce the proof of Theorem \ref{main} to the following (even) more technical proposition: 
\begin{proposition}\label{mainProp} Let $\alpha,\beta,\tau \in (0,1]$, $C_{\alpha},C_{\beta} > 0$, and assume that
\begin{displaymath} \sigma > \tfrac{\beta}{2 - \alpha}. \end{displaymath}
 Then, there exist $\zeta = \zeta(\alpha,\beta,C_{\alpha},\sigma,\tau) > 0$ such that $\zeta$ stays bounded away from zero as long as $\sigma$ stays bounded away from $\beta/(2 - \alpha)$, and the following holds.
Assume that there exists a parameter $\eta_{0} > 0$, and a scale $\Delta_{0} > 0$, such that
\begin{displaymath} \mathcal{H}^{\tau}_{\infty}(\{\theta \in [0,1] : \mu(B(1) \cap H_{\theta}(\spt \mu,\Delta^{-\sigma},[\Delta,1])) \geq \Delta^{\eta_{0}}\}) \leq \Delta^{\eta_{0}}, \quad 0 < \Delta \leq \Delta_{0}, \end{displaymath}
whenever $\mu = \mu_{\alpha} \times \mu_{\beta}$ is a product of an $(\alpha,C_{\alpha})$-regular measure $\mu_{\alpha}$, and a $(\beta,C_{\beta})$-regular measure $\mu_{\beta}$. Then, there exists a parameter $\eta > 0$ and a scale $\delta_{0} > 0$, both depending only on $\alpha,\beta,C_{\alpha},C_{\beta},\sigma,\tau,\Delta_{0},\eta_{0}$, such that
\begin{displaymath} \mathcal{H}^{\tau}_{\infty}(\{\theta \in [0,1] : \mu(B(1) \cap H_{\theta}(\spt \mu,\delta^{-\sigma + \zeta},[\delta,1])) \geq \delta^{\eta}\}) \leq \delta^{\eta}, \qquad 0 < \delta \leq \delta_{0},\end{displaymath} 
whenever $\mu = \mu_{\alpha} \times \mu_{\beta}$ satisfies the same hypotheses as above.   \end{proposition}

The main point here is that, at the cost of decreasing $\delta$ and $\eta$, we may decrease the (relative) multiplicity from $\Delta^{-\sigma}$ to $\delta^{-(\sigma - \zeta)}$, where $\zeta > 0$ stays bounded from below until "$\sigma$" reaches $\beta/(2 - \alpha)$ (from above). Proposition \ref{mainProp} will prove Theorem \ref{main}, once it is coupled with the following trivial "base case", where $\sigma$ is assumed to be sufficiently large:
\begin{proposition}\label{baseProp} Let $\sigma > 1.1$. Then, there exists an absolute constant $\Delta_{0} > 0$ such that
\begin{displaymath} H_{\theta}(\spt \mu,\Delta^{-\sigma},[\Delta,1]) = \emptyset, \qquad 0 < \Delta \leq \Delta_{0}, \, \theta \in [0,1], \end{displaymath}
for all measures $\mu$ on $\R^{2}$. \end{proposition}
\begin{proof} Fix $\sigma > 1.1$. Then, note that if $K \subset \R^{2}$ is an arbitrary closed set, for example $K = \spt \mu$ for some measure on $\R^{2}$, then 
\begin{displaymath} \m_{K,\theta}(x \mid [\Delta,1]) := N_{\delta}(B(x,1) \cap K_{\delta} \cap \pi_{\theta}^{-1}\{\pi_{\theta}(x)\}) \leq 10\Delta^{-1} < \Delta^{-\sigma} \end{displaymath} 
for every $x \in \R^{2}$ and $\theta \in [0,1]$, as soon as $\Delta^{0.1} \leq 1/10$, that is $\Delta \leq 10^{-10}$. \end{proof}

Let us then prove Theorem \ref{main}, using Propositions \ref{mainProp} and \ref{baseProp}.
\begin{proof}[Proof of Theorem \ref{main}] Fix $\alpha,\beta,\tau \in (0,1]$ and $C_{\alpha},C_{\beta} > 0$. Let $\Sigma := \Sigma(\alpha,\beta,C_{\alpha},C_{\beta},\tau)$ be the infimum of those $\sigma > \beta/(2 - \alpha)$ such that we know how to prove Theorem \ref{main} for these fixed values of parameters. In other words, $\Sigma$ is the infimum of those $\sigma > \beta/(2 - \alpha)$ with the property that there exists a threshold $\Delta_{0} = \Delta_{0}(\alpha,\beta,\sigma,\tau,C_{\alpha},C_{\beta}) > 0$, and a parameter $\eta_{0} = \eta_{0}(\alpha,\beta,\sigma,\tau,C_{\alpha},C_{\beta}) > 0$, such that the following holds for all product measures $\mu = \mu_{\alpha} \times \mu_{\beta}$ as in \nref{A1}-\nref{A2}, and for all $0 < \delta \leq \Delta_{0}$:
\begin{equation}\label{form236} \mathcal{H}^{\tau}_{\infty}(\{\theta \in [0,1] : \mu(B(1) \cap H_{\theta}(\spt \mu,\delta^{-\sigma},[\delta,1])) \geq \delta^{\eta_{0}}\}) \leq \delta^{\eta_{0}}. \end{equation} 
We claim that in fact $\Sigma = \beta/(2 - \alpha)$ for all $\tau > 0$, which will prove Theorem \ref{main}.

Assume to the contrary that $\Sigma > \beta/(2 - \alpha)$. Then, let $\sigma > \Sigma$ be such that $\sigma - \zeta < \Sigma$, where $\zeta := \zeta(\alpha,\beta,C_{\alpha},\sigma,\tau)$ is the parameter appearing in Proposition \ref{mainProp}. This can be done, since $\zeta$ is bounded away from zero when $\sigma$ stays bounded away from $\beta/(2 - \alpha)$, and now this is true for all $\sigma > \Sigma$ (even $\sigma \geq \Sigma$), by the counter assumption. 

Since $\sigma > \Sigma$, there exist parameters $\Delta_{0} > 0$ and $\eta_{0} > 0$ such that \eqref{form236} holds for this specific "$\sigma$", and for all product measures $\mu = \mu_{\alpha} \times \mu_{\beta}$ satisfying \nref{A1}-\nref{A2}. This is precisely what Proposition \ref{mainProp} asks for, and therefore the proposition yields new parameters $\eta > 0$ and $\delta_{0} > 0$ such that \eqref{form236} holds with "$\sigma$" replaced by "$\sigma - \zeta$", with "$\eta_{0}$" replaced by "$\eta$", for all $0 < \delta \leq \delta_{0}$, and again for all product measures $\mu = \mu_{\alpha} \times \mu_{\beta}$ as in \nref{A1}-\nref{A2}. By definition of the number "$\Sigma$", this means that in fact $\Sigma \leq \sigma - \zeta$, and a contradiction has been reached. This completes the proof of Theorem \ref{main}. \end{proof}

\section{Proof of Proposition \ref{mainProp}}

The proof of Proposition \ref{mainProp} uses many rescaling arguments, and we start by checking that the class of "products of regular measures" is invariant under rescaling maps.  

\begin{remark}\label{rem1} Let $T_{z_{0},r_{0}} \colon \R^{2} \to \R^{2}$ be a rescaling map, with $z_{0} = (x_{0},y_{0}) \in \R^{2}$ and $r_{0} > 0$. Then $T_{z_{0},r_{0}}$ can be written as a product of rescaling maps on $\R$, namely
\begin{displaymath} T_{z_{0},r_{0}}(x,y) = \left(\tfrac{x - x_{0}}{r_{0}},\tfrac{y - y_{0}}{r_{0}} \right) = (T_{x_{0},r_{0}}(x), T_{y_{0},r_{0}}(y)), \qquad (x,y) \in \R^{2}. \end{displaymath}
Let $\mu = \mu_{\alpha} \times \mu_{\beta}$ be a product of an $(\alpha,C_{\alpha})$-regular measure $\mu_{\alpha}$ and a $(\beta,C_{\beta})$-regular measure $\mu_{\beta}$. Then, writing $\gamma := \alpha + \beta$, the rescaled and renormalised measure  
\begin{equation}\label{form64} \mu_{z_{0},r_{0}} := r_{0}^{-\gamma} \cdot T_{x_{0},r_{0}}\mu = (r_{0}^{-\alpha}T_{x_{0},r_{0}}\mu_{\alpha}) \times (r_{0}^{-\beta}T_{y_{0},r_{0}}\mu_{\beta}) =: \mu_{\alpha,x_{0},r_{0}} \times \mu_{\beta,y_{0},r_{0}} \end{equation}
can again be expressed as a product of a $(\alpha,C_{\alpha})$-regular and $(\beta,C_{\beta})$-regular measures 
\begin{displaymath} \mu_{\alpha,x_{0},r_{0}} := r_{0}^{-\alpha}T_{x_{0},r_{0}}\mu_{\alpha} \quad \text{and} \quad \mu_{\beta,y_{0},r_{0}} := r_{0}^{-\beta}T_{y_{0},r_{0}}\mu_{\beta}. \end{displaymath}
In other words, $\mu_{z_{0},r_{0}}$ is a product of two new regular measures with precisely the same constants as $\mu_{\alpha}$ and $\mu_{\beta}$. \end{remark}
\begin{notation}\label{n1} If $\mu$ is a $(\gamma,C_{\gamma})$-regular measure on $\R^{2}$ (nearly always a product of two regular measures on $\R$ in this paper), and $B = B(z,r) \subset \R^{2}$ is a disc, we write
\begin{displaymath} \mu_{B} := r^{-\gamma} \cdot T_{B}\mu, \end{displaymath}
which is another $(\gamma,C_{\gamma})$-regular measure on $\R^{2}$. More accurate notation would be $\mu_{B,\gamma}$, but the index "$\gamma$" should always be clear from context. \end{notation}

We then repeat the statement of Proposition \ref{mainProp}:
\begin{proposition}\label{mainPropTec} Let $\alpha,\beta,\tau \in (0,1]$, $C_{\alpha},C_{\beta} > 0$, and assume that
\begin{equation}\label{form233} \sigma > \tfrac{\beta}{2 - \alpha}. \end{equation}
 Then, there exist $\zeta = \zeta(\alpha,\beta,C_{\alpha},\sigma,\tau) > 0$ such that $\zeta$ stays bounded away from zero as long as $\sigma$ stays bounded away from $\beta/(2 - \alpha)$, and the following holds.
Assume that there exists a parameter $\eta_{0} > 0$, and a scale $\Delta_{0} > 0$, such that
\begin{equation}\label{form201} \mathcal{H}^{\tau}_{\infty}(\{\theta \in [0,1] : \mu(B(1) \cap H_{\theta}(\spt \mu,\Delta^{-\sigma},[\Delta,1])) \geq \Delta^{\eta_{0}}\}) \leq \Delta^{\eta_{0}}, \qquad 0 < \Delta \leq \Delta_{0}, \end{equation}
whenever $\mu = \mu_{\alpha} \times \mu_{\beta}$ is a product of an $(\alpha,C_{\alpha})$-regular measure $\mu_{\alpha}$, and a $(\beta,C_{\beta})$-regular measure $\mu_{\beta}$. Then, there exists a parameter $\eta > 0$ and a scale $\delta_{0} > 0$, both depending only on $\alpha,\beta,C_{\alpha},C_{\beta},\sigma,\tau,\Delta_{0},\eta_{0}$, such that
\begin{displaymath} \mathcal{H}^{\tau}_{\infty}(\{\theta \in [0,1] : \mu(B(1) \cap H_{\theta}(\spt \mu,\delta^{-\sigma + \zeta},[\delta,1])) \geq \delta^{\eta}\}) \leq \delta^{\eta}, \qquad 0 < \delta \leq \delta_{0},\end{displaymath} 
whenever $\mu = \mu_{\alpha} \times \mu_{\beta}$ satisfies the same hypotheses as above.   \end{proposition}

\begin{remark}\label{rem:outline} This remark is a continuation of the proof outline presented in Section \ref{s:outline}. The proof of Proposition \ref{mainPropTec} proceeds roughly in the manner we described there, with $A = \spt \mu_{\alpha}$, $B = \spt \mu_{\beta}$, and "exceptional set"
\begin{displaymath} E = \{\theta \in [0,1] : \mu(B(1) \cap H_{\theta}(\spt \mu,\delta^{-\sigma + \zeta},[\delta,1])) \geq \delta^{\eta}\}. \end{displaymath}
We make the counter assumption $\mathcal{H}^{\tau}_{\infty}(E) \geq \delta^{\eta}$, and fix $\theta_{0} \in E$. The purpose here is to point out two technical difficulties which we glossed over in Section \ref{s:outline}. The first one is related to the following sentence above \eqref{form152}: \emph{Recalling that $A \times B \subset A' \times \R$, up to a change of coordinates, one more precisely rewrites $A \times B = A' \times B'$, where $A' = A + \theta_{0}B$, and $B'$ represents the intersection of $A \times B$ with the "typical fibre" under the map $\pi_{\theta_{0}}(x,y) = x + \theta_{0}y$.} This was a rather inaccurate description: even though $A \times B \subset A' \times \R$, the fibres $(A \times B) \cap \pi_{\theta_{0}}^{-1}\{t\}$ are not so uniform, in general, that we could rewrite $A \times B = A' \times B'$.

We do the following instead, imitating an idea which first appeared in \cite{MR4055989}. We fix a small scale $\delta > 0$, and let $\mathbf{T} = \pi_{\theta_{0}}^{-1}(I)$ be an arbitrary tube of width $\delta^{1/2}$. Then, it roughly speaking turns out that there exists a product set of the form $A' \times B'$, where $A' = \pi_{\theta_{0}}((A \times B) \cap \mathbf{T})$, with the property
\begin{equation}\label{form151} N_{\delta}(A' + (\theta_{1} - \theta_{0})B') \sim N_{\delta}(\pi_{\theta}((A \times B) \cap \mathbf{T})), \qquad |\theta - \theta_{0}| \leq \delta^{1/2}. \end{equation}
This corresponds to \eqref{form121} in the actual proof (in reality, one needs to replace $A' \times B'$ by a "fat" subset $G \subset A' \times B'$, but this is only a minor technical problem). To summarise: rewriting $A \times B = A' \times B'$ is hopeless, but instead it is possible to associate to $(A \times B) \cap \mathbf{T}$ a product set $A' \times B' = \pi_{\theta_{0}}((A \times B) \cap \mathbf{T}) \times B'$, in the sense that \eqref{form151} holds.

In Section \ref{s:outline}, the next step was to \emph{fix another direction $\theta_{1} \in E$ such that $|\theta_{0} - \theta_{1}| \sim 1$.} In reality, with \eqref{form151} in mind, we rather need to find $\theta_{1} \in E$ such that $|\theta_{1} - \theta_{0}| \approx \delta^{1/2}$. This is non-trivial: our only assumption on $E$ is that $\mathcal{H}^{\tau}_{\infty}(E) \geq \delta^{\eta}$, and it may be impossible to find a pair of points $\theta_{0},\theta_{1} \in E$ with separation $|\theta_{0} - \theta_{1}| \approx \delta^{1/2}$, for a given scale $\delta > 0$. This issue is resolved by pigeonholing another scale $\bar{\delta} \in [\delta,1]$, which is not too much larger than $\delta$, and for which we can find two points $\theta_{0},\theta_{1} \in E$ such that $|\theta_{0} - \theta_{1}| \approx \bar{\delta}^{1/2}$. This is accomplished in Section \ref{s:pigeonholing}. 

After such a scale $\bar{\delta}$ has been located, the steps mentioned above for the pair $(\delta,\bar{\delta}^{1/2})$ are, in reality, carried out for the pair $(\bar{\delta},\bar{\delta}^{1/2})$. In particular, the tube $\mathbf{T} = \pi_{\theta_{0}}^{-1}(I)$ will have width $|I| = \bar{\delta}^{1/2}$. It will be chosen (in Section \ref{s:tubeSelection}) in such a manner that
\begin{equation}\label{form153} N_{\bar{\delta}}(\pi_{\theta_{1}}((A \times B) \cap \mathbf{T})) \approx N_{\bar{\delta}}(\pi_{\theta_{0}}((A \times B) \cap \mathbf{T})) = N_{\bar{\delta}}(A'), \end{equation}
see \nref{T2}. Via an analogue of \eqref{form151} at scale $\bar{\delta}$, this roughly implies that 
\begin{displaymath} N_{\bar{\delta}}(A' + (\theta_{1} - \theta_{0})B') \stackrel{\eqref{form151}}{\approx} N_{\bar{\delta}}(\pi_{\theta_{1}}((A \times B) \cap \mathbf{T})) \stackrel{\eqref{form153}}{\approx} N_{\bar{\delta}}(A'). \end{displaymath}
This equation is a more accurate analogue of \eqref{form152} from Section \ref{s:outline}. It can still be used in the same manner to draw conclusions about the "branching" structure of $A'$ and $B'$.  \end{remark}

We begin the proof of Proposition \ref{mainPropTec}. We fix the parameters $\alpha,\beta,\tau \in (0,1]$, $C_{\alpha},C_{\beta} > 0$, and we let $\Delta_{0},\eta_{0} > 0$ be constants such that \eqref{form201} holds for all products $\mu = \mu_{\alpha} \times \mu_{\beta}$ of an $(\alpha,C_{\alpha})$-regular measure $\mu_{\alpha}$ and a $(\beta,C_{\beta})$-regular measure $\mu_{\beta}$. 

\subsection{Choosing constants}\label{s:constants} We take a moment to list and specify a few other constants. First of all, we write $\gamma := \alpha + \beta$ and $C_{\gamma} := 5C_{\alpha}C_{\beta}$. Then $\mu = \mu_{\alpha} \times \mu_{\beta}$ is a $(\gamma,C_{\gamma})$-regular measure on $\R^{2}$. Second, we note that the condition $\sigma > \beta/(2 - \alpha)$ is equivalent to 
\begin{displaymath} D(\alpha,\beta,\sigma) := (2 - \alpha) \sigma - \beta = (1 - \alpha)\sigma + \alpha - \gamma + \sigma > 0. \end{displaymath}
We specify small constants $\epsilon,\rho,\zeta_{0} \in (0,1)$, and a large integer $m_{0} \in \N$, such that the following inequality holds for all $m \geq m_{0}$:
\begin{equation}\label{form237} \gamma - \sigma + 10\zeta_{0} < (1 - \alpha - \epsilon)(\sigma - 10(\epsilon + \zeta_{0})) + (1 - \rho)(1 - \tfrac{30(\epsilon + \zeta_{0})}{\alpha \rho} - \tfrac{10C_{\alpha}}{\alpha \rho m})\alpha. \end{equation}
Since $\gamma - \sigma = (1 - \alpha)\sigma + \alpha - D(\alpha,\beta,\sigma) < (1 - \alpha)\sigma + \alpha$, it is qualitatively clear that the constants $\epsilon,\rho,\zeta_{0}$, and $m_{0}$ can be chosen so that \eqref{form237} holds, but let us be more specific about their dependencies. The constant $\rho > 0$ should be chosen first: if $0 < \rho < \tfrac{1}{2} \cdot D(\alpha,\beta,\sigma)$, then
\begin{displaymath} \gamma - \sigma = (1 - \alpha)\sigma + \alpha < (1 - \alpha)\sigma + (1 - \rho)\alpha - \tfrac{D(\alpha,\beta,\sigma)}{2}, \end{displaymath}
since $\alpha \leq 1$. Next, the remaining three constants $\epsilon,\zeta_{0} \in (0,1)$, and $m_{0} \geq 1$, can be chosen in an arbitrary order such that \eqref{form237} holds for all $m \geq m_{0}$. Evidently $\epsilon,\zeta_{0}$ only depend on $\alpha,\beta$ and $\rho$ (hence on $\alpha,\beta$, and $D(\alpha,\beta,\sigma)$). The constant $m_{0}$ additionally depends on the regularity constant "$C_{\alpha}$".

The notation "$\zeta_{0}$" suggests that this constant should have something to do with the constant $\zeta = \zeta(\alpha,\beta,C_{\alpha},\sigma,\tau) > 0$, whose existence is the main claim in Proposition \ref{mainProp}. To specify this connection, we should state Shmerkin's inverse theorem \cite[Theorem 2.1]{Sh}, but the statement is so long that we postpone the full details to Theorem \ref{shmerkin}. However, the theorem begins with the following words: \emph{For every $\epsilon > 0$ and $m_{0} \geq 1$ there exists $\kappa = \kappa(\epsilon,m_{0}) > 0$ and $m \geq m_{0}$ such that the following holds for large enough $N$.} Now, the number $\zeta > 0$ from the claim of Proposition \ref{mainProp} can be taken to be any parameter satisfying
\begin{equation}\label{form238} 0 < O \cdot \zeta \leq \min\{\zeta_{0},\kappa(\epsilon,m_{0})\}, \end{equation}
where $\epsilon,\zeta_{0},m_{0}$ are familiar from the discussion above, and $O = O(\tau) > 0$ will be a constant depending only $\tau$. In particular, since these constants together only depend on $\alpha,\beta,\sigma,\tau$, and $C_{\alpha}$, the same will be true for the constant $\zeta > 0$.

It is not a mistake (as far as I know!) that the constants $\epsilon,\rho,\zeta_{0},m_{0}$ do not depend on the regularity constant "$C_{\beta}$" appearing in Proposition \ref{mainPropTec}. This constant is present in $C_{\gamma} = 5C_{\alpha}C_{\beta}$, and will also influence the thresholds for the parameters "$\eta$" and "$\delta_{0}$".

\subsection{The counter assumption} Now that we have clarified the roles of future constants, we make a counter assumption: the conclusion of the proposition fails for a certain product measure $\mu = \mu_{\alpha} \times \mu_{\beta}$, where $\mu_{\alpha}$ is $(\alpha,C_{\alpha})$-regular, $\mu_{\beta}$ is $(\beta,C_{\beta})$-regular, for a certain small scale $\delta \in (0,\Delta_{0}]$, and a certain parameter $\eta \in (0,\eta_{0}]$:
\begin{equation}\label{form200} \mathcal{H}^{\tau}_{\infty}(\{\theta \in [0,1] : \mu(B(1) \cap H_{\theta}(K,\delta^{-\sigma + \zeta},[\delta,1])) \geq \delta^{\eta}\}) \geq \delta^{\eta},\end{equation}
Here $K := \spt \mu$. The reader should think that $\delta \ll \Delta_{0}$ and $\eta \ll \eta_{0}$, and additionally that $\eta$ is small compared to the previously fixed small constants $\epsilon,\eta_{0},\zeta \leq \zeta_{0},\tau$. We should be on the safe side, if we assume that
\begin{equation}\label{form256} 0 < \eta < C^{-1} \cdot (\alpha \beta \epsilon \rho \tau \eta_{0}\zeta)^{C} \end{equation}
for a suitable absolute constant $C \geq 1$, so in particular "$\eta$" will not depend on $\Delta_{0}$. The threshold for $\delta > 0$ for which \eqref{form200} leads to a contradiction will only depend (in principle effectively) on the parameters $\alpha,\beta,C_{\alpha},C_{\beta},\sigma,\tau,\Delta_{0}$, and $\eta_{0}$, but this dependence will not be tracked explicitly. 

\begin{remark}\label{dyadicRem} On several occasions, we will need to assume that the scale $\delta > 0$ has a special form: very commonly $\delta > 0$ needs to be  a dyadic number, and later on \emph{a fortiori} we need that $\delta = (2^{-m})^{N}$ for some integer $N \geq 1$, where $m \geq m_{0}$ is the integer produced by Shmerkin's inverse theorem with initial data $\epsilon$ and $m_{0}$ (see above \eqref{form238}). Such assumptions are harmless: if our counter assumption \eqref{form200} involves a parameter $\delta > 0$, which is not of the correct form, then there exists a parameter $\delta' \in [\delta,C\delta]$ of the correct form, where $C \geq 1$ only depends on the usual parameters listed above. But now
\begin{displaymath} H_{\theta}(K,\delta^{-\sigma + \zeta},[\delta,1]) \subset H_{\theta}(K,\tfrac{1}{C} \cdot (\delta')^{-\sigma + \zeta},[\delta',1]) \subset H_{\theta}(K,(\delta')^{-\sigma + 2\zeta},[\delta',1]), \end{displaymath} 
where the first inclusion follows from Lemma \ref{lemma6}, and the second inclusion is true as soon as $\delta' > 0$ is so small that $\tfrac{1}{C} \geq (\delta')^{\zeta}$. Using this inclusion, we may find arbitrarily small values of $\delta'$ (a scale of the "correct form") such that the counter assumption \eqref{form200} holds with "$2\zeta$" in place of "$\zeta$". We can then use this variant of \eqref{form200}, instead, to derive a contradiction. In the sequel, we will assume that $\delta$ (and a certain other scales "$\bar{\delta}$" derived from $\delta$) is of the "correct form" without further remark. \end{remark}

Since the set $E := \{\theta \in [0,1] : \mu(B(1) \cap H_{\theta}(K,\delta^{-\sigma + \zeta},[\delta,1])) \geq \delta^{\eta}\}$ is assumed, by \eqref{form200}, to have $\tau$-dimensional Hausdorff-content bounded from below by $\delta^{\eta}$, we may find (see \cite[Theorem 8.8]{zbMATH01249699})) a $(\tau,C\delta^{-\eta})$-Frostman probability measure $\nu$ with $\spt \nu \subset E$. Here $C > 0$ is an absolute constant. From this point on, the counter assumption \eqref{form200} will only be used via the following:
\begin{equation}\label{form255} \mu(B(1) \cap H_{\theta}(K,\delta^{-\sigma + \zeta},[\delta,1])) \geq \delta^{\eta}, \qquad \theta \in \spt \nu. \end{equation}

\begin{notation} We will denote by "$O$" a generic large constant which may depend on $\alpha,\beta,C_{\alpha},C_{\beta},C_{\gamma},\sigma,\tau$. In similar spirit, we will denote by "$\omega$" a small positive constant, which is bounded away from zero in a manner depending only on $\alpha,\beta,C_{\alpha},C_{\beta},C_{\gamma},\sigma,\tau$. In fact, $\omega = O^{-1}$. If either $O$ or $\omega$ only depends on a subset of the parameters above, this will occasionally be emphasised by writing, for example, "$O(\alpha,\beta)$" instead of "$O$". 

The precise values of the constants "$O$" and "$\omega$" may vary from line to line. This will often lead to inequalities of the form "$2O \leq O$", which are not typos. \end{notation}

\subsection{Pigeonholing a branching scale for $\nu$}\label{s:pigeonholing} The measure $\nu$ may have "no branching" between the scales $\delta^{1/2}$ and $\delta$. Defining this defect carefully is not worth the effort, but we roughly mean the possibility that
\begin{displaymath} N(\spt \nu,\delta) \approx N(\spt \nu,\delta^{1/2}). \end{displaymath}
Fortunately, it follows from the $\tau$-Frostman property of $\nu$ that opposite behaviour must occur at some scale $\bar{\delta} > \delta$, which also satisfies the (roughly) converse inequality $\bar{\delta} \leq \delta^{\omega(\alpha,\beta,\sigma,\tau)}$. Finding the scale $\bar{\delta}$, and making these statements more precise, is the goal of this section. The final conclusion will be \eqref{form73}. 

Recall that $\nu$ is a $(\tau,C\delta^{-\eta})$-Frostman probability measure satisfying \eqref{form255}. Recall that $\epsilon = \epsilon(\alpha,\beta,\sigma) > 0$ was one of the constants fixed in Section \ref{s:constants}. We define the following increasing scale sequence: $\delta_{0} := \delta^{(1 + \epsilon)/2}$, and 
\begin{displaymath} \delta_{j + 1} := \delta_{j}^{1/(1 + \epsilon)}, \qquad j \geq 1. \end{displaymath} 
Thus $\delta_{1} = \delta^{1/2}$, and $\delta_{j} = \delta_{0}^{1/(1 + \epsilon)^{j}} = (\delta^{1/2})^{1/(1 + \epsilon)^{j - 1}}$ for $j \geq 0$. Recall that $\calD_{\delta_{j}}$ is the partition of $[0,1)$ into dyadic intervals of length $\delta_{j}$ (if these numbers are not dyadic to begin with, the closest dyadic numbers would work as well). We mention that $\mathcal{D}_{\delta_{j + 1}}$ consists of intervals longer than those in $\mathcal{D}_{\delta_{j}}$. We will prove the following lemma:
\begin{lemma}\label{lemma8} There exists an index $j \leq O(\alpha,\beta,\sigma,\tau)$, and a Borel set $G \subset [0,1]$ with $\nu(G) \geq \omega(\alpha,\beta,\sigma,\tau)$ such that the renormalised measure $\bar{\nu} = \nu(G)^{-1} \cdot \nu|_{G}$ has the following properties:
\begin{itemize}
\item $\bar{\nu}$ is a $(\tau,O\delta^{-\eta})$-Frostman measure with $O = O(\alpha,\beta,\sigma,\tau) > 1$, 
\item If $I \in \mathcal{D}_{j + 1}$ with $\bar{\nu}(I) > 0$, then
\begin{equation}\label{form149} \bar{\nu}_{I}(J) \leq \delta^{\omega(\alpha,\beta,\sigma,\tau)}, \qquad J \in \mathcal{D}_{\delta_{j}}(I). \end{equation}
\end{itemize}
Here $\bar{\nu}_{I} = \bar{\nu}(I)^{-1} \cdot \bar{\nu}|_{I}$, and $\mathcal{D}_{\delta_{j}}(I)$ refers to the intervals in $\mathcal{D}_{\delta_{j}}$ which are contained in $I$. Note that the length of these intervals is $|J| = \delta_{j} = \delta_{j + 1}^{1 + \epsilon} = |I|^{1 + \epsilon}$.
\end{lemma}

The relation between Lemma \ref{lemma8}, and our search for the scale "$\bar{\delta}$", is simply that 
\begin{displaymath} \bar{\delta} := \delta_{j + 1}^{2}, \quad \text{hence} \quad \delta_{j + 1} = \bar{\delta}^{1/2}. \end{displaymath}
Since $1 < 1 + \epsilon \leq 2$, we then have $\bar{\delta} \leq \delta_{j} \ll \bar{\delta}^{1/2}$. The inequality \eqref{form149} informally says that the measure $\bar{\nu}$ has non-trivial branching between the scales $\delta_{j}$ and $\delta_{j + 1} = \bar{\delta}^{1/2}$, and so in particular between the scales $\bar{\delta}$ and $\bar{\delta}^{1/2}$.

We begin the search for $j_{0}$ and $G$. Since $\nu$ is a $(\tau,C\delta^{-\eta})$-Frostman probability measure on $[0,1]$, we have the uniform bound
\begin{displaymath} \nu(I_{0}) \leq C\delta^{-\eta}\delta_{0}^{\tau} \qquad I_{0} \in \mathcal{D}_{\delta_{0}}, \end{displaymath} 
which gives the following lower bound for the entropy of $\nu$ at scale $\delta_{0}$:
\begin{equation}\label{form21} H(\nu,\mathcal{D}_{\delta_{0}}) = \sum_{I_{0} \in \mathcal{D}_{\delta_{0}}} \nu(I_{0}) \log \tfrac{1}{\nu(I_{0})} \geq \log \delta_{0}^{-\tau} - \log C - \log \delta^{-\eta} \geq \log \delta^{-\tau/2}. \end{equation}
In the final inequality, we used that $\delta_{0}^{-\tau} = \delta^{-\tau (1 + \epsilon)/2} = \delta^{-\tau/2} \cdot \delta^{-\epsilon \tau/2}$, and the logarithm of the second factor exceeds the "error term" $\log C + \log \delta^{-\eta}$ if $\delta > 0$ is sufficiently small, and $\eta < \tau \epsilon/2$ (as is implied by \eqref{form256}).

Now, let $n \sim_{\epsilon,\tau} 1$ (therefore $n \leq O(\alpha,\beta,\sigma,\tau)$) be the smallest integer satisfying
\begin{equation}\label{form257} \frac{1}{2(1 + \epsilon)^{n - 1}} \leq \frac{\tau}{4}. \end{equation}
This is not too important, but then in fact $n \sim \epsilon^{-1} \cdot \log (1/\tau)$. Then, observing that $\mathcal{D}_{\delta_{n}}$ consists of intervals of length $\delta_{0}^{1/(1 + \epsilon)^{n}} = \delta^{1/[2(1 + \epsilon)^{n - 1}]}$, we have
\begin{displaymath} H(\nu,\mathcal{D}_{\delta_{n}}) \leq \log |\mathcal{D}_{\delta_{n}}| \leq \log \delta^{-1/[2(1 + \epsilon)^{n - 1}]} \leq \log \delta^{-\tau/4}. \end{displaymath} 
Combining this estimate with \eqref{form21}, and using the conditional entropy formula \eqref{entropy} repeatedly to the nested partitions $\mathcal{D}_{\delta_{j}}$, we find that
\begin{displaymath} \log \delta^{-\tau/4} \leq H(\nu,\mathcal{D}_{\delta_{0}}) - H(\nu,\mathcal{D}_{\delta_{n}}) = \sum_{j = 0}^{n - 1} H(\nu,\mathcal{D}_{\delta_{j}} \mid \mathcal{D}_{\delta_{j + 1}}). \end{displaymath}
Consequently, there exists $j \in \{0,\ldots,n - 1\}$ with the property
\begin{equation}\label{form23} \sum_{I \in \mathcal{D}_{\delta_{j + 1}}} \nu(I) \cdot H(\nu_{I},\mathcal{D}_{\delta_{j}}) = H(\nu,\mathcal{D}_{\delta_{j}} \mid \mathcal{D}_{\delta_{j + 1}}) \geq \log \delta^{-\tau/(4n)}. \end{equation} 
Here $\nu_{I} = \nu(I)^{-1} \nu|_{I}$ for any $I \in \mathcal{D}_{\delta_{j + 1}}$ with $\nu(I) \neq 0$. As we discussed after the statement of Lemma \ref{lemma8}, we set $\bar{\delta} = \delta_{j + 1}^{2}$ (note that $\delta_{j + 1} \geq \delta_{1} = \delta^{1/2}$, so $\bar{\delta} \geq \delta$). Let us record that
\begin{equation}\label{form75} \bar{\delta} \leq \delta^{1/(1 + \epsilon)^{n - 1}} \leq \delta^{\tau/4},  \end{equation}
by \eqref{form257}, or equivalently $\delta \geq \bar{\delta}^{4/\tau}$. Indeed, since "$n$" is the smallest number satisfying \eqref{form257}, we have $1/(2(1 + \epsilon)^{n - 2}) > \tau/4$, and hence $1/(1 + \epsilon)^{n - 1} \geq (1 + \epsilon)^{-1} \cdot \tau/2 \geq \tau/4$. Thus $\bar{\delta}$ remains somewhat comparable to $\delta$, in a manner depending only on the fixed parameter "$\tau$". Motivated by \eqref{form23}, we set
\begin{displaymath} \bar{\tau} := \tau/(4n) \sim \epsilon \cdot \tau/\log(1/\tau) \geq \epsilon \tau^{2}, \end{displaymath}
and we note that $\bar{\tau} \geq \omega(\alpha,\beta,\sigma,\tau)$.

The choice of the index $j \in \{0,\ldots,n - 1\}$ at \eqref{form23} roughly tells us that for many "long" intervals $I \in \mathcal{D}_{\delta_{j + 1}}$ with $\nu(I) > 0$, the re-normalised restriction $\nu_{I}$ is not concentrated on very few "short" sub-intervals of $I$ in the family $\mathcal{D}_{\delta_{j}}$. By restricting $\nu$ a little bit, we may replace the word "many" by "all", as we will see next. In fact, since certainly
\begin{displaymath} H(\nu_{I},\mathcal{D}_{\delta_{j}}) \leq \log \delta^{-1}, \qquad I \in \mathcal{D}_{\delta_{j + 1}}, \end{displaymath}
we may infer from \eqref{form23} that there exists a collection of arcs $\mathcal{G}_{j + 1} \subset \mathcal{D}_{\delta_{j + 1}}$ of total $\nu$-measure 
\begin{equation}\label{form239} \nu(\cup \mathcal{G}_{j + 1}) \geq \bar{\tau}/2 \end{equation} 
with the properties $\nu(I) > 0$ and $H(\nu_{I},\mathcal{D}_{\delta_{j}}) \geq \log \delta^{-\bar{\tau}/2}$ for all $I \in \mathcal{G}_{j + 1}$. Indeed, if this failed, then
\begin{align*} \sum_{I \in \mathcal{D}_{\delta_{j + 1}}} \nu(I) H(\nu_{I},\mathcal{D}_{\delta_{j}}) & < \sum_{H(\ldots) \geq \log \delta^{-\bar{\tau}/2}} \nu(I) \cdot \log \delta^{-1} + \sum_{H(\ldots) < \log \delta^{-\bar{\tau}/2}} \nu(I) \cdot \log \delta^{-\bar{\tau}/2}\\
& \leq \frac{\bar{\tau}}{2} \cdot \log \delta^{-1} + \log \delta^{-\bar{\tau}/2} = \log \delta^{-\bar{\tau}}, \end{align*} 
which contradicts \eqref{form23}. We write $G_{j + 1} := \cup \mathcal{G}_{j + 1}$, and we renormalise $\nu$ to $G_{j + 1}$:
\begin{displaymath} \bar{\nu}_{1} := \tfrac{1}{\nu(G_{j + 1})} \cdot \nu|_{G_{j + 1}}. \end{displaymath}
Observe, by \eqref{form239}, that $\bar{\nu}_{1} $ is a $(\tau,O\delta^{-\eta})$-Frostman probability measure with the additional feature that if $I \in \mathcal{D}_{\delta_{j + 1}}$ is an arc with $\bar{\nu}_{1}(I) > 0$, then
\begin{equation}\label{form24} H(\bar{\nu}_{1,I},\mathcal{D}_{\delta_{j}}) = H(\nu_{I},\mathcal{D}_{\delta_{j}}) \geq \log \delta^{-\bar{\tau}/2}. \end{equation} 
It is worth noting that $\nu_{I} = \bar{\nu}_{1,I}$ for all $I \in \mathcal{G}_{j + 1}$ ($= \{I \in \mathcal{D}_{\delta_{j + 1}} : \nu_{1}(I) > 0\}$). The inequality \eqref{form24} means that the measure $\bar{\nu}_{1} $ cannot be completely concentrated inside any single interval of length $\delta_{j} = |I|^{1 + \epsilon}$. What we would prefer is, \emph{a fortiori}, that $\nu_{I} = \bar{\nu}_{1,I}$ satisfies a Frostman-type estimate at the smaller scale $\delta_{j}$, as stated in \eqref{form149}. This will be achieved by renormalising $\bar{\nu}_{1}$ further to the measure $\bar{\nu}$, which finally satisfies Lemma \ref{lemma8}.

Fix $I \in \mathcal{D}_{\delta_{j + 1}}$ with $\nu_{1}(I) > 0$, write $\mathcal{D}_{\delta_{j}}(I) := \{J \in \mathcal{D}_{\delta_{j}} : J \subset I\}$ for $I \in \mathcal{D}_{j + 1}$, and then estimate
\begin{displaymath} \log \delta^{-\bar{\tau}/2} \stackrel{\eqref{form24}}{\leq} \sum_{J \in \mathcal{D}_{\delta_{j}}(I)} \nu_{I}(J)\log \tfrac{1}{\nu_{I}(J)} \leq \sum_{\nu_{I}(J) \geq \delta^{\bar{\tau}/4}} \nu_{I}(J) \log \delta^{-\bar{\tau}/4} + \sum_{\nu_{I}(J) < \delta^{\bar{\tau}/4}} \nu_{I}\log \tfrac{1}{\nu_{I}(J)}. \end{displaymath}
Since $\nu_{I}$ is a probability measure, the first term is bounded from above by $\tfrac{1}{2} \log \delta^{-\bar{\tau}/2}$, and consequently the second term has the lower bound
\begin{equation}\label{form22} \sum_{\nu_{I}(J) < \delta^{\bar{\tau}/4}} \nu_{I}(J)\log \tfrac{1}{\nu_{I}(J)} \geq \log \delta^{-\bar{\tau}/4}. \end{equation}
On the other hand, if the total $\nu_{I}$-measure of the intervals in the sum above, namely $\mathcal{G}_{I} := \{J \in \mathcal{D}_{\delta_{j}}(I) : 0 < \nu_{I}(J) < \delta^{\bar{\tau}/4}\} \subset \mathcal{D}_{\delta_{j}}$, is denoted by $m_{I} := \nu_{I}(\cup \mathcal{G}_{I})$ we have
\begin{align*} \log \delta^{-\bar{\tau}/4} & \stackrel{\eqref{form22}}{\leq} \sum_{\nu_{I}(J) < \delta^{\bar{\tau}/4}} \nu_{I}(J)\log \tfrac{1}{\nu_{I}(J)} = m_{I}\sum_{J \in \mathcal{G}_{I}} \frac{\nu_{I}(J)}{m_{I}} \log \tfrac{1}{\nu_{I}(J)}\\
&  \leq m_{I} \log \left(\sum_{J \in \mathcal{G}_{I}} \tfrac{1}{m_{I}} \right) = m_{I} \log |\mathcal{G}_{I}| - m_{I} \log m_{I}, \end{align*}
by Jensen's inequality applied to the discrete probability measure $J \mapsto \nu_{I}(J)/m_{I}$ on $\mathcal{G}_{I}$. Since $m_{I} \in (0,1]$, the second term satisfies $|m_{I} \log m_{I}| \leq \tfrac{1}{2}$, and on the other hand $|\mathcal{G}_{I}| \leq |\mathcal{D}_{\delta_{j}}(I)| \leq \delta^{-1}$, by a crude estimate. Assuming that $\delta > 0$ is so small that $\log \delta^{-\bar{\tau}/4} - \tfrac{1}{2} \geq \log \delta^{-\bar{\tau}/8}$, we find from the estimate above that
\begin{displaymath} \log \delta^{-m_{I}} \geq \log |\mathcal{G}_{I}|^{m_{I}} \geq \log \delta^{-\bar{\tau}/4} + m_{I} \log m_{I} \geq \log \delta^{-\bar{\tau}/8}, \end{displaymath}
and consequently $\nu_{I}(\cup \mathcal{G}_{I}) = m_{I} \geq \bar{\tau}/8$. In other words, if we define
\begin{displaymath} \mathcal{G}_{j} := \bigcup_{I \in \mathcal{G}_{j + 1}} \mathcal{G}_{I} \subset \mathcal{D}_{\delta_{j}} \quad \text{and} \quad G := G_{j} := \cup \mathcal{G}_{j} \subset G_{j + 1}, \end{displaymath}
then 
\begin{displaymath} \bar{\nu}_{1}(G) = \sum_{I \in \mathcal{G}_{j + 1}} \bar{\nu}_{1}(I) \cdot \nu_{I}(\cup \mathcal{G}_{I}) \geq \bar{\tau}/8. \end{displaymath}
Therefore, the measure
\begin{displaymath} \bar{\nu} := \tfrac{1}{\bar{\nu}_{1}(G)} \cdot \bar{\nu}_{1}|_{G} = \tfrac{1}{\nu(G)} \cdot \nu|_{G} \end{displaymath}
remains a $(\tau,O\delta^{-\eta})$-Frostman probability measure with the feature that if $\bar{\nu}(I) > 0$ for some $I \in \mathcal{D}_{\delta_{j + 1}}$, then
\begin{equation}\label{form25} \bar{\nu}_{I}(J) \leq \delta^{\bar{\tau}/4}, \qquad J \in \mathcal{D}_{\delta_{j}}(I). \end{equation}
Note that $\bar{\nu}$ was finally defined by restricting the original measure "$\nu$" to a certain union $G$ of dyadic intervals (of length $\delta_{j}$), whose total $\nu$ measure is bounded from below by $\geq \omega(\alpha,\beta,\sigma,\tau)$. Thus $\bar{\nu}$ remains a $(\tau,O\delta^{-\eta})$-Frostman probability measure with the property
\begin{equation}\label{form72} \mu(B(1) \cap H_{\theta}(K,\delta^{-\sigma + \zeta},[\delta,1])) \geq \delta^{\eta}, \qquad \theta \in \spt \bar{\nu} \subset E. \end{equation}
We have now proven Lemma \ref{lemma8}.

Since $\bar{\nu}$ satisfies roughly the same hypotheses as $\nu$ (and additionally the Frostman property \eqref{form25}), we redefine $\nu := \bar{\nu}$ to simplify notation. We also recall that $\bar{\delta} := \delta_{j + 1}^{2}$ (so if $j = 0$, simply $\bar{\delta} = \delta$). We also write $\mathcal{D}_{1} := \mathcal{D}_{\delta_{j}}$ for the dyadic partition of $[0,1)$ into intervals of length $\bar{\delta} = \bar{\delta}^{1}$, and we write $\mathcal{D}_{1/2} := \mathcal{D}_{\delta_{j + 1}}$ (the dyadic partition to intervals of length $\bar{\delta}^{1/2} = \delta_{j + 1}$). Then, \eqref{form25} implies that
\begin{equation}\label{form73} I \in \mathcal{D}_{1/2} \text{ and } \nu(I) > 0 \quad \Longrightarrow \quad \nu_{I}(J) \leq \delta^{\omega(\alpha,\beta,\sigma,\tau)} \leq \bar{\delta}^{\omega(\alpha,\beta,\sigma,\tau)} \text{ for } J \in \mathcal{D}_{(1 + \epsilon)/2}(I), \end{equation}
where $\mathcal{D}_{(1 + \epsilon)/2}(I) \subset \mathcal{D}_{\delta_{j}}$ are the dyadic intervals of length $\delta_{j} = \delta_{j + 1}^{1 + \epsilon} = \bar{\delta}^{(1 + \epsilon)/2}$ contained in $I$. To close the section, we observe that it is well possible that $\bar{\delta} = \delta_{j + 1}^{2} = \delta$; this happens if the index $j \in \{0,\ldots,n - 1\}$ fixed at \eqref{form23} happens to be $j = 0$, so $\bar{\delta} = \delta_{0 + 1}^{2} = (\delta^{1/2})^{2} = \delta$. This will lead to a simpler special case of the proof below. On the the other hand, if $\bar{\delta} > \delta$, then $\bar{\delta} = \delta_{j + 1}^{2}$ for some $j \geq 1$. Consequently $\bar{\delta} \geq \delta_{2}^{2} = \delta^{1/(1 + \epsilon)}$, and hence $\bar{\delta}$ is "much longer" than $\delta$: in fact
\begin{equation}\label{form203} \frac{\delta}{\bar{\delta}} \leq \delta^{1 - 1/(1 + \epsilon)}. \end{equation}
Note that the right hand side is a positive power of $\delta$.

\subsection{High multiplicity sets at scale $\bar{\delta}$} Now we have found a scale $\bar{\delta} \in [\delta,\delta^{\omega}]$ such that the measure $\nu$ has non-trivial "branching" between the scales $\bar{\delta}$ and $\bar{\delta}^{1/2}$, as quantified in \eqref{form73}. The following problem now materialises: our counter assumption \eqref{form255} concerned the scale $\delta$, and there is a risk that all the information is lost when replacing "$\delta$" by "$\bar{\delta}$". We resolve the issue by proving the following lemma: 
\begin{lemma}\label{lemma9} There exists a subset $S \subset [0,1]$ with $\nu(S) \geq \omega \cdot \delta^{\eta}$ with the property
\begin{equation}\label{form74} \mu_{B(5)}(B(1) \cap H_{\theta}(T_{B(5)}(K),\bar{\delta}^{-\sigma + O(\tau)\zeta},[\bar{\delta},1])) \geq \omega \cdot \delta^{\eta}, \quad \theta \in S. \end{equation}
\end{lemma}

Thus, modulo replacing "$\zeta$" by "$O \cdot \zeta$", which is harmless for our purposes, the counter assumption \eqref{form255} at scale "$\delta$" can be used (in cooperation with our hypothesis \eqref{form201}) to infer similarly bad behaviour at scale "$\bar{\delta}$" for the dilated regular set $T_{B(5)}(K)$, and the renormalised measure $\mu_{B(5)}$ supported on $T_{B(5)}(K)$. As we discussed in Remark \ref{rem1}, the measure $\mu_{B(5)}$ is of the same form as $\mu$, with precisely the same constants. In particular our inductive hypothesis \eqref{form201} may later be applied to $\mu_{B(5)}$ and $T_{B(5)}(K)$.

Note that if $\delta = \bar{\delta}$, then every $\theta \in \spt \nu$ already satisfies \eqref{form74} (with $O(\alpha,\beta,\sigma,\tau) = 1$ and $\mu,K$ in place of $\mu_{B(5)},T_{B(5)}(K)$) by virtue of our initial counter assumption \eqref{form255}. In this case the argument in the present section will not be needed. So, for the time being, we will assume that $\bar{\delta} > \delta$, which implies by \eqref{form203} that $\bar{\delta}$ is substantially larger than $\delta$.

In this case, we will apply Proposition \ref{prop3}, whose statement we include here, but whose lengthy proof is postponed to Section \ref{s:prop2}:
\begin{proposition}\label{prop3} Let $\theta \in [0,1]$, and let $1 \leq M \leq N < \infty$ be constants, let $0 < r \leq R \leq 1$, and let $\mu$ be a $(\gamma,C_{\gamma})$-regular measure with $K := \spt \mu \subset \R^{2}$. Abbreviate $\mu_{s} := \mu|_{B(s)}$ for $s > 0$. Then, there exist absolute constants $c,C > 0$ such that
\begin{equation}\label{form202} \mu_{1}(H_{\theta}(K,CN,[r,1])) \leq \mu_{1}(H_{\theta}(K,cM,[4R,5])) + CC_{\gamma}^{2} \cdot \mu_{4}(H_{\theta}(K,c\tfrac{N}{M},[4r,7R])).  \end{equation} 
\end{proposition}
We will apply the proposition to the $(\gamma,C_{\gamma})$-regular measure $\mu = \mu_{\alpha} \times \mu_{\beta}$ with the parameters $M \leq N$ such that 
\begin{displaymath} CN = \delta^{-\sigma + \zeta} \quad \text{and} \quad c \cdot \tfrac{N}{M} = \left(\delta/\bar{\delta} \right)^{-\sigma}, \end{displaymath}
from which we may solve that
\begin{equation}\label{form83} c M = c^{2}\left(\tfrac{\delta}{\bar{\delta}} \right)^{\sigma} \cdot N = \tfrac{c^{2}}{C} \left(\tfrac{\delta}{\bar{\delta}} \right)^{\sigma} \cdot \delta^{-\sigma + \zeta} = \tfrac{c^{2}}{C} \cdot \bar{\delta}^{-\sigma}\cdot \delta^{\zeta} \stackrel{\eqref{form75}}{\geq} \bar{\delta}^{-\sigma + O(\tau)\zeta}. \end{equation}
In the final inequality, we first took $\delta$ so small that $(c^{2}/C) \geq \delta^{\zeta}$, and then we used that $\bar{\delta} \leq \delta^{\omega(\tau)}$ or, (see \eqref{form75}) more precisely $\delta \geq \bar{\delta}^{4/\tau}$. It will be convenient to abbreviate
\begin{equation}\label{form84} \bar{\zeta} := O(\tau) \cdot \zeta > 0, \end{equation}
where the implicit constant $O(\tau)$ is determined by \eqref{form83}. This constant is the one which appears in \eqref{form74}, and it is also the one which was already mentioned in \eqref{form238}. More precisely, at the end of the day, we will require from the parameter "$\zeta$" appearing in Proposition \ref{mainProp} that 
\begin{equation}\label{form141} 0 < 6 \cdot \bar{\zeta} = 6 \cdot O(\tau) \cdot \zeta < \min\{\zeta_{0},\kappa(\epsilon,m_{0})\}. \end{equation} 

For the scales $r \leq R$ appearing in the statement of Proposition \ref{prop3}, we will take $r := \delta$ and $R := \bar{\delta}/4 \geq \delta$ (by \eqref{form203}, and assuming that $0 < \delta \leq \delta_{0}(\epsilon) = \delta_{0}(\alpha,\beta,\sigma)$). With these choices, \eqref{form202} yields the following inequality for every $\theta \in [0,1]$:
\begin{align*} \mu(B(1) \cap H_{\theta}(K,\delta^{-\sigma + \zeta},[\delta,1])) & \leq \mu(B(1) \cap H_{\theta}(K,\bar{\delta}^{-\sigma + \bar{\zeta}},[\bar{\delta},5]))\\
& \quad + CC_{\gamma}^{2} \cdot \mu\left(B(4) \cap H_{\theta}\left(K,\left(\tfrac{\delta}{\bar{\delta}} \right)^{-\sigma},[4\delta,4\bar{\delta}] \right) \right). \end{align*}
Now, recalling from \eqref{form255} (or to be precise \eqref{form72}) that the left hand side is bounded from below by $\geq \delta^{\eta}$ for all $\theta \in \spt \nu$, we obtain
\begin{align} \delta^{\eta} & \leq \int_{0}^{1} \mu(B(1) \cap H_{\theta}(K,\bar{\delta}^{-\sigma + \bar{\zeta}},[\bar{\delta},5])) \, d\nu(\theta) \notag\\
&\label{form80} \quad + CC_{\gamma}^{2} \cdot \int_{0}^{1} \mu\left(B(4) \cap H_{\theta}\left(K,\left(\tfrac{\delta}{\bar{\delta}} \right)^{-\sigma},[4\delta,4\bar{\delta}] \right) \right) \, d\nu(\theta). \end{align}
Our plan is, next, to use the "inductive" hypothesis \eqref{form201} to show that the second term is $\leq \delta^{\eta}/2$ if $\delta > 0$ and $\eta > 0$ are chosen sufficiently small. This will eventually show that
\begin{equation}\label{form82} \int_{0}^{1} \mu(B(1) \cap H_{\theta}(K,\bar{\delta}^{-\sigma + \bar{\zeta}},[\bar{\delta},5])) \, d\nu(\theta) \geq \tfrac{1}{2} \cdot \delta^{\eta}, \end{equation}
which is already nearly the same as \eqref{form74}. To deal with the term on line \eqref{form80}, we let $\mathcal{K}$ be a minimal cover of $K \cap B(4)$ by discs of radius $4\bar{\delta}$. By the $(\gamma,C_{\gamma})$-regularity of $\mu$, we have $|\mathcal{K}| \leq C_{\gamma} \cdot \bar{\delta}^{-\gamma}$. Then, we decompose
\begin{equation}\label{form81} \eqref{form80} \leq CC_{\gamma}^{2} \cdot \sum_{B \in \mathcal{K}} \int_{0}^{1} \mu\left(B \cap H_{\theta} \left(K,\left(\tfrac{\delta}{\bar{\delta}} \right)^{-\sigma},[4\delta,4\bar{\delta}] \right) \right) \, d\nu(\theta). \end{equation}
To treat the individual terms on the right hand side, we consider the rescaled and renormalised measures $\mu_{B} = (4\bar{\delta})^{-\gamma} \cdot T_{B}\mu$ familiar from Notation \ref{n1}, and we write
\begin{equation}\label{form240} \mu\left(B \cap H_{\theta} \left(K,\left(\tfrac{\delta}{\bar{\delta}} \right)^{-\sigma},[4\delta,4\bar{\delta}] \right) \right) = (4\bar{\delta})^{\gamma} \mu_{B}\left(B(1) \cap H_{\theta}\left(T_{B}(K),\Delta^{-\sigma},[\Delta,1] \right) \right), \end{equation}
for any $\theta \in [0,1]$, where $\Delta := \delta/\bar{\delta}$. This equation is easily deduced from Lemma \ref{lemma7} with $r_{0} = 4\bar{\delta}$. As we explained in Remark \ref{rem1}, see \eqref{form64}, the push-forward $\mu_{B}$ can be expressed as a product of an $(\alpha,C_{\alpha})$-regular measure $\mu_{\alpha,B}$, and a $(\beta,C_{\beta})$-regular measure $\mu_{\beta,B}$. Therefore the hypothesis \eqref{form201} is applicable to the measure $\mu_{B}$, assuming that $\Delta = \delta/\bar{\delta} \leq \Delta_{0}$, as we may (by \eqref{form203}, and taking $\delta > 0$ sufficiently small, depending this time on $\alpha,\beta,\sigma,\Delta_{0}$). The conclusion is that the $\tau$-dimensional Hausdorff contents of the sets
\begin{displaymath} E_{B} := \left\{\theta \in [0,1] : \mu_{B}\left(B(1) \cap H_{\theta} \left(T_{B}(K),\Delta^{-\sigma},[\Delta,1] \right) \right) \geq \Delta^{\eta_{0}} \right\}, \quad B \in \mathcal{K},  \end{displaymath} 
satisfy the uniform upper bounds $\mathcal{H}^{\tau}_{\infty}(E_{B}) \leq \Delta^{\eta_{0}} = (\delta/\bar{\delta})^{\eta_{0}}$. Since $\nu$ is a $(\tau,O\delta^{-\eta})$-Frostman measure, hence absolutely continuous with respect to $\mathcal{H}^{\tau}_{\infty}$ with density bounded by $O\delta^{-\eta}$, it follows that
\begin{displaymath} \nu(E_{B}) \leq O\delta^{-\eta} \cdot \mathcal{H}^{\tau}_{\infty}(E_{B}) \leq O\delta^{-\eta} \cdot \left(\tfrac{\delta}{\bar{\delta}} \right)^{\eta_{0}} \stackrel{\eqref{form203}}{\leq} O\delta^{-\eta} \cdot \delta^{(1 - 1/(1 + \epsilon))\eta_{0}}. \end{displaymath}
Since $1 - 1/(1 + \epsilon) \geq \epsilon/2 \geq \omega(\alpha,\beta,\sigma) > 0$, we may ensure that the right hand side is bounded from above by $\delta^{\eta}/(64CC_{\gamma}^{4})$ by choosing $\eta \leq \epsilon \eta_{0}/8$, and then $\delta > 0$ sufficiently small.  Therefore, splitting the integration to $E_{B}$ and $[0,1] \, \setminus \, E_{B}$, and for $\theta \in E_{B}$ using the uniform upper bound $\mu_{B}(B(1)) \leq C_{\gamma}$, we find that
\begin{displaymath} \int_{0}^{1} \mu_{B}\left(B(1) \cap H_{\theta}\left(T_{B}(K),\left(\tfrac{\delta}{\bar{\delta}} \right)^{-\sigma},[\tfrac{\delta}{\bar{\delta}},1] \right) \right) \, d\nu(\theta) \leq \frac{\delta^{\eta}}{32C_{\gamma}^{3}}, \qquad B \in \mathcal{K}. \end{displaymath}
Plugging this back into \eqref{form81}, and recalling that $|\mathcal{K}_{\gamma}| \leq C_{\gamma} \cdot \bar{\delta}^{-\gamma}$, and \eqref{form240}, we find
\begin{displaymath} \eqref{form81} \leq CC_{\gamma}^{2} \cdot \frac{\delta^{\eta}}{32C_{\gamma}^{3}} \cdot |\mathcal{K}| \cdot (4\bar{\delta})^{\gamma} \leq \frac{\delta^{\eta}}{2}. \end{displaymath} 
As we discussed earlier, this concludes the proof of \eqref{form82}. 

Based on \eqref{form82}, and recalling that $\nu$ is a probability measure, there exists a subset $S \subset S^{1}$ of measure $\nu(S) \geq \omega \cdot \delta^{\eta}$ with the property
\begin{equation}\label{form85} \mu(B(1) \cap H_{\theta}(K,\bar{\delta}^{-\sigma + \bar{\zeta}},[\bar{\delta},5])) \geq \omega \cdot \delta^{\eta}, \qquad \theta \in S. \end{equation}
For tidiness' sake, let us replace "$5$" by "$1$" with the following trick: notice that
\begin{equation}\label{form208}  \mu(B(5) \cap H_{\theta}(K,\bar{\delta}^{-\sigma + 2\bar{\zeta}},[5\bar{\delta},5])) \geq \mu(B(1) \cap H_{\theta}(K,\bar{\delta}^{-\sigma + \bar{\zeta}},[5\bar{\delta},5])), \quad \theta \in [0,1], \end{equation}
using first the inclusion $H_{\theta}(K,M,[r,R]) \subset H_{\theta}(K,\tfrac{M}{5},[5r,R])$ from Lemma \ref{lemma6}, and then assuming that $\bar{\delta} > 0$ is sufficiently small that $\tfrac{1}{5} \cdot \bar{\delta}^{\bar{\zeta}} \geq \bar{\delta}^{2\bar{\zeta}}$. But now, using Lemma \ref{lemma7}, the left hand side of \eqref{form208} can be rewritten as
\begin{equation}\label{form209} 5 \cdot \mu_{B(5)} \mu(B(1) \cap H_{\theta}(T_{B(5)}(K),\bar{\delta}^{-\sigma + 2\bar{\zeta}},[\bar{\delta},1])). \end{equation}
Here $\bar{\mu} := \mu_{B(5)}$ is a product of an $(\alpha,C_{\alpha})$-regular measure and a $(\beta,C_{\beta})$-measure supported on $\bar{K} := T_{B(5)}(K)$, and according to \eqref{form85}, we have
\begin{displaymath} \bar{\mu}(B(1) \cap H_{\theta}(\bar{K},\bar{\delta}^{-\sigma + 2\bar{\zeta}},[\bar{\delta},1])) \geq \tfrac{\omega}{5} \cdot \delta^{\eta}, \qquad \theta \in S. \end{displaymath}
This completes the proof of Lemma \ref{lemma9}. Since the coming arguments would see no difference between $\mu,\bar{\mu}$, or $K,\bar{K}$, or $\bar{\zeta},2\bar{\zeta}$, or $\omega,\omega/5$, we save a little on notation and assume that \eqref{form85} already holds with "$1$" in place of "$5$", that is we assume
\begin{equation}\label{form86} \mu(B(1) \cap H_{\theta}(K,\bar{\delta}^{-\sigma + \bar{\zeta}},[\bar{\delta},1])) \geq \omega \cdot \delta^{\eta}, \qquad \theta \in S, \end{equation}
where $\nu(S) \geq \omega \cdot \delta^{\eta}$. Note that if $\bar{\delta} = \delta$, then (a stronger version of) \eqref{form86} follows immediately from our initial counter assumption \eqref{form255}. From this point on in the proof, the reader may essentially forget about the difference between the scales $\delta$ and $\bar{\delta}$: we will work with the properties \eqref{form73} and \eqref{form86}. The second one is automatically satisfied with $\bar{\delta} = \delta$, but securing the first one necessitated the "change of scale" operation witnessed above.  

\subsection{Removing very high multiplicity subsets} The plan of this section is to use the hypothesis \eqref{form201} to remove from $H_{\theta}(K,\bar{\delta}^{-\sigma + \bar{\zeta}},[\bar{\delta},1])$ points with ("very high") multiplicity $\bar{\delta}^{-\sigma}$, while still retaining the lower bound \eqref{form86} for at least $\tfrac{1}{2}$ of the points in $S$. More precisely, we will prove the following:
\begin{lemma}\label{lemma12} If $\bar{\delta} > 0$ and $\eta < \eta_{0}$ are sufficiently small, then there exists a subset $\bar{S} \subset S$ of measure $\nu(\bar{S}) \geq \tfrac{1}{4} \cdot \nu(S)$ with the property
\begin{displaymath} \mu(B(1) \cap H_{\theta}(K,\bar{\delta}^{-\sigma + \bar{\zeta}},[\bar{\delta},1]) \cap G_{\theta}) \geq \omega \cdot \delta^{\eta}, \qquad \theta \in \bar{S}, \end{displaymath}
where
\begin{displaymath}  G_{\theta} := B(1) \, \setminus \left[H_{\theta}(K,(\bar{\delta}^{1/2})^{-\sigma},[5\bar{\delta}^{1/2},5]) \cup H_{\theta}(K,(\bar{\delta}^{1/2})^{-\sigma},[5\bar{\delta},5\bar{\delta}^{1/2}])\right]. \end{displaymath} \end{lemma}

Proving Lemma \ref{lemma12} is possible, because the "inductive" hypothesis \eqref{form201} tells us that the points with multiplicity $\geq \bar{\delta}^{-\sigma}$ have small $\mu$ measure compared to the lower bound $\mu(B(1) \cap H_{\theta}(K,\bar{\delta}^{-\sigma + \bar{\zeta}},[\bar{\delta},1])) \geq \omega \cdot \delta^{\eta}$ obtained in \eqref{form86}. Here we will of course need that $\eta \ll \eta_{0}$, recall \eqref{form256}. We also recall that $\bar{\delta} \leq \delta^{\omega(\tau)}$, so the assumption "$\bar{\delta} > 0$ is sufficiently small" in Lemma \ref{lemma12} can be arranged by taking $\delta > 0$ sufficiently small.

We begin the proof of Lemma \ref{lemma12}. First, we claim that using \eqref{form201}, and if $\bar{\delta}^{1/2} \leq \Delta_{0}$ is sufficiently small, then
\begin{equation}\label{form205} \mathcal{H}^{\tau}_{\infty}(\{\theta \in [0,1] : \mu(B(1) \cap H_{\theta}(K,(\bar{\delta}^{1/2})^{-\sigma},[5\bar{\delta}^{1/2},5])) \geq 5 \cdot (\bar{\delta}^{1/2})^{\eta_{0}}\}) \leq (\bar{\delta}^{1/2})^{\eta_{0}}. \end{equation} 
This would be a direct consequence of \eqref{form201} without the factors of "$5$". As stated, \eqref{form205} also needs a rescaling argument, just like the one between \eqref{form208}-\eqref{form209}, and literally applying \eqref{form201} to the measure $\mu_{B(5)}$. We omit the details.

As a consequence of \eqref{form205}, and the $(\tau,O\delta^{-\eta})$-Frostman property of $\nu$, the $\nu$-measure of those $\theta \in [0,1]$ appearing in \eqref{form205} is bounded from above by 
\begin{displaymath} O\delta^{-\eta} \cdot (\bar{\delta}^{1/2})^{\eta_{0}} \leq O \cdot \delta^{\omega(\tau)\eta_{0} - \eta}. \end{displaymath}
This bound is smaller than $\nu(S)/10 \geq \omega \cdot \delta^{\eta}/10$ if $\eta < \omega(\tau)\eta_{0}/2$, and $\delta > 0$ is assumed sufficiently small. Therefore, with appropriate choices of $\eta$ and $\delta$, for all "$\theta$" in a subset $S' \subset S$ of $\nu$-measure $\nu(S') \geq \tfrac{1}{2}\nu(S)$, the reverse of \eqref{form205} holds, that is,
\begin{displaymath} \mu(B(1) \cap H_{\theta}(K,(\bar{\delta}^{1/2})^{-\sigma},[5\bar{\delta}^{1/2},5])) < 5 \cdot (\bar{\delta}^{1/2})^{\eta_{0}} \lesssim \delta^{\omega(\tau)\eta_{0}}, \qquad \theta \in S'. \end{displaymath} 
But, according to \eqref{form86}, taking $\eta < \omega(\tau)\eta_{0}$ and $\delta > 0$ small enough, as usual, the upper bound on the right is smaller than $\tfrac{1}{2} \cdot \mu(B(1) \cap H_{\theta}(K,\bar{\delta}^{-\sigma + \bar{\zeta}},[\bar{\delta},1]))$. Therefore, 
\begin{equation}\label{form206} \mu(B(1) \cap H_{\theta}(K,\bar{\delta}^{-\sigma + \bar{\zeta}},[\bar{\delta},1]) \, \setminus \, H_{\theta}(K,(\bar{\delta}^{1/2})^{-\sigma},[5\bar{\delta}^{1/2},5])) \geq \omega \cdot \delta^{\eta}, \qquad \theta \in S'. \end{equation}
%There will be no difference between the sets $S$ and $S'$ below (we will only use that $\nu(S') \geq \omega \cdot \delta^{\eta}$), so we assume that $S = S'$ to simplify notation. 

We then repeat a similar operation to remove, further, the following subsets:
\begin{displaymath} B(1) \cap H_{\theta}(K,(\bar{\delta}^{1/2})^{-\sigma},[5\bar{\delta},5\bar{\delta}^{1/2}]). \end{displaymath}
The argument is similar to the one above: we will first check that
\begin{equation}\label{form243} \nu(\{\theta \in [0,1] : \mu(B(1) \cap H_{\theta}(K,(\bar{\delta}^{1/2})^{-\sigma},[5\bar{\delta},5\bar{\delta}^{1/2}])) \geq \tfrac{1}{2} \cdot \omega \cdot \delta^{\eta}\}) \leq \tfrac{1}{2} \cdot \nu(S'). \end{equation}
Here the constant "$\omega$" is temporarily the same as in \eqref{form206}: the point of \eqref{form243} is that for at least $\tfrac{1}{2}$ of the $\nu$ measure of the directions $\theta \in S'$, at most $\tfrac{1}{2}$ of the set appearing in \eqref{form206} can lie in the set shown in \eqref{form243}. Hence the set in \eqref{form243} can, and will, be removed from the set in \eqref{form206} without reducing its measure too much.

To prove \eqref{form243}, we essentially repeat our treatment for the term \eqref{form80} above: we let $\mathcal{K}$ be a minimal cover of $B(1) \cap K$ by discs of radius  radius $5\bar{\delta}^{1/2}$. By the $(\gamma,C_{\gamma})$-regularity of $\mu$, we then have $|\mathcal{K}| \leq C_{\gamma} \cdot (5\bar{\delta}^{1/2})^{-\gamma}$. We write
\begin{equation}\label{form241} \int_{0}^{1} \mu(B(1) \cap H_{\theta}(K,(\bar{\delta}^{1/2})^{-\sigma},[5\bar{\delta},5\bar{\delta}^{1/2}])) \, d\nu(\theta) \leq \sum_{B \in \mathcal{K}} \int_{0}^{1} \mu(B \cap \ldots) \, d\nu(\theta). \end{equation}
Applying Lemma \ref{lemma7} as in \eqref{form240}, the individual terms on the right can be rewritten as
\begin{equation}\label{form242} \mu(B \cap \ldots) = (5\bar{\delta}^{1/2})^{\gamma} \cdot \mu_{B}(B(1) \cap H_{\theta}(T_{B}(K),(\bar{\delta}^{1/2})^{-\sigma},[\bar{\delta}^{1/2},1])), \quad B \in \mathcal{K}, \end{equation}
where $\mu_{B}$ is the renormalised and rescaled measure associated to the disc $B$, which is again a product of some $(\alpha,C_{\alpha})$-regular measure and some $(\beta,C_{\beta})$-regular measure. Now, assuming that $(\bar{\delta}^{1/2}) \leq \Delta_{0}$, the hypothesis \eqref{form201} applies to the measures $\mu_{B}$ individually, and shows that
\begin{equation}\label{form137} \mathcal{H}^{\tau}_{\infty}(\{\theta \in [0,1] : \mu_{B}(B(1) \cap H_{\theta}(T_{B}(K),(\bar{\delta}^{1/2})^{-\sigma},[\bar{\delta}^{1/2},1])) \geq (\bar{\delta}^{1/2})^{\eta_{0}}\}) \leq (\bar{\delta}^{1/2})^{\eta_{0}} \end{equation} 
for every $B \in \mathcal{K}$. This, and the $(\tau,O\delta^{-\eta})$-Frostman property of $\nu$, implies that if $\eta \ll \omega(\tau)\eta_{0}$ is small enough, then the $\nu$ measure of the points $\theta \in [0,1]$ in \eqref{form137} is no larger than $O\delta^{-\eta} \cdot (\bar{\delta}^{1/2})^{\eta_{0}} \leq \delta^{3\eta}/(2C_{\gamma}^{2})$. Combining this estimate with \eqref{form242}, and noting the uniform upper bound $\mu_{B}(B(1)) \leq C_{\gamma}$, we arrive at the following upper bound for \eqref{form241}:
\begin{displaymath} \eqref{form241} \leq |\mathcal{K}| \cdot (5\bar{\delta}^{1/2})^{\gamma} \cdot [(\bar{\delta}^{1/2})^{\eta_{0}} + C_{\gamma}  \cdot \tfrac{\delta^{3\eta}}{2C_{\gamma}^{2}}] \leq \delta^{3\eta}. \end{displaymath}
By Chebyshev's inequality, this shows that
\begin{displaymath} \nu(\{\theta \in [0,1] : \mu(B(1) \cap H_{\theta}(K,(\bar{\delta}^{1/2})^{-\sigma},[5\bar{\delta},5\bar{\delta}^{1/2}])) \geq \tfrac{1}{2} \cdot \omega \cdot \delta^{\eta}\}) \leq \tfrac{2}{\omega} \cdot \delta^{2\eta}. \end{displaymath} 
Recalling that $\nu(S') \geq \tfrac{1}{2}\nu(S) \geq \omega \cdot \delta^{\eta}$, the right hand side is at most $\tfrac{1}{2} \cdot \nu(S')$ if $\delta > 0$ is chosen sufficiently small. This proves \eqref{form243}.

Now, after discarding at most half of (the $\nu$ measure of) the points in $S'$, we finally arrive at a set $\bar{S} \subset S' \subset S$ with measure $\nu(\bar{S}) \geq \tfrac{1}{4}\nu(S)$ such that a converse of \eqref{form243} holds for all $\theta \in \bar{S}$. In other words, we may upgrade \eqref{form206} to
\begin{equation}\label{form207} \mu(B(1) \cap H_{\theta}(K,\bar{\delta}^{-\sigma + \bar{\zeta}},[\bar{\delta},1]) \cap G_{\theta}) \geq \omega \cdot \delta^{\eta}, \qquad \theta \in \bar{S}, \end{equation}
where
\begin{displaymath} G_{\theta} := B(1) \, \setminus \left[H_{\theta}(K,(\bar{\delta}^{1/2})^{-\sigma},[5\bar{\delta}^{1/2},5]) \cup H_{\theta}(K,(\bar{\delta}^{1/2})^{-\sigma},[5\bar{\delta},5\bar{\delta}^{1/2}])\right]. \end{displaymath}
This completes the proof of Lemma \ref{lemma12}. We now abbreviate
\begin{equation}\label{form210} K_{\theta} := B(1) \cap H_{\theta}(K,\bar{\delta}^{-\sigma + \bar{\zeta}},[\bar{\delta},1]) \cap G_{\theta}, \qquad \theta \in \bar{S}. \end{equation} 
To end this section, we claim that there exist two points $\theta_{0},\theta_{1} \in \bar{S} \subset S$ with separation 
\begin{equation}\label{form128} \bar{\delta}^{(1 + \epsilon)/2} \leq |\theta_{0} - \theta_{1}| \leq \bar{\delta}^{1/2} \end{equation}
with the property $\mu(K_{\theta_{0}} \cap K_{\theta_{1}}) \geq \omega \cdot \delta^{4\eta}$. The scale $\bar{\delta}$ has been chosen so that this is possible; it may be a good idea to recall \eqref{form73} at this point. Since 
\begin{displaymath} \omega \cdot \delta^{\eta} \leq \nu(\bar{S}) = \sum_{I \in \mathcal{D}_{1/2}} \nu(I) \cdot \nu_{I}(\bar{S}), \end{displaymath}
and $\nu$ is a probability measure, there exists an arc $I \in \mathcal{D}_{1/2} = \mathcal{D}_{\bar{\delta}^{1/2}}$ of length $\bar{\delta}^{1/2}$ with the property $\nu_{I}(\bar{S}) \geq \omega \cdot \delta^{\eta}$. Write $S_{I} := I \cap \bar{S}$. Then, $\mu(K_{\theta}) \geq \omega \cdot \delta^{\eta}$ for all $\theta \in S_{I}$ according to \eqref{form207}, and hence
\begin{align*} \omega \cdot \delta^{2\eta} & \leq \int_{S_{I}} \mu(K_{\theta}) \, d\nu_{I}(\theta) \leq \int_{B(1)} \nu_{I}(\{\theta \in S_{I} : x \in K_{\theta}\}) \, d\mu(x)\\
& \leq C_{\gamma} \cdot \left( \int_{B(1)} \nu_{I}(\{(\theta_{0},\theta_{1}) \in S_{I} \times S_{I} : x \in K_{\theta_{0}} \cap K_{\theta_{1}}\}) \, d\mu(x) \right)^{1/2}\\
& = C_{\gamma} \cdot \left( \iint_{S_{I} \times S_{I}} \mu(K_{\theta_{0}} \cap K_{\theta_{1}}) \, d\nu_{I}(\theta_{1}) \, d\nu_{I}(\theta_{0}) \right)^{1/2}. \end{align*} 
On the other hand, we infer from \eqref{form73} and the upper bound $\mu(K_{\theta_{0}} \cap K_{\theta_{1}}) \leq \mu(B(1)) \leq C_{\gamma}$ that
\begin{displaymath} \int_{S_{I}}\int_{S_{I} \cap B(\theta_{0},\bar{\delta}^{(1 + \epsilon)/2})} \mu(K_{\theta_{0}} \cap K_{\theta_{1}}) \, d\nu_{I}(\theta_{1}) \, d\nu_{I}(\theta_{0}) \lesssim C_{\gamma} \cdot \bar{\delta}^{\omega(\alpha,\beta,\sigma,\tau)}. \end{displaymath}  
Therefore, if $\eta > 0$ is sufficiently small, depending only on $\alpha,\beta,\sigma,\tau$, and $\delta > 0$ is small enough, there must exist a pair of points $\theta_{0},\theta_{1} \in S_{I} \subset \bar{S}$ such that $\theta_{1} \notin B(\theta_{0},\bar{\delta}^{(1 + \epsilon)/2})$ and such that $\mu(K_{\theta_{0}} \cap K_{\theta_{1}}) \geq \omega \cdot \delta^{4\eta}$. If one cares to track the constants, then the number "$\omega(\alpha,\beta,\sigma,\tau) > 0$" here is the one coming from \eqref{form73}, and there $\omega(\alpha,\beta,\sigma,\tau) \gtrsim \epsilon \tau^{2}$, so the requirement for $\eta$ mentioned in \eqref{form256} is sufficient. 

We now fix such a pair of points $\theta_{0},\theta_{1} \in \bar{S} \subset S$ for the rest of the proof, and we write
\begin{equation}\label{form211} \bar{K} := K_{\theta_{0}} \cap K_{\theta_{1}}. \end{equation}  
As we will see in the next section, the existence of the set $\bar{K}$ will lead to a contradiction, and this will complete the proof of Proposition \ref{mainProp}.

\subsection{The geometry of the set $\bar{K}$}\label{s:tubeSelection} In this subsection, we will establish a few properties of $\bar{K}$, as in \eqref{form211}, at scales $\bar{\delta}$ and $\bar{\delta}^{1/2}$. We list the desired properties imprecisely in \nref{T1}-\nref{T2}; the rigorous versions of these properties are stated later, in Lemma \ref{lemma10}, once all the relevant notation has been set up in the course of this section. 

We aim to find a tube $\mathbf{T}_{0} \subset \R^{2}$ of the form $\mathbf{T}_{0} = \pi_{\theta_{0}}^{-1}(I)$ with $|I| \sim \bar{\delta}^{1/2}$, such that
\begin{itemize}
\item[(T1) \phantomsection\label{T1}] $N_{\bar{\delta}^{1/2}}(\bar{K} \cap \mathbf{T}_{0}) =: \M \gtrapprox (\bar{\delta}^{1/2})^{-\sigma}$, and $N_{\bar{\delta}}(\bar{K} \cap \mathbf{T}_{0}) \approx \M \cdot (\bar{\delta}^{1/2})^{-\gamma}$,
\item[(T2) \phantomsection\label{T2}] $N_{\bar{\delta}}(\pi_{\theta_{j}}(\bar{K} \cap \mathbf{T}_{0})) \lessapprox (\bar{\delta}^{1/2})^{\sigma - \gamma}$ for both $j \in \{0,1\}$.
\end{itemize}
The set $\bar{K} \cap \mathbf{T}_{0}$ (after some extra pruning) will eventually play the role of the set $A' \times B'$ discussed in the proof outline, Section \ref{s:outline}, see also Remark \ref{rem:outline}.

Finding a tube $\mathbf{T}_{0}$ with the properties \nref{T1}-\nref{T2} is based on the fact that
\begin{equation}\label{form92} \bar{K} \subset B(1) \cap H_{\theta_{j}}(K,\bar{\delta}^{-\sigma + \bar{\zeta}},[\bar{\delta},1]) \cap G_{\theta_{j}}, \qquad j \in \{0,1\}, \end{equation}
recall \eqref{form211}. We start by claiming the following:
\begin{equation}\label{form90} N_{\bar{\delta}}(\pi_{\theta_{j}}(\bar{K})) \lesssim C_{\gamma} \cdot \bar{\delta}^{\sigma - \gamma - \bar{\zeta}}, \qquad j \in \{0,1\}. \end{equation}
To see \eqref{form90}, fix $\theta \in \{\theta_{0},\theta_{1}\}$, and let $\mathcal{T}_{\theta}$ be a minimal cover of $\bar{K}$ by tubes of the form $T = \pi_{\theta}^{-1}\{I\}$, where $I \in \mathcal{D}_{\bar{\delta}}(\R)$. (We will write "$T$" for $\bar{\delta}$-tubes and "$\mathbf{T}$" for $\bar{\delta}^{1/2}$-tubes.) Then, each $T \in \mathcal{T}_{\theta}$ contains a point $x_{T} \in \bar{K} \subset H_{\theta}(K,\bar{\delta}^{-\sigma + \bar{\zeta}},[\bar{\delta},1])$, hence $B(x,1) \cap \pi_{\theta}^{-1}\{\pi_{\theta}(x_{T})\} \subset B(2) \cap T$, and
\begin{displaymath} N_{\bar{\delta}}(K_{\bar{\delta}} \cap B(2) \cap T) \geq \m_{K,\theta}(x_{T} \mid [\bar{\delta},1]) \geq \bar{\delta}^{-\sigma + \bar{\zeta}}. \end{displaymath} 
This implies that
\begin{displaymath} |\mathcal{T}_{\theta}| \cdot \bar{\delta}^{-\sigma + \bar{\zeta}} \lesssim N_{\bar{\delta}}(K_{\bar{\delta}} \cap B(2)) \lesssim C_{\gamma} \cdot \bar{\delta}^{-\gamma}, \end{displaymath}
and \eqref{form90} follows by rearranging.

We next consider the projections of $\bar{K}$ at scale $\bar{\delta}^{1/2}$, but we first do an initial reduction. Let $\mathcal{K}_{\bar{\delta}^{1/2}}$ be a minimal cover of $\bar{K}$ by discs of radius $\bar{\delta}^{1/2}$, so in particular $\bar{K} \cap B \neq \emptyset$ for all $B \in \mathcal{K}_{\bar{\delta}^{1/2}}$. Since $\bar{K} \subset K \cap B(1)$, and $\mu(\bar{K}) \geq \omega \cdot \delta^{4\eta}$, we have 
\begin{equation}\label{form97} \omega \cdot \delta^{4\eta} \cdot (\bar{\delta}^{1/2})^{-\gamma} \leq |\mathcal{K}_{\bar{\delta}^{1/2}}| \leq N_{\bar{\delta}^{1/2}}(K \cap B(1)) \leq C_{\gamma} \cdot (\bar{\delta}^{1/2})^{-\gamma}. \end{equation}
A disc $B \in \mathcal{K}_{\bar{\delta}^{1/2}}$ is called \emph{heavy} if 
\begin{equation}\label{form215} \mu(B \cap \bar{K}) \geq (\mu(\bar{K})/2C_{\gamma}) \cdot (\bar{\delta}^{1/2})^{\gamma} \geq \omega \cdot \delta^{4\eta} \cdot (\bar{\delta}^{1/2})^{\gamma}. \end{equation}
Then, the total $\mu$ measure of the \emph{light} (that is, non-heavy) discs if bounded from above by $|\mathcal{K}_{\bar{\delta}^{1/2}}| \cdot (\mu(\bar{K})/2C_{\gamma}) \cdot (\bar{\delta}^{1/2})^{\gamma} \leq \mu(\bar{K})/2$. Therefore, if we replace $\bar{K}$ by the intersection
\begin{equation}\label{form111} \bar{K} \cap \bigcup_{B \in \mathcal{K}_{\bar{\delta}^{1/2}} \mathrm{\,heavy}} B, \end{equation}
then $\mu(\bar{K}) \geq (\omega/2) \cdot \delta^{4\eta}$, and the key property \eqref{form92} of $\bar{K}$ remains valid (it is worth emphasising here that the high-multiplicity set in \eqref{form92} is defined relative to $K$, not $\bar{K}$). These are all the properties we will need in the sequel, so, without loss of generality, we may assume that all the discs in $\mathcal{K}_{\bar{\delta}^{1/2}}$ are heavy.

Now, in order to consider the $\pi_{\theta_{0}}$ and $\pi_{\theta_{1}}$ projections of $\bar{K}$ at scale $\bar{\delta}^{1/2}$, let $\mathcal{T}_{\bar{\delta}^{1/2}}$ be a minimal cover of $\bar{K}$ by tubes $\mathbf{T} = \pi_{\theta_{0}}^{-1}(I)$ with $I \in \mathcal{D}_{\bar{\delta}^{1/2}}(\R)$. It will be good to keep in mind that $|\theta_{0} - \theta_{1}| \leq \bar{\delta}^{1/2}$ by \eqref{form128}, so the $\pi_{\theta_{0}}$ and $\pi_{\theta_{1}}$ projections of $\bar{K}$ are virtually indistinguishable at scale $\bar{\delta}^{1/2}$: for example 
\begin{displaymath} N_{\bar{\delta}^{1/2}}(\pi_{\theta_{0}}(\bar{K})) \sim N_{\bar{\delta}^{1/2}}(\pi_{\theta_{1}}(\bar{K})). \end{displaymath}
Also, even though the tubes in $\mathcal{T}_{\bar{\delta}^{1/2}}$ were defined via the projection $\pi_{\theta_{0}}$, they still satisfy the following property for both $j \in \{0,1\}$: the sets $\pi_{\theta_{j}}(B(1) \cap \mathbf{T})$ have bounded overlap as $\mathbf{T} \in \mathcal{T}_{\bar{\delta}^{1/2}}$ varies. 

We now claim the following upper bound, assuming that $\delta > 0$ and $\eta > 0$ are sufficiently small:
\begin{equation}\label{form99} N_{\bar{\delta}^{1/2}}(\pi_{\theta_{0}}(\bar{K})) \sim |\mathcal{T}_{\bar{\delta}^{1/2}}| \lesssim C_{\gamma} \cdot (\bar{\delta}^{1/2})^{\sigma - \gamma - 4\bar{\zeta}}.\end{equation} 
The proof goes as follows. Every tube $\mathbf{T} \in \mathcal{T}_{\bar{\delta}^{1/2}}$ meets at least one (heavy!) ball $B_{\mathbf{T}} \in \mathcal{K}_{\bar{\delta}^{1/2}}$. We will shortly see that
\begin{equation}\label{form98} N_{\bar{\delta}}(\pi_{\theta_{j}}(\bar{K} \cap B)) \geq (\bar{\delta}^{1/2})^{\sigma - \gamma + 2\bar{\zeta}}, \qquad j \in \{0,1\}, \, B \in \mathcal{K}_{\bar{\delta}^{1/2}}, \end{equation}
so in particular this holds with $B = B_{\mathbf{T}}$. Since the sets $\pi_{\theta_{0}}(\bar{K} \cap B_{\mathbf{T}})$ have bounded overlap as $\mathbf{T}$ varies in $\mathcal{T}_{\bar{\delta}^{1/2}}$, it follows from \eqref{form90} and \eqref{form98} that
\begin{displaymath} C_{\gamma} \cdot \bar{\delta}^{\sigma - \gamma - \bar{\zeta}} \gtrsim N_{\bar{\delta}}(\pi_{\theta_{0}}(\bar{K})) \gtrsim |\mathcal{T}_{\bar{\delta}^{1/2}}| \cdot (\bar{\delta}^{1/2})^{\sigma - \gamma + 2\bar{\zeta}}, \end{displaymath}
and then \eqref{form99} follows by rearranging terms. 

Let us then prove \eqref{form98}. Fix $B \in \mathcal{K}_{\bar{\delta}^{1/2}}$, and keep in mind that $\mu(B \cap \bar{K}) \geq \omega \cdot \delta^{4\eta} \cdot \bar{\delta}^{\gamma/2}$, since all the balls in $\mathcal{K}_{\bar{\delta}^{1/2}}$ are heavy. Let us first check that \eqref{form98} follows if we manage to show the next claim: if $T_{\bar{\delta}} = \pi_{\theta_{j}}^{-1}(I)$ is an arbitrary tube of width $|I| = \bar{\delta}$, then
\begin{equation}\label{form100} N_{\bar{\delta}}(B \cap \bar{K} \cap T_{\bar{\delta}}) \leq (\bar{\delta}^{1/2})^{-\sigma - \bar{\zeta}}, \qquad B \in \mathcal{K}_{\bar{\delta}^{1/2}}. \end{equation}
Indeed, since $\mu(B \cap \bar{K}) \geq \omega \cdot \delta^{4\eta} \cdot (\bar{\delta}^{1/2})^{\gamma}$ for all (heavy) discs $B \in \mathcal{K}_{\bar{\delta}^{1/2}}$, we have 
\begin{displaymath} N_{\bar{\delta}}(B \cap \bar{K}) \geq \omega \cdot \delta^{4\eta} \cdot (\bar{\delta}^{1/2})^{-\gamma} \end{displaymath}
by the $(\gamma,C_{\gamma})$-regularity of $\mu$. Combining this lower bound with \eqref{form100} implies that it takes $\geq \omega \cdot \delta^{4\eta} \cdot (\bar{\delta}^{1/2})^{\sigma - \gamma + \bar{\zeta}}$ tubes of the form $T_{\bar{\delta}} = \pi_{\theta_{j}}^{-1}(I)$, with $|I| = \bar{\delta}$, to cover $B \cap \bar{K}$. This implies \eqref{form98} for $\delta,\eta > 0$ small enough that $\omega \cdot \delta^{4\eta} \geq \bar{\delta}^{\bar{\zeta}/2}$. 

It remains to establish \eqref{form100}. Fix a tube $T_{\bar{\delta}} := \pi_{\theta_{j}}^{-1}(I)$ with $|I| = \bar{\delta}$, and assume to the contrary that $B \in \mathcal{K}_{\bar{\delta}^{1/2}}$ is a disc with
\begin{equation}\label{form138} N_{\bar{\delta}}(B \cap \bar{K} \cap T_{\bar{\delta}}) \geq (\bar{\delta}^{1/2})^{-\sigma - \bar{\zeta}}. \end{equation} 
Then, if $x_{0} \in B \cap \bar{K} \cap T_{\bar{\delta}}$ is arbitrary, note that $B \subset B(x_{0},5\bar{\delta}^{1/2})$. From this, combined with \eqref{form138}, and assuming $\delta > 0$ sufficiently small in terms of $\bar{\zeta}$, it follows easily that
\begin{equation}\label{form139} \m_{K,\theta_{j}}(x_{0} \mid [5\bar{\delta},5\bar{\delta}^{1/2}]) := N_{5\bar{\delta}}(B(x_{0},5\bar{\delta}^{1/2}) \cap K_{5\bar{\delta}} \cap \pi_{\theta_{j}}^{-1}\{\pi_{\theta_{j}}(x_{0})\}) \geq (\bar{\delta}^{1/2})^{-\sigma},  \end{equation}
see Figure \ref{fig1} for further details. In other words $x_{0} \in \bar{K} \cap H_{\theta_{j}}(K,(\bar{\delta}^{1/2})^{-\sigma},[5\bar{\delta},5\bar{\delta}^{1/2}])$. However, by the definition in \eqref{form211}, the set $\bar{K}$ is a subset of $K_{\theta_{j}}$, and this $K_{\theta_{j}}$, by its definition in \eqref{form210}, contains no points of $H_{\theta_{j}}(K,(\bar{\delta}^{1/2})^{-\sigma},[5\bar{\delta},5\bar{\delta}^{1/2}])$. Therefore \eqref{form100} holds, and this completes the proof of \eqref{form99}. 

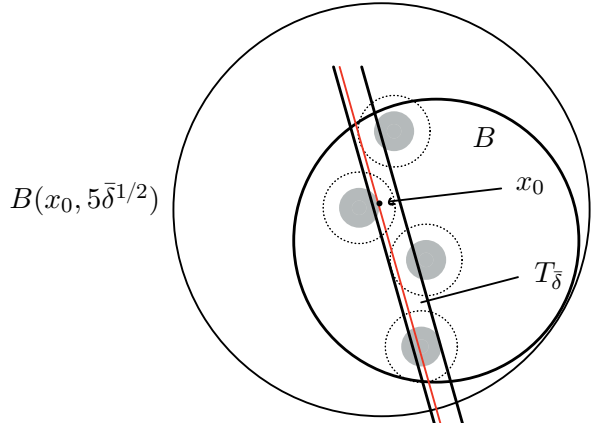
\begin{figure}
\begin{center}
\begin{overpic}[scale = 1.1]{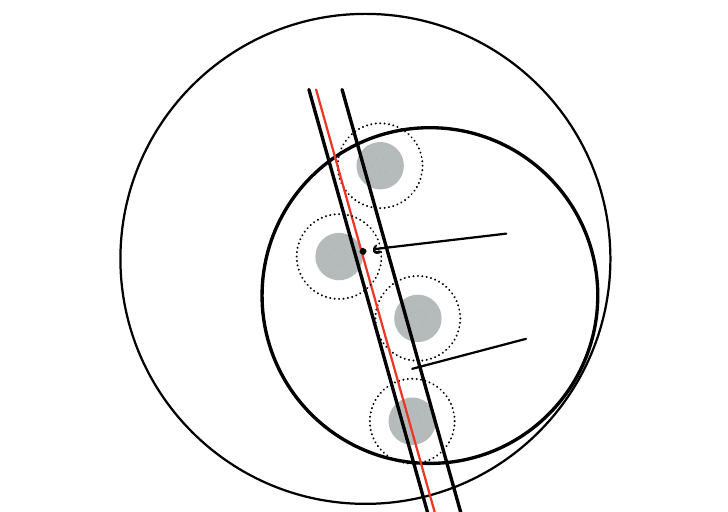}
\put(65,45){$B$}
\put(-10,35){$B(x_{0},5\bar{\delta}^{1/2})$}
\put(75,23){$T_{\bar{\delta}}$}
\put(72,38){$x_{0}$}
\end{overpic}
\caption{One of the heavy discs $B \in \mathcal{K}_{\bar{\delta}^{1/2}}$ intersected with a tube $T_{\bar{\delta}} = \pi_{\theta_{j}}^{-1}(I)$, and a point $x_{0} \in B \cap \bar{K} \cap T_{\bar{\delta}}$. The grey discs form a minimal $\bar{\delta}$-cover for $B \cap \bar{K} \cap T_{\bar{\delta}}$, as in \eqref{form138}, and their $5$-times enlargements, shown in dotted lines, all intersect the red line $\pi_{\theta_{j}}^{-1}\{\pi_{\theta_{j}}(x_{0})\}$. This gives \eqref{form139}.}\label{fig1}
\end{center}
\end{figure}

We have now shown, in \eqref{form99}, that $\bar{K}$ can be covered by $\lesssim C_{\gamma} \cdot (\bar{\delta}^{1/2})^{\sigma - \gamma - 4\bar{\zeta}}$ tubes of the form $\mathbf{T} = \pi_{\theta_{0}}^{-1}(I)$, $I \in \mathcal{D}_{\bar{\delta}^{1/2}}(\R)$, whose collection we denoted $\mathcal{T}_{\bar{\delta}^{1/2}}$. This implies that the average tube in $\mathcal{T}_{\bar{\delta}^{1/2}}$ meets 
\begin{displaymath} \gtrsim C_{\gamma}^{-1} \cdot |\mathcal{K}_{\bar{\delta}^{1/2}}| \cdot (\bar{\delta}^{1/2})^{\gamma - \sigma + 4\bar{\zeta}} \stackrel{\eqref{form97}}{\geq} \omega \cdot \delta^{4\eta} \cdot (\bar{\delta}^{1/2})^{-\sigma + 4\bar{\zeta}} \end{displaymath}
discs in $\mathcal{K}_{\bar{\delta}^{1/2}}$. For technical convenience, we wish to arrange that the statement above holds for every tube in $\mathcal{T}_{\bar{\delta}^{1/2}}$, and this can be accomplished by another pruning argument, as follows. We say that a tube $\mathbf{T} \in \mathcal{T}_{\bar{\delta}^{1/2}}$ is \emph{heavy} if
\begin{displaymath} |\{B \in \mathcal{K}_{\bar{\delta}^{1/2}} : B \cap \mathbf{T} \neq \emptyset\}| \geq \omega \cdot \delta^{4\eta} \cdot (\bar{\delta}^{1/2})^{-\sigma + 4\bar{\zeta}}, \end{displaymath}
for a suitable constant $\omega > 0$. Since $|\mathcal{K}_{\bar{\delta}^{1/2}}| \geq \omega \cdot  \delta^{4\eta} \cdot (\bar{\delta}^{1/2})^{-\gamma}$ by \eqref{form97}, and $|\mathcal{T}_{\bar{\delta}^{1/2}}| \lesssim C_{\gamma} \cdot (\bar{\delta}^{1/2})^{\sigma - \gamma - 4\bar{\zeta}}$ as we just argued, the union of the \emph{light} (non-heavy) tubes in $\mathcal{T}_{\bar{\delta}^{1/2}}$ intersects at most half of the discs in $\mathcal{K}_{\bar{\delta}^{1/2}}$, assuming that the "$\omega$" constant in the definition of heaviness above is chosen appropriately. We now redefine $\mathcal{K}_{\bar{\delta}^{1/2}}$ to be those discs in (former) $\mathcal{K}_{\bar{\delta}^{1/2}}$ which intersect a heavy tube, and we redefine $\bar{K}$ to be the part of $\bar{K}$ covered by the union of the (remaining) discs in $\mathcal{K}_{\bar{\delta}^{1/2}}$. We also restrict $\mathcal{T}_{\bar{\delta}^{1/2}}$ to the heavy tubes. With this new notation, and taking $\delta,\eta > 0$ small enough so that $\omega \cdot \delta^{4\eta} \geq (\bar{\delta}^{1/2})^{\bar{\zeta}}$, 
\begin{equation}\label{form104} |\{B \in \mathcal{K}_{\bar{\delta}^{1/2}} : B \cap \mathbf{T} \neq \emptyset\}| \geq \omega \cdot \delta^{4\eta} \cdot (\bar{\delta}^{1/2})^{-\sigma + 4\bar{\zeta}} \geq (\bar{\delta}^{1/2})^{-\sigma + 5\bar{\zeta}}, \qquad \mathbf{T} \in \mathcal{T}_{\bar{\delta}^{1/2}}. \end{equation}
Moreover, $\mathcal{K}_{\bar{\delta}^{1/2}}$ continues to satisfy the estimates from \eqref{form97}, namely $|\mathcal{K}_{\bar{\delta}^{1/2}}| \approx (\bar{\delta}^{1/2})^{-\gamma}$. The tubes in $\mathcal{T}_{\bar{\delta}^{1/2}}$ need not quite cover $\bar{K}$ any longer. This is because some disc in $\mathcal{K}_{\bar{\delta}^{1/2}}$ might have been initially half-half covered by a light tube and a heavy tube; then the ball was selected to the new $\mathcal{K}_{\bar{\delta}^{1/2}}$, but is not covered by the heavy tubes it touches. To fix this, we inflate the tubes in $\mathcal{T}_{\bar{\delta}^{1/2}}$ by a factor of $3$ without changing notation: then $\mathcal{T}_{\bar{\delta}^{1/2}}$ consists of tubes of width $3\bar{\delta}^{1/2}$, and
\begin{equation}\label{form213} \bar{K} = \bigcup_{B \in \mathcal{K}_{\bar{\delta}^{1/2}}} \bar{K} \cap B \subset \bigcup_{\mathbf{T} \in \mathcal{T}_{\bar{\delta}^{1/2}}} \mathbf{T}. \end{equation}
We also write
\begin{equation}\label{form110} \bar{K}_{\mathbf{T}} := \mathop{\bigcup_{B \in \mathcal{K}_{\bar{\delta}^{1/2}}}}_{B \subset \mathbf{T}} \bar{K} \cap B \subset \mathbf{T}, \qquad \mathbf{T} \in \mathcal{T}_{\bar{\delta}^{1/2}}. \end{equation}
Since the tubes in $\mathcal{T}_{\bar{\delta}^{1/2}}$ are heavy (and since "$\mathbf{T}$" now already refers to the $3$-times inflated tubes), the cardinality of this union is bounded from below by \eqref{form104}, for every $\mathbf{T} \in \mathcal{T}_{\bar{\delta}^{1/2}}$.

We would next like to show that for "most" of the tubes $\mathbf{T} \in \mathcal{T}_{\bar{\delta}^{1/2}}$, the projections $\pi_{\theta_{j}}(\bar{K}_{\mathbf{T}})$, for both $j \in \{0,1\}$, are fairly small at scale $\bar{\delta}$, say
\begin{equation}\label{form107} N_{\bar{\delta}}(\pi_{\theta_{j}}(\bar{K}_{\mathbf{T}})) \leq (\bar{\delta}^{1/2})^{\sigma - \gamma - 6\bar{\zeta}}, \qquad j \in \{0,1\}. \end{equation} 
The only clue available is \eqref{form90}, which controls the $\pi_{\theta_{j}}$-projections of the whole set $\bar{K}$. For $j \in \{0,1\}$ fixed, this can be easily used to show that there are only few tubes $\mathbf{T} \in \mathcal{T}_{\bar{\delta}^{1/2}}$ such that \eqref{form107} fails. Namely, if $\mathcal{T}_{\bar{\delta}^{1/2},j} \subset \mathcal{T}_{\bar{\delta}^{1/2}}$ is the sub-family of "bad" tubes for which \eqref{form107} fails, then it follows from \eqref{form90}, and the bounded overlap of the projections $\pi_{\theta_{j}}(\bar{K}_{\mathbf{T}})$, $\mathbf{T} \in \mathcal{T}_{\bar{\delta}^{1/2}}$, that
\begin{displaymath} |\mathcal{T}_{\bar{\delta}^{1/2},j}| \cdot (\bar{\delta}^{1/2})^{\sigma - \gamma - 6\bar{\zeta}} \lesssim N_{\bar{\delta}}(\pi_{\theta_{j}}(\bar{K})) \lesssim C_{\gamma} \cdot \bar{\delta}^{\sigma - \gamma - \bar{\zeta}}, \qquad j \in \{0,1\}, \end{displaymath}
and hence
\begin{equation}\label{form108} |\mathcal{T}_{\bar{\delta}^{1/2},j}| \leq (\bar{\delta}^{1/2})^{\sigma - \gamma + 3\bar{\zeta}}, \end{equation}
assuming that $\bar{\delta}$ is so small that the constants (implicit and $C_{\gamma}$) are bounded from above by $(\bar{\delta}^{1/2})^{-\bar{\zeta}}$. So, now we know that \eqref{form107} can only fail for few tubes in $\mathcal{T}_{\bar{\delta}^{1/2}}$. What we really wanted was, instead, that \eqref{form107} holds for "most" tubes in $\mathcal{T}_{\bar{\delta}^{1/2}}$. To deduce the latter statement from the former, we need to show that the families $\mathcal{T}_{\bar{\delta}^{1/2},j}$ of "bad" tubes above only constitute a small fraction of all the tubes in $\mathcal{T}_{\bar{\delta}^{1/2}}$. According to \eqref{form108}, this follows if we manage to show that
\begin{equation}\label{form91} |\mathcal{T}_{\bar{\delta}^{1/2}}| \geq (\bar{\delta}^{1/2})^{\sigma - \gamma + 2\bar{\zeta}}. \end{equation} 
The proof of \eqref{form91} is extremely similar to the proof of the lower bound in \eqref{form98}. Instead of showing \eqref{form91} directly, we prove that 
\begin{equation}\label{form212} N_{\bar{\delta}^{1/2}}(\bar{K} \cap \mathbf{T}) \leq (\bar{\delta}^{1/2})^{-\sigma - \bar{\zeta}}, \qquad \mathbf{T} \in \mathcal{T}_{\bar{\delta}^{1/2}}, \end{equation}
which is roughly a reverse of \eqref{form104}. Once \eqref{form212} has been established, \eqref{form91} follows (assuming that $\delta,\eta > 0$ are small enough, as usual), since the union of the tubes in $\mathcal{T}_{\bar{\delta}^{1/2}}$ covers all the discs in $\mathcal{K}_{\bar{\delta}^{1/2}}$ by \eqref{form213}, and $|\mathcal{K}_{\bar{\delta}^{1/2}}| \geq \omega \cdot \delta^{4\eta} \cdot (\bar{\delta}^{1/2})^{-\gamma}$ by \eqref{form97}.

To prove \eqref{form212}, fix $x_{0} \in \bar{K} \cap \mathbf{T}$ arbitrary, and note that
\begin{equation}\label{form214} \m_{K,\theta_{0}}(x_{0} \mid [5\bar{\delta}^{1/2},5]) = N_{5\bar{\delta}^{1/2}}(B(x_{0},5) \cap K_{5\bar{\delta}^{1/2}} \cap \pi_{\theta_{0}}^{-1}\{\pi_{\theta_{0}}(x_{0})\}) \gtrsim N_{\bar{\delta}^{1/2}}(\bar{K} \cap \mathbf{T}), \end{equation}
using that $\bar{K} \subset K \subset B(1)$, and the width of $\mathbf{T}$ is $3\bar{\delta}^{1/2}$. The geometry of the inequality in \eqref{form214} is similar to the one depicted in Figure \ref{fig1}, the main difference being that the scales "$\bar{\delta}$" and "$\bar{\delta}^{1/2}$" are replaced by "$\bar{\delta}^{1/2}$" and "$1$". Since $\bar{K} \subset K_{\theta_{0}} \subset G_{\theta_{0}}$ by \eqref{form92}, we have 
\begin{displaymath} x_{0} \notin H_{\theta_{0}}(K,(\bar{\delta}^{1/2})^{-\sigma},[5\bar{\delta}^{1/2},5]), \end{displaymath}
and hence the left hand side in \eqref{form214} is no larger than $(\bar{\delta}^{1/2})^{-\sigma}$. For $\delta > 0$ sufficiently small, this yields \eqref{form212}. 

Combining \eqref{form108}-\eqref{form91}, we see that if $\bar{\delta} > 0$ is small enough, then only a small fraction of the tubes in $\mathcal{T}_{\bar{\delta}^{1/2}}$ lies in $\mathcal{T}_{\bar{\delta}^{1/2},0} \cup \mathcal{T}_{\bar{\delta}^{1/2},1}$. In particular, we may find a tube $\mathbf{T}_{0} \in \mathcal{T}_{\bar{\delta}^{1/2}} \, \setminus \, [\mathcal{T}_{\bar{\delta}^{1/2},0} \cup \mathcal{T}_{\bar{\delta}^{1/2},1}]$. We gather the relevant properties of $\mathbf{T}_{0}$ for future reference:

\begin{lemma}\label{lemma10} There exists a tube $\mathbf{T}_{0} \subset \R^{2}$ of the form $\mathbf{T}_{0} = \pi_{\theta_{0}}^{-1}(I_{0})$, where $I_{0} \subset \R$ is an interval of length $3\bar{\delta}^{1/2}$, such that $\mathbf{T}_{0}$ has the following properties:
\begin{itemize}
\item[(K1) \phantomsection \label{K1}] $|\{B \in \mathcal{K}_{\bar{\delta}^{1/2}} : B \subset \mathbf{T}_{0}\}| \geq (\bar{\delta}^{1/2})^{-\sigma + 5\bar{\zeta}}$ by \eqref{form104}, and all the balls $B \in \mathcal{K}_{\bar{\delta}^{1/2}}$ here are heavy, that is, $\mu(B \cap \bar{K}) \geq \omega \cdot \delta^{4\eta} \cdot \bar{\delta}^{\gamma/2}$, recall \eqref{form215}. 
\item[(K2) \phantomsection \label{K2}] The set $\bar{K}_{0} := \bar{K}_{\mathbf{T}_{0}}$ is a subset of $G_{\theta_{0}} \cap G_{\theta_{1}}$, and has small $\pi_{\theta_{j}}$-projections at scale $\bar{\delta}$ in the sense that $N_{\bar{\delta}}(\pi_{\theta_{j}}(\bar{K}_{0})) \leq (\bar{\delta}^{1/2})^{\sigma - \gamma - 6\bar{\zeta}}$ for both $j \in \{0,1\}$, see \eqref{form107}.
\end{itemize}
\end{lemma}

The properties \nref{K1}-\nref{K2} are, finally, the precise versions of \nref{T1}-\nref{T2}.

\subsection{Statement of Shmerkin's inverse theorem}

We pause the main line of the proof for a moment to introduce Shmerkin's inverse theorem, and some associated notation. 

\begin{definition}[$\delta$-measures and $L^{2}$-norms]\label{def:deltaMeasures} Let $\delta \in 2^{-\N}$ be a dyadic rational. Then, any probability measure supported on the discrete set $\delta \cdot \Z \cap [-1,1]$ is called a \emph{$\delta$-measure}. The $L^{2}$-norm of a $\delta$-measure $\mu$ is defined by
\begin{displaymath} \|\mu\|_{L^{2}} := \left(\sum_{z \in \delta \cdot \Z} \mu(\{z\})^{2} \right)^{1/2}. \end{displaymath}
\end{definition}

\begin{thm}[Shmerkin]\label{shmerkin} Given $\epsilon > 0$ and $m_{0} \in \N$, there are $\kappa = \kappa(\epsilon,m_{0}) > 0$ and $m \geq m_{0}$ such that the following holds for all large enough $N \in \N$. Let $\delta = (2^{-m})^{N}$, and let $\eta_{1},\eta_{2}$ be $\delta$-measures such that
\begin{equation}\label{inverseHyp} \|\eta_{1} \ast \eta_{2}\|_{L^{2}} \geq \delta^{\kappa}\|\eta_{1}\|_{L^{2}}. \end{equation}
Then, there exist ($\delta$-separated) sets $U \subset \spt \eta_{1}$ and $V \subset \spt \eta_{2}$ such that
\begin{displaymath} \|\eta_{1}|_{U}\|_{L^{2}} \geq \delta^{\epsilon}\|\eta_{1}\|_{L^{2}} \quad \text{and} \quad \eta_{2}(V) \geq \delta^{\epsilon}, \end{displaymath} 
and the following properties hold:
\begin{itemize}
\item[\textup{(A)}] there is a sequence $(R_{s}^{1})_{s = 0}^{N - 1} \subset \{1,\ldots,2^{m}\}^{N - 1}$, such that 
\begin{displaymath} N(U \cap I,2^{-(s + 1)m}) = R_{s}^{1} \end{displaymath}
for all dyadic intervals $I$ of length $2^{-ms}$ intersecting $U$,
\item[\textup{(B)}] there is a sequence $(R_{s}^{2})_{s = 0}^{N - 1} \subset \{1,\ldots,2^{m}\}^{N - 1}$, such that 
\begin{displaymath} N(V \cap I,2^{-(s + 1)m}) = R_{s}^{2} \end{displaymath}
for all dyadic intervals $I$ of length $2^{-ms}$ intersecting $V$.
\end{itemize}
For each $s \in \{0,\ldots,N - 1\}$, either $R_{s}^{2} = 1$ or $R_{s}^{1} \geq 2^{(1 - \epsilon)m}$, and the set $\calS = \{s : R_{s}^{1} \geq 2^{(1 - \epsilon)m}\}$ satisfies
\begin{equation}\label{form116} m|\calS| \geq \log \|\eta_{2}\|_{L^{2}}^{-2} - \epsilon \log_{2} \delta. \end{equation}
\end{thm}
\begin{remark} The numbers $R_{s}^{j}$ are the \emph{branching numbers} of the measures $\eta_{j}$; this terminology is heuristically useful, but imprecise, since the numbers $R_{s}^{j}$ may depend on the specific choices of $U$ and $V$. If this imprecision is tolerated, and if $R_{s}^{1} \geq 2^{\alpha m}$ for all $s \in \{0,\ldots,N - 1\}$, then we might say that "$\eta_{1}$ has $\alpha$-dimensional branching at all scales". This terminology was used in the proof outline, Section \ref{s:outline}. \end{remark}

\subsection{Deriving a product structure from $\bar{K}_{0}$} We the return to the main line of the argument. Based on the properties \nref{K1}-\nref{K2} of the set $\bar{K}_{0} = \bar{K}_{\mathbf{T}_{0}}$, we will construct a pair of $\bar{\delta}$-measures $\eta_{1},\eta_{2}$ which eventually contradict the statement of Shmerkin's inverse theorem. This contradiction will show that our counter assumption \eqref{form200} must be false, and the proof of Proposition \ref{mainProp} will be completed.

To simplify notation a little, we assume without loss of generality that $\mathbf{T}_{0}$ is the tube 
\begin{displaymath} \mathbf{T}_{0} = \pi_{\theta_{0}}^{-1}(I_{0}), \quad \text{where} \quad I_{0} := [0,3\bar{\delta}^{1/2}]. \end{displaymath}
We will also assume that $3\bar{\delta}^{1/2}$ is a dyadic rational. Let $\mathcal{T}_{\bar{\delta}}$ be a minimal cover of $\bar{K}_{0}$ with tubes of the form $\pi_{\theta_{0}}^{-1}(I)$, where $I \in \mathcal{D}_{\bar{\delta}}(\R)$; a more accurate notation would be $\mathcal{T}_{\theta_{0},\bar{\delta}}$, but the choice between $\theta_{0}$ and $\theta_{1}$ at this point is completely arbitrary. Then the tubes in $\mathcal{T}_{\bar{\delta}}$ are all contained in $\mathbf{T}_{0}$, or in other words $I \subset I_{0}$, since $\bar{K}_{0} \subset \mathbf{T}_{0}$, recall \eqref{form110}.  Moreover,
\begin{equation}\label{form109} |\mathcal{T}_{\bar{\delta}}| \leq (\bar{\delta}^{1/2})^{\sigma - \gamma - 6\bar{\zeta}} \end{equation}
according to property \nref{K2}. We also let $\mathcal{K}_{\bar{\delta}^{1/2}}(\mathbf{T}_{0}) := \{B \in \mathcal{K}_{\bar{\delta}^{1/2}} : B \subset \mathbf{T}_{0}\}$, and we write
\begin{equation}\label{form118} \M := |\mathcal{K}_{\bar{\delta}^{1/2}}(\mathbf{T}_{0})| \stackrel{\mathrm{\nref{K1}}}{\geq} (\bar{\delta}^{1/2})^{-\sigma + 5\bar{\zeta}}. \end{equation}
For purposes in the distant future, we next want to remove from $\mathcal{T}_{\bar{\delta}}$ a few tubes which have a relatively sparse intersection with $\bar{K}_{0}$. To make this precise, and motivate the numerology, we introduce some notation. For $B \in \mathcal{K}_{\bar{\delta}^{1/2}}(\mathbf{T}_{0})$, let 
\begin{equation}\label{form140} \mathcal{K}_{\bar{\delta}}(B) \text{ be a minimal cover of $\bar{K}_{0} \cap B$ by discs of radius $\bar{\delta}$}, \end{equation}
so in particular $B_{\bar{\delta}} \cap \bar{K} \cap B \neq \emptyset$ for all $B_{\bar{\delta}} \in \mathcal{K}_{\bar{\delta}}(B)$. Recall that each disc $B \in \mathcal{K}_{\bar{\delta}^{1/2}}$ satisfies $\mu(B \cap \bar{K}) \geq \omega \cdot \delta^{4\eta} \cdot \bar{\delta}^{\gamma/2}$ by \nref{K1}, so the $(\gamma,C_{\gamma})$-regularity of $\mu$ implies
\begin{equation}\label{form112} \omega \cdot \delta^{4\eta} \cdot (\bar{\delta}^{1/2})^{-\gamma} \leq |\mathcal{K}_{\bar{\delta}}(B)| \leq C_{\gamma} \cdot (\bar{\delta}^{1/2})^{-\gamma}, \qquad B \in \mathcal{K}_{\bar{\delta}^{1/2}}(\mathbf{T}_{0}). \end{equation}
By discarding at most $\tfrac{1}{2}$ of the $\mu$ measure of $\bar{K}_{0} \cap B$, we may assume that all the discs in $\mathcal{K}_{\bar{\delta}}(B)$ are \emph{heavy}, meaning this time that 
\begin{equation}\label{form244} \mu(B_{\bar{\delta}}) \geq \omega \cdot \delta^{4\eta} \cdot \bar{\delta}^{\gamma}, \qquad B_{\bar{\delta}} \in \mathcal{K}_{\bar{\delta}}(B). \end{equation}
We then define $\bar{K}_{B}$ to be the union of the heavy discs in $\mathcal{K}_{\bar{\delta}}(B)$ intersected with $\bar{K}_{0}$.

We let $\mathcal{K}_{\bar{\delta}}$ to be the union of all the families $\mathcal{K}_{\bar{\delta}}(B)$, with $B \in \mathcal{K}_{\bar{\delta}^{1/2}}(\mathbf{T}_{0})$. Thus 
\begin{equation}\label{form114} \omega \cdot \delta^{4\eta} \cdot \M \cdot (\bar{\delta}^{1/2})^{-\gamma} \leq |\mathcal{K}_{\bar{\delta}}| \leq C_{\gamma} \cdot \M \cdot (\bar{\delta}^{1/2})^{-\gamma}. \end{equation}
We also recall from \eqref{form100} that if $T = T_{\bar{\delta}} = \pi_{\theta_{0}}^{-1}(I)$ is an arbitrary tube with $|I| = \bar{\delta}$, in particular if $T \in \mathcal{T}_{\bar{\delta}}$, and if $B \in \mathcal{K}_{\bar{\delta}^{1/2}}(\mathbf{T}_{0}) \subset \mathcal{K}_{\bar{\delta}^{1/2}}$, then
\begin{equation}\label{form113} |\{B_{\bar{\delta}} \in \mathcal{K}_{\bar{\delta}}(B) : T \cap B_{\bar{\delta}} \neq \emptyset\}| \lesssim N_{\bar{\delta}}(B \cap \bar{K} \cap T) \leq (\bar{\delta}^{1/2})^{-\sigma - \bar{\zeta}}. \end{equation}
Therefore, it is reasonable to define that $T \in \mathcal{T}_{\bar{\delta}}$ is $B$-\emph{dense} if 
\begin{equation}\label{form119} |\{B_{\bar{\delta}} \in \mathcal{K}_{\bar{\delta}}(B) : T \cap B_{\bar{\delta}} \neq \emptyset\}| \geq (\bar{\delta}^{1/2})^{-\sigma + 7\bar{\zeta}}, \end{equation}
and otherwise $T$ is $B$-\emph{sparse}. In particular, a $B$-dense tube satisfies
\begin{equation}\label{form220} \mu(\bar{K}_{B} \cap 2T) \geq \omega \cdot \delta^{4\eta} \cdot (\bar{\delta}^{1/2})^{-\sigma + 7\bar{\zeta}} \cdot \bar{\delta}^{\gamma}, \end{equation}
since $2T$ contains $\gtrsim (\bar{\delta}^{1/2})^{-\sigma + 7\bar{\zeta}}$ discs in $\mathcal{K}_{\bar{\delta}}(B)$, all of which are heavy in the sense \eqref{form244}. Then, we say that $T \in \mathcal{T}_{\bar{\delta}}$ is \emph{$\bar{K}_{0}$-sparse} if
\begin{displaymath} |\{B \in \mathcal{K}_{\bar{\delta}^{1/2}}(\mathbf{T}_{0}) : T \text{ is $B$-dense}\}| \leq \M \cdot \bar{\delta}^{8\bar{\zeta}}. \end{displaymath}
Otherwise $T$ is \emph{$\bar{K}_{0}$-dense}. This numerology is also sensible, because each tube $T \in \mathcal{T}_{\bar{\delta}}$ can meet at most $\M$ discs in $\mathcal{K}_{\bar{\delta}^{1/2}}(\mathbf{T}_{0})$ (namely all of them). We next claim that only a small fraction of all the discs $B \in \mathcal{K}_{\bar{\delta}}$ intersect some $\bar{K}_{0}$-sparse tube, which will allow us to restrict attention to $\bar{K}_{0}$-dense tubes in the sequel. Using the uniform upper bound \eqref{form113}, every fixed $\bar{K}_{0}$-sparse tube $T \in \mathcal{T}_{\bar{\delta}}$ satisfies
\begin{align*} |\{B_{\bar{\delta}} \in \mathcal{K}_{\bar{\delta}} : T \cap B_{\bar{\delta}} \neq \emptyset\}| & \leq \mathop{\sum_{B \in \mathcal{K}_{\bar{\delta}^{1/2}}(\mathbf{T}_{0})}}_{T \, \text{is $B$-dense}} |\{B_{\bar{\delta}} \in \mathcal{K}_{\bar{\delta}}(B) : T \cap B_{\bar{\delta}} \neq \emptyset\}| + \mathop{\sum_{B \in \mathcal{K}_{\bar{\delta}^{1/2}}(\mathbf{T}_{0})}}_{T \, \text{is $B$-sparse}} \ldots\\
& \lesssim [\M \cdot \bar{\delta}^{8\bar{\zeta}}] \cdot (\bar{\delta}^{1/2})^{-\sigma - \bar{\zeta}} + \M \cdot (\bar{\delta}^{1/2})^{-\sigma + 7\bar{\zeta}}\\
& \leq 2 \cdot \M \cdot (\bar{\delta}^{1/2})^{-\sigma + 7\bar{\zeta}}. \end{align*}
Since the number of $\bar{K}_{0}$-sparse tubes is bounded from above by $|\mathcal{T}_{\bar{\delta}}| \leq (\bar{\delta}^{1/2})^{\sigma - \gamma - 6\bar{\zeta}}$ by \eqref{form109}, we conclude that the number of discs in $\mathcal{K}_{\bar{\delta}}$ which meet \textbf{some} sparse tube is bounded from above by $\lesssim \M \cdot (\bar{\delta}^{1/2})^{-\gamma + \bar{\zeta}}$. Recalling from \eqref{form114} that 
\begin{displaymath} |\mathcal{K}_{\bar{\delta}}| \geq \omega \cdot \delta^{4\eta} \cdot \M \cdot (\bar{\delta}^{1/2})^{-\gamma}, \end{displaymath}
we may finally infer that if $\delta,\eta > 0$ are small enough, there exist $\geq \tfrac{1}{2} \cdot |\mathcal{K}_{\bar{\delta}}|$ discs in $\mathcal{K}_{\bar{\delta}}$ which intersect some $\bar{K}_{0}$-dense tube in $\mathcal{T}_{\bar{\delta}}$. After this observation, we discard all $\bar{K}_{0}$-sparse tubes from $\mathcal{T}_{\bar{\delta}}$ without changing notation; in other words, we assume in the sequel that all the tubes in $\mathcal{T}_{\bar{\delta}}$ are $\bar{K}_{0}$-dense, and in particular if $T \in \mathcal{T}_{\bar{\delta}}$, then \eqref{form220} holds for at least $\M \cdot \bar{\delta}^{8\bar{\zeta}}$ choices of discs $B \in \mathcal{K}_{\bar{\delta}^{1/2}}(\mathbf{T}_{0})$. 

Before proceeding, we claim the following almost converse to \eqref{form109}:
\begin{equation}\label{form115} |\mathcal{T}_{\bar{\delta}}| \geq (\bar{\delta}^{1/2})^{\sigma - \gamma + 2\bar{\zeta}}, \end{equation}
assuming that $\delta,\eta > 0$ are sufficiently small. Indeed, according to \eqref{form118} and \eqref{form113}, every tube $T \in \mathcal{T}_{\bar{\delta}}$ intersects $\lesssim \M \cdot (\bar{\delta}^{1/2})^{-\sigma - \bar{\zeta}}$ discs in $\mathcal{K}_{\bar{\delta}}$. But since the tubes in $\mathcal{T}_{\bar{\delta}}$ in total intersect $\geq \tfrac{1}{2} \cdot |\mathcal{K}_{\bar{\delta}}| \geq \omega \cdot \delta^{4\eta} \cdot \M \cdot (\bar{\delta}^{1/2})^{-\gamma}$ discs in $\mathcal{K}_{\bar{\delta}}$, the lower bound \eqref{form115} follows.

Recall that the tubes in $\mathcal{T}_{\bar{\delta}}$ have the form $T = \pi_{\theta_{0}}^{-1}(I_{T})$, where $I_{T} \subset I_{0} = [0,3\bar{\delta}^{1/2}]$ is a dyadic interval of length $\bar{\delta}$. Therefore, the left end-points of the intervals $I_{T}$, with $I \in \mathcal{T}_{\bar{\delta}}$, form a certain finite subset $A_{1} \subset \bar{\delta} \cdot \Z \cap I_{0}$. In other words, $A_{1}$ consists of those points $x \in \bar{\delta} \cdot \Z \cap I_{0}$ such that $T_{x} := \pi_{\theta_{0}}^{-1}([x,x + \bar{\delta})) \in \mathcal{T}_{\bar{\delta}}$. In particular, keep in mind that such tubes $T_{x}$ are all $\bar{K}_{0}$-dense (this will be needed in Section \ref{s:|owerBranching}, which is still relatively far away). We record the following corollary of \eqref{form109} and \eqref{form115}:
\begin{equation}\label{form109a} (\bar{\delta}^{1/2})^{\sigma - \gamma + 2\bar{\zeta}} \stackrel{\eqref{form115}}{\leq} |A_{1}| = |\mathcal{T}_{\bar{\delta}}| \stackrel{\eqref{form109}}{\leq} (\bar{\delta}^{1/2})^{\sigma - \gamma - 6\bar{\zeta}}. \end{equation}

Let $\Pi_{1}$ be the uniformly distributed probability measure on $A_{1}$. Then
\begin{equation}\label{form125} \spt \Pi_{1} = A_{1} \subset [0,3\bar{\delta}^{1/2}] \cap \bar{\delta} \cdot \Z. \end{equation}
This measure is a $\bar{\delta}$-measure in the sense of Definition \ref{def:deltaMeasures}. We note that
\begin{equation}\label{form116} \|\Pi_{1}\|_{L^{2}} = \left(\sum_{a \in A_{1}} \Pi_{1}(\{a\})^{2} \right)^{1/2} = |A_{1}|^{-1/2} = |\mathcal{T}_{\bar{\delta}}|^{-1/2} \stackrel{\eqref{form115}}{\leq} (\bar{\delta}^{1/2})^{(\gamma - \sigma - 2\bar{\zeta})/2}. \end{equation} 
We next define another discrete measure, associated to the $y$-coordinates of the discs $B \in \mathcal{K}_{\bar{\delta}^{1/2}}(\mathbf{T}_{0})$. In fact, for every $B \in \mathcal{K}_{\bar{\delta}^{1/2}}(\mathbf{T}_{0})$, let 
\begin{displaymath} y_{B} \in \left(\bar{\delta}^{1/2} \cdot \Z\right) \cap \pi_{\infty}(B), \end{displaymath}
where $\pi_{\infty}(x,y) = y$ is the projection to the $y$-coordinate. This point exists, since $\pi_{\infty}(B)$ is an interval of length $2\bar{\delta}^{1/2}$. We also note that $\pi_{\infty}(B) \subset [-2,2]$, since $B$ intersects $\bar{K}$, hence $B(1)$. Next, let $A_{2} := \{y_{B} : B \in \mathcal{K}_{\bar{\delta}^{1/2}}\}$, and let $\Pi_{2}$ be the uniformly distributed probability measure on $A_{2}$. Then 
\begin{equation}\label{form126} \spt \Pi_{2} = A_{2} \subset [-2,2] \cap \left(\bar{\delta}^{1/2} \cdot \Z\right). \end{equation}
Thus $\Pi_{2}$ is a $\bar{\delta}^{1/2}$-measure, and
\begin{equation}\label{form117} \|\Pi_{2}\|_{L^{2}} = \left(\sum_{a \in A_{2}} \Pi_{2}(\{a\})^{2} \right)^{1/2} = |A_{2}|^{-1/2} \sim \M^{-1/2} \stackrel{\eqref{form118}}{\leq} (\bar{\delta}^{1/2})^{(\sigma - 5\bar{\zeta})/2}. \end{equation} 
Now $\Pi := \Pi_{1} \times \Pi_{2}$ is a discrete probability measure supported on $A_{1} \times A_{2}$, but it is not evident that $\Pi$ has anything to do with the set $\bar{K}_{0} = \bar{K}_{\mathbf{T}_{0}}$. To clarify the connection, we need to define a certain subset of $A_{1} \times A_{2}$ of substantial $\Pi$-measure. Recall that that every point $x \in A_{1}$ was the left end-point of a certain $\bar{\delta}$-interval $I_{x} = [x,x + \bar{\delta}) \subset [0,3\bar{\delta}^{1/2}]$ such that $T_{x} = \pi_{\theta_{0}}^{-1}(I_{x}) \in \mathcal{T}_{\bar{\delta}}$. Similarly, recall that every point $y \in A_{2}$ was contained in the $\pi_{\infty}$-projection of a certain disc $B_{y} \in \mathcal{K}_{\bar{\delta}^{1/2}}(\mathbf{T}_{0})$. With this notation, we define
\begin{equation}\label{form122} G := \{(x,y) \in A_{1} \times A_{2} : T_{x} \cap B_{\bar{\delta}} \neq \emptyset \text{ for some } B_{\bar{\delta}} \in \mathcal{K}_{\bar{\delta}}(B_{y})\}. \end{equation} 
Morally, $G$ consists of those tube-disc pairs $(T,B) \in \mathcal{T}_{\bar{\delta}} \times \mathcal{K}_{\bar{\delta}^{1/2}}(\mathbf{T}_{0})$, where $T$ intersects (the $\bar{\delta}$-neighbourhood of) $\bar{K}$ inside $B$. Note that if $x \in A_{1}$, then $T_{x}$ is a $\bar{K}_{0}$-dense tube, so $T_{x} \cap B_{\bar{\delta}} \neq \emptyset$ for some $B_{\bar{\delta}} \in \mathcal{K}_{\bar{\delta}^{1/2}}(B)$ (see \eqref{form119}) for $\geq \M \cdot \bar{\delta}^{8\bar{\zeta}}$ distinct discs $B \in \mathcal{K}_{\bar{\delta}^{1/2}}(\mathbf{T}_{0})$. In other words,
\begin{displaymath} |\pi_{1}^{-1}\{x\} \cap G| \geq \M \cdot \bar{\delta}^{8\bar{\zeta}}, \qquad x \in A_{1}, \end{displaymath}
and consequently (recall that $|A_{2}| \sim \M$)
\begin{equation}\label{form120} \Pi(G) = \frac{|G|}{|A_{1}||A_{2}|} \gtrsim \bar{\delta}^{8\bar{\zeta}}. \end{equation}

\subsection{From projections to convolutions} While constructing the sets $A_{1},A_{2},G$ above, all the arguments were based on the structure of $\bar{K}_{0} = \bar{K}_{\mathbf{T}_{0}}$ relative to tubes which were pre-images of intervals under the $\pi_{\theta_{0}}$-projection. Next, we exploit the information available for the $\pi_{\theta_{1}}$-projection in \nref{K2}, namely that
\begin{equation}\label{form123} N_{\bar{\delta}}(\pi_{\theta_{1}}(\bar{K}_{0})) \leq (\bar{\delta}^{1/2})^{\sigma - \gamma - 6\bar{\zeta}}. \end{equation}
The plan is, roughly speaking, to use \eqref{form123} to show that the convolution between $\Pi_{1}$ and $(\theta_{1} - \theta_{0})\Pi_{2}$ has nearly the same $L^{2}$-norm as $\Pi_{1}$, where $(\theta_{1} - \theta_{0})\Pi_{2}$ refers to the push-forward of $\Pi_{2}$ under the map $y \mapsto (\theta_{1} - \theta_{0})y$. This will be quantified in \eqref{form130}. To get started, we claim that there exists an absolute constant $C > 0$ such that
\begin{equation}\label{form121} \pi_{\theta_{1} - \theta_{0}}(G) \subset [\pi_{\theta_{1}}(\bar{K}_{0})]_{C\bar{\delta}}, \end{equation}
where $G \subset A_{1} \times A_{2}$ is the set defined in \eqref{form122}, and the right hand side refers to the $C\bar{\delta}$-neighbourhood. To prove \eqref{form121}, fix $(x,y) \in G$, so that $T_{x} \cap B_{\bar{\delta}} \neq \emptyset$ for some $B_{\bar{\delta}} \in \mathcal{K}_{\bar{\delta}^{1/2}}(B_{y})$, where $B_{y} \in \mathcal{K}_{\bar{\delta}^{1/2}}(\mathbf{T}_{0})$. Here $T_{x} = \pi_{\theta_{0}}^{-1}(I_{x})$ for some interval $I_{x} \in \mathcal{D}_{\bar{\delta}}(\R)$, whose left end-point is $x$. Then we know the following:
\begin{enumerate}
\item $x \in I_{x}$ and $y \in \pi_{\infty}(B_{y})$ by definitions of $T_{x}$ and $B_{y}$.
\item There exists a point $(x_{0},y_{0}) \in T_{x} \cap B_{\bar{\delta}}$, and $B_{\bar{\delta}} \in \mathcal{K}_{\bar{\delta}}(B_{y})$, so $B_{\bar{\delta}} \cap \bar{K}_{0} \cap B_{y} \neq \emptyset$ (recall from \eqref{form140} that $\mathcal{K}_{\bar{\delta}}(B_{y})$ was a minimal $\bar{\delta}$-cover of $\bar{K}_{0} \cap B_{y}$). In particular, $\dist((x_{0},y_{0}),\bar{K}_{0}) \leq \bar{\delta}$ and $\dist((x_{0},y_{0}),B_{y}) \leq \bar{\delta}$.
\end{enumerate}
From (2), we first deduce that $\pi_{\theta_{0}}(x_{0},y_{0}) \in I_{x}$, hence
\begin{displaymath} |(x_{0} + \theta_{0}y_{0}) - x| \leq \bar{\delta}. \end{displaymath}
Also, from (1)-(2) it follows that $|y_{0} - y| \lesssim \bar{\delta}^{1/2}$. Therefore, recalling also from \eqref{form128} that $|\theta_{0} - \theta_{1}| \leq \bar{\delta}^{1/2}$, we find
\begin{align*} |\pi_{\theta_{1}}(x_{0},y_{0}) & - \pi_{\theta_{1} - \theta_{0}}(x,y)| = |(x_{0}  + \theta_{1}y_{0}) - (x + (\theta_{1} - \theta_{0})y)|\\
& \leq |(x_{0} + \theta_{0}y_{0}) - x| + |\theta_{1} - \theta_{0}| \cdot |y_{0} - y| \lesssim \bar{\delta}.  \end{align*}
Since $\pi_{\theta_{1}}(x_{0},y_{0}) \in [\pi_{\theta_{1}}(\bar{K}_{0})]_{\bar{\delta}}$ by (2), we conclude the proof of \eqref{form121}. Combining \eqref{form123}-\eqref{form121}, we obtain
\begin{equation}\label{form124} N_{\bar{\delta}}(\pi_{\theta_{1} - \theta_{0}}(G)) \lesssim (\bar{\delta}^{1/2})^{\sigma - \gamma - 6\bar{\zeta}}. \end{equation}
We abbreviate $\theta := \theta_{1} - \theta_{0}$ from now on. Combined with the lower bound \eqref{form120} for the $\Pi$-measure of $G$, we will shortly infer from \eqref{form124} a lower bound for the $L^{2}$-norm of the projection $\pi_{\theta}\Pi$. To make this perfectly precise, we will need an additional piece of notation. We would prefer $\pi_{\theta}$ to map $\R^{2}$ inside the discrete set $\bar{\delta} \cdot \Z$. So, let us, in place of $\pi_{\theta}$, consider the map $\bar{\pi}_{\theta} \colon \R^{2} \to \bar{\delta} \cdot \Z$,
\begin{displaymath} \bar{\pi}_{\theta} := [x] + [\theta y], \end{displaymath} 
where $[r] \in \bar{\delta} \cdot \Z$ and $0 \leq c < \bar{\delta}$ are determined by $r = [r] + c$. Now it follows from \eqref{form124} that
\begin{equation}\label{form129} |\bar{\pi}_{\theta}(G)| \lesssim (\bar{\delta}^{1/2})^{\sigma - \gamma - 6\bar{\zeta}}. \end{equation}
The benefit of considering the projection $\bar{\pi}_{\theta}$ is that the $\bar{\pi}_{\theta}$-projection of $\Pi$ can be expressed as the following convolution:
\begin{displaymath} \bar{\pi}_{\theta}(\Pi) = [\Pi_{1}] \ast [\theta\Pi_{2}] = \Pi_{1} \ast [\theta \Pi_{2}] \end{displaymath}
where $[\theta \Pi_{2}]$ refers to the push-forward of $\Pi_{2}$ under the map $y \mapsto [\theta y]$ (recall from \eqref{form125} that $\Pi_{1}$ is already supported on $\bar{\delta} \cdot \Z$, so $[\Pi_{1}] = \Pi$). Note that both $\Pi_{1}$ and $[\theta \Pi_{2}]$ are discrete measures supported on $\bar{\delta} \cdot \Z$, so the same is true for their convolution. From the facts $(\Pi_{1} \times \Pi_{2})(G) = \Pi(G) \gtrsim \bar{\delta}^{8\bar{\zeta}}$ (recall \eqref{form120}) and \eqref{form129}, we can deduce the following lower bound for the (discrete) $L^{2}$-norm of this convolution:
\begin{align*} \bar{\delta}^{8\bar{\zeta}} \lesssim \sum_{z \in \bar{\pi}_{\theta}(G)} \bar{\pi}_{\theta}(\Pi|_{G})(\{z\}) & \leq |\bar{\pi}_{\theta}(G)|^{1/2} \left(\sum_{z \in \bar{\delta} \cdot \Z} \bar{\pi}_{\theta}(\Pi)(\{z\})^{2} \right)^{1/2}\\
& \lesssim (\bar{\delta}^{1/2})^{(\sigma - \gamma - 6\bar{\zeta})/2} \cdot \|\Pi_{1} \ast [\theta \Pi_{2}]\|_{L^{2}}, \end{align*}
or in other words
\begin{equation}\label{form130} \|\Pi_{1} \ast [\theta \Pi_{2}]\|_{L^{2}} \gtrsim (\bar{\delta}^{1/2})^{(\gamma - \sigma + 22\bar{\zeta})/2} \stackrel{\eqref{form116}}{\geq} \bar{\delta}^{6\bar{\zeta}} \cdot \|\Pi_{1}\|_{L^{2}}. \end{equation}
The lower bound \eqref{form130} will eventually place in a position to apply Shmerkin's inverse theorem, Theorem \ref{shmerkin}. Since $|\theta| = |\theta_{1} - \theta_{0}| \leq \bar{\delta}^{1/2}$ by \eqref{form128}, both $\Pi_{1}$ and $[\theta \Pi_{2}]$ are probability measures supported on $[-5\bar{\delta}^{1/2},5\bar{\delta}^{1/2}] \cap \bar{\delta} \cdot \Z \subset [-1,1] \cap \bar{\delta} \cdot \Z$. 

To close this section, we record an upper bound for the $L^{2}$-norm of $[\theta \Pi_{2}]$. This is based on the following elementary observation: since $|\theta| \geq \bar{\delta}^{(1 + \epsilon)/2}$ by \eqref{form128}, we have
\begin{equation}\label{form132} |\{y \in \bar{\delta}^{1/2} \cdot \Z :  [\theta y] = z\}| \lesssim \bar{\delta}^{-\epsilon/2}, \qquad z \in \bar{\delta} \cdot \Z. \end{equation}
Indeed, if $[\theta y_{1}] = [\theta y_{2}]$, then certainly $|\theta| \cdot |y_{1} - y_{2}| \leq \bar{\delta}$, hence $|y_{1} - y_{2}| \lesssim \bar{\delta}^{1/2 - \epsilon/2}$. But any fixed interval of length $\sim \bar{\delta}^{1/2 - \epsilon/2}$ contains $\lesssim \bar{\delta}^{-\epsilon/2}$ points from $\bar{\delta}^{1/2} \cdot \Z$, and this implies \eqref{form132}. Now, recall from \eqref{form126} that $\Pi_{2}$ was defined to be the uniform probability measure on the set $A_{2} \subset [-2,2] \cap \bar{\delta}^{1/2} \cdot \Z$, and $|A_{2}| \sim \M$, so $\Pi_{2}(\{y\}) \sim \M^{-1}$ for all $y \in A_{2}$. Hence, it follows from \eqref{form132}, and $|[\theta A_{2}]| \leq |A_{2}| \sim \M$, that
\begin{align} \|[\theta \Pi_{2}]\|_{2} & = \left( \sum_{z \in [\theta A_{2}]} [\theta \Pi_{2}](\{z\})^{2} \right)^{1/2} \lesssim \left( \M \cdot \M^{-2} \cdot \bar{\delta}^{-\epsilon} \right)^{1/2} \notag\\
&\label{form232} = \bar{\delta}^{-\epsilon/2} \cdot \M^{-1/2} \stackrel{\eqref{form118}}{\lesssim} (\bar{\delta}^{1/2})^{(\sigma - 5\bar{\zeta})/2 - \epsilon} \leq (\bar{\delta}^{1/2})^{\sigma/2 - 3\bar{\zeta} - \epsilon}. \end{align}  

\subsection{Applying Shmerkin's inverse theorem}

We now arrive at the core of the proof of Proposition \ref{mainProp}: the proof will be formally concluded in this section, although some technicalities will spill over to the next one. In \eqref{form130}, we have seen that $\Pi_{1}$ and $[\theta \Pi_{2}]$ are $\bar{\delta}$-measures with the property 
\begin{equation}\label{form135} \|\Pi_{1} \ast [\theta \Pi_{2}]\| \gtrsim \bar{\delta}^{6\bar{\zeta}}\|\Pi_{1}\|_{L^{2}}. \end{equation}
This places us in a position to apply Shmerkin's inverse theorem, Theorem \ref{shmerkin}. We have already fixed the parameters "$\epsilon$" and "$m_{0}$" a while ago, in Section \ref{s:constants}. In particular, they were chosen so that the following holds for all $m \geq m_{0}$:
\begin{equation}\label{form248} \gamma - \sigma + 10\zeta_{0} < (1 - \alpha - \epsilon)(\sigma - 10(\epsilon + \zeta_{0})) + (1 - \rho)(1 - \tfrac{30(\epsilon + \zeta_{0})}{\alpha \rho} - \tfrac{10C_{\alpha}}{\alpha \rho m})\alpha. \end{equation}
The constants "$\epsilon,\rho,\zeta_{0}$" only depended on $\alpha,\beta,\sigma$, while "$m_{0}$" additionally depended on $C_{\alpha}$ (as is evident from the inequality above). Now, Theorem \ref{shmerkin} tells us that associated with these two parameters "$\epsilon$" and "$m_{0}$" there correspond constants $\kappa = \kappa(\epsilon,m_{0}) > 0$, and $m \geq m_{0}$, such that if 
\begin{equation}\label{form247} \bar{\delta} = (2^{-m})^{N} \end{equation}
for some sufficiently large integer $N \geq 1$, and and if $\eta_{1},\eta_{2}$ are any $\bar{\delta}$-measures satisfying 
\begin{equation}\label{form245} \|\eta_{1} \ast \eta_{2}\| \geq \delta^{\kappa}\|\eta_{1}\|_{L^{2}}, \end{equation}
then interesting things start to happen. We have previously established that our specific $\bar{\delta}$-measures $\eta_{1} = \Pi_{1}$ and $\eta_{2} = [\theta \Pi_{2}]$ satisfy \eqref{form135}, and \eqref{form245} follows if $6\bar{\zeta} = 6O(\tau)\zeta < \kappa(\epsilon,m_{0})$. This is indeed one of the restrictions on the parameter "$\zeta$" which we imposed in \eqref{form238} (and this was made completely precise in \eqref{form141}), so interesting things do happen: there exist sets $U \subset \spt \Pi_{1} = A_{1}$ and $V \subset \spt [\theta \Pi_{2}] = A_{2}$ such that
\begin{equation}\label{form219} \Pi_{1}(U) \geq \bar{\delta}^{2\epsilon}\ \quad \text{and} \quad [\theta \Pi_{2}](V) \geq \bar{\delta}^{\epsilon}, \end{equation} 
(the first condition is equivalent to "$\|(\Pi_{1})|_{U}\|_{L^{2}} \geq \bar{\delta}^{\epsilon}\|\Pi_{1}\|_{L^{2}}$" since $\Pi_{1}$ is the uniformly distributed measure on $A_{1}$) and the following properties hold:
\begin{itemize}
\item[\textup{(A)}] there is a sequence $(R_{s}^{1})_{s = 0}^{N - 1} \subset \{1,\ldots,2^{m}\}^{N - 1}$, such that 
\begin{displaymath} N(U \cap I,2^{-(s + 1)m}) = R_{s}^{1} \end{displaymath}
for all dyadic intervals $I$ of length $2^{-ms}$ intersecting $U$,
\item[\textup{(B)}] there is a sequence $(R_{s}^{2})_{s = 0}^{N - 1} \subset \{1,\ldots,2^{m}\}^{N - 1}$, such that 
\begin{displaymath} N(V \cap I,2^{-(s + 1)m}) = R_{s}^{2} \end{displaymath}
for all dyadic intervals $I$ of length $2^{-ms}$ intersecting $V$.
\end{itemize}
For each $s \in \{0,\ldots,N - 1\}$, either $R_{s}^{2} = 1$ or $R_{s}^{1} \geq 2^{(1 - \epsilon)m}$, and the set $\calS = \{s : R_{s}^{1} \geq 2^{(1 - \epsilon)m}\}$ satisfies
\begin{equation}\label{form216} m|\calS| \geq \log \|[\theta \Pi_{2}]\|_{L^{2}}^{-2} - \epsilon \log (1/\bar{\delta}). \end{equation}
Based on the information above, we compute a preliminary lower bound for the cardinality of the set $U$, and hence $A_{1} \supset U$. The cardinality of $U$ equals the product of the "branching" numbers $R_{s}^{1}$, $s \in \{0,\ldots,N - 1\}$, and hence
\begin{displaymath} |A_{1}| \geq |U| = \prod_{s = 0}^{N - 1} R_{s}^{1} \geq \prod_{s \in \mathcal{S}} 2^{(1 - \epsilon)m} \times \prod_{s \notin \mathcal{S}} R_{s}^{1} = 2^{(1 - \epsilon)m|\mathcal{S}|} \times \prod_{s \notin \mathcal{S}} R_{s}^{1}.  \end{displaymath}
From \eqref{form216} and \eqref{form232} we can obtain a decent lower bound on the first factor, but the second factor needs more work. It turns out that, thanks to the regularity of the measure $\mu_{A}$, we can prove an "almost uniform" lower bound on the numbers $R_{s}^{1}$:

\begin{lemma}\label{lemma11} For any $\rho \in (0,1)$ (in particular the constant $\rho = \rho(\alpha,\beta,\sigma)$ chosen in Section \ref{s:constants}), there exists a subset of indices
\begin{equation}\label{form230} \mathcal{G} \subset \{N/2,\ldots,N - 1\} \quad \text{with} \quad |\mathcal{G}| \geq \left(1 - \left[\tfrac{27\bar{\zeta} + 4\epsilon}{\alpha \rho} + \tfrac{5C_{\alpha}}{\alpha \rho m}\right] \right) \cdot \tfrac{N}{2} \end{equation}
with the property that
\begin{equation}\label{form217} R_{s}^{1} \geq 2^{(1 - \rho)\alpha m}, \qquad s \in \mathcal{G}. \end{equation}
\end{lemma}

Before proving Lemma \ref{lemma11}, we use it to conclude our estimation for $|A_{1}|$:
\begin{align} |A_{1}| \geq 2^{(1 - \epsilon)m|\mathcal{S}|} \times \prod_{s \in \mathcal{G} \, \setminus \, \mathcal{S}} R_{s}^{1} & \geq 2^{(1 - \epsilon)m|\mathcal{S}|} \cdot 2^{(1 - \rho)\alpha m(|\mathcal{G}| - |\mathcal{S}|)} \notag\\
&\label{form246} \geq 2^{(1 - \alpha - \epsilon)m|\mathcal{S}|} \cdot 2^{(1 - \rho)\alpha m|\mathcal{G}|}. \end{align}
From \eqref{form216} and \eqref{form232}, we further infer that
\begin{displaymath} m|\mathcal{S}| \geq \log \|[\theta \Pi_{2}]\|_{L^{2}}^{-2} - \epsilon \log (1/\bar{\delta}) \stackrel{\eqref{form232}}{\geq} (\tfrac{\sigma}{2} - 3\bar{\zeta} - 2\epsilon) \log (1/\bar{\delta}). \end{displaymath} 
Comparing the ensuing lower bound for $|A_{1}|$ with the upper bound obtained in \eqref{form109a}, and noting that $2^{-mN/2} = \bar{\delta}^{1/2}$ with our notation (by \eqref{form247}), we obtain
\begin{displaymath} (\bar{\delta}^{1/2})^{\sigma - \gamma - 6\bar{\zeta}} \stackrel{\eqref{form109a}}{\geq} |A_{1}| \stackrel{\eqref{form230}-\eqref{form246}}{\geq} \bar{\delta}^{-(1 - \alpha - \epsilon)(\sigma/2 - 3\bar{\zeta} - 2\epsilon)} \cdot (\bar{\delta}^{1/2})^{-(1 - \rho)\alpha \left(1 - \left[\tfrac{27\bar{\zeta} + 4\epsilon}{\alpha \rho} + \tfrac{5C_{\alpha}}{\alpha \rho m}\right]\right)}. \end{displaymath}
This inequality however contradicts our choices of parameters at \eqref{form248}, assuming that $6\bar{\zeta} = 6O(\tau)\zeta \leq \zeta_{0}$. Finally, therefore, a contradiction has been obtained: \eqref{form135} cannot hold, hence \eqref{form200} cannot hold for $\delta,\eta > 0$ small enough (depending on all the constants $\alpha,\beta,C_{\alpha},C_{\beta},\sigma,\tau,\Delta_{0},\eta_{0}$, as we have seen). The proof of Proposition \ref{mainProp} is complete.

\subsection{Lower bounds for branching numbers}\label{s:|owerBranching} It remains to prove Lemma \ref{lemma11}. We start by defining the following collection of tubes:
\begin{equation}\label{form142} \mathcal{T}_{U} := \{\pi_{\theta_{0}}^{-1}(2I_{x}) : x \in U\}. \end{equation}
Here $I_{x} \in \mathcal{D}_{\bar{\delta}}(\R)$ refers to the dyadic interval whose left endpoint is $x \in U \subset A_{1}$. So, $\mathcal{T}_{U}$ is a sub-collection of $\mathcal{T}_{\bar{\delta}}$, except that all the tubes have been thickened by a factor of $2$. We also write $T_{U} := \cup \mathcal{T}_{U}$. We claim the following: if $\delta > 0$ is sufficiently small, then there exists a disc $B_{0} \in \mathcal{K}_{\bar{\delta}^{1/2}}(\mathbf{T}_{0})$ such that
\begin{equation}\label{form218} \mu(B_{0} \cap T_{U}) \geq (\bar{\delta}^{1/2})^{\gamma + 26\bar{\zeta} + 4\epsilon}. \end{equation}
The proof is based on our arrangement that all the tubes $T \in \mathcal{T}_{\bar{\delta}}$, and in particular $T \in \mathcal{T}_{U}$, are $\bar{K}_{0}$-dense. We recall from \eqref{form220} that this means that
\begin{equation}\label{form222} \mu(\bar{K}_{B} \cap \pi_{\theta_{0}}^{-1}(2I_{x})) \geq \omega \cdot \delta^{4\eta} \cdot (\bar{\delta}^{1/2})^{-\sigma + 7\bar{\zeta}} \cdot \bar{\delta}^{\gamma}, \qquad \pi_{\theta_{0}}^{-1}(2I_{x}) \in \mathcal{T}_{U}, \end{equation}
holds for at least $\M \cdot \bar{\delta}^{8\bar{\zeta}}$ different choices of $B \in \mathcal{K}_{\bar{\delta}^{1/2}}(\mathbf{T}_{0})$. As the following computation shows, this will imply that the average disc $B \in \mathcal{K}_{\bar{\delta}^{1/2}}(\mathbf{T}_{0})$ satisfies \eqref{form218}. Before we begin, it will be useful to recall that $\M := |\mathcal{K}_{\bar{\delta}^{1/2}}(\mathbf{T}_{0})|$ by \eqref{form118}, and to note that since $\Pi_{1}$ is the uniformly distributed measure on $A_{1}$, 
\begin{equation}\label{form221} |U| \stackrel{\eqref{form219}}{\geq} \bar{\delta}^{2\epsilon}|A_{1}| \stackrel{\eqref{form109a}}{\geq} (\bar{\delta}^{1/2})^{\sigma - \gamma + 2\bar{\zeta} + 4\epsilon}. \end{equation}
With these facts in mind, we estimate as follows:
\begin{align*} \frac{1}{\M} \sum_{B \in \mathcal{K}_{\bar{\delta}^{1/2}}(\mathbf{T}_{0})} \mu(B \cap T_{U}) & \gtrsim \sum_{x \in U} \frac{1}{\M} \sum_{B \in \mathcal{K}_{\bar{\delta}^{1/2}}(\mathbf{T}_{0})} \mu(\bar{K}_{B} \cap \pi_{\theta_{0}}^{-1}(2I_{x}))\\
& \stackrel{\eqref{form222}}{\geq} \sum_{x \in U} \frac{\omega}{\M} \cdot \M \cdot \bar{\delta}^{8\bar{\zeta}} \cdot \delta^{4\eta} \cdot (\bar{\delta}^{1/2})^{-\sigma + 7\bar{\zeta}} \cdot \bar{\delta}^{\gamma}\\
& \stackrel{\eqref{form221}}{\geq} \omega \cdot \delta^{4\eta} \cdot (\bar{\delta}^{1/2})^{\gamma + 25\bar{\zeta} + 4\epsilon}. \end{align*} 
If $\delta,\eta > 0$ is small enough that $\omega \cdot \delta^{4\eta}\geq \bar{\delta}^{\bar{\zeta}/2}$, this proves the existence of a disc $B_{0} \in \mathcal{K}_{\bar{\delta}^{1/2}}(\mathbf{T}_{0})$ satisfying \eqref{form218}.

Recall that $\mu = \mu_{A} \times \mu_{B}$, where $\mu_{A}$ is $(\alpha,C_{\alpha})$-regular, and $\mu_{B}$ is $(\beta,C_{\beta})$-regular. Recall also that $\gamma = \alpha + \beta$.  By Fubini's theorem,
\begin{displaymath} (\bar{\delta}^{1/2})^{\gamma + 26\bar{\zeta} + 4\epsilon} \leq \mu(B_{0} \cap T_{U}) \leq \int_{\pi_{\infty}(B_{0})} \mu_{A}(\{\bar{x} \in \R : (\bar{x},y) \in B_{0} \cap T_{U}\}) \, d\mu_{B}(y). \end{displaymath}
Since $\mu_{B}(\pi_{\infty}(B_{0})) \leq 2C_{\beta} \cdot (\bar{\delta}^{1/2})^{\beta}$, we infer that there exists $y_{0} \in \pi_{\infty}(B_{0})$ with the property
\begin{equation}\label{form223} \mu_{A}(\{\bar{x} \in \R : (\bar{x},y_{0}) \in B_{0} \cap T_{U}\}) \geq (\bar{\delta}^{1/2})^{\alpha + 27\bar{\zeta} + 4\epsilon},  \end{equation}
for $\delta > 0$ small enough. Let us abbreviate $\bar{U} := \{\bar{x} \in \R : (\bar{x},y_{0}) \in B_{0} \cap T_{U}\}$. We observe that 
\begin{displaymath} \bar{U} \subset \pi_{0}(B_{0}) =: \bar{I}, \end{displaymath}
which is an interval of length $2\bar{\delta}^{1/2}$. For slight technical convenience, we actually prefer that $|\bar{I}| = \bar{\delta}^{1/2}$, and this can be arranged: if $\bar{U}$ is replaced by its intersection with either the left or the right half of $\bar{I}$, then \eqref{form223} remains valid, and this new $\bar{U}$ fits inside an interval of length precisely $\bar{\delta}^{1/2}$. With this reduction, we assume that $\diam(\bar{U}) \leq |\bar{I}| = \bar{\delta}^{1/2}$.

At this point we pause and clarify the connection between the sets $\bar{U}$ and $U$ (where the latter was defined around \eqref{form219}, and, more importantly, was used to define the tubes $\mathcal{T}_{U}$ in \eqref{form142}). We claim that
\begin{equation}\label{form224} \bar{U} + y_{0}\theta_{0} \subset U_{2\bar{\delta}},  \end{equation}
where the right hand side refers to the $2\bar{\delta}$-neighbourhood of $U$. To see this, fix $\bar{x} \in \bar{U}$, and observe that $(\bar{x},y_{0}) \in T_{U}$ by definition. In other words $(\bar{x},y_{0}) \in \pi_{\theta_{0}}^{-1}(2I_{x})$ for some $x \in U$. Since $x \in I_{x}$, this implies that $|\bar{x} + \theta_{0}y_{0} - x| \leq 2\bar{\delta}$, and \eqref{form224} has been verified. This inclusion means that we may easily obtain (lower) bounds for the branching numbers $R_{s}^{1}$ associated to $U$ by studying the set $\bar{U}$ instead.

The following proposition will quantify that $\bar{U}$ has nearly $\alpha$-dimensional branching at almost all scales:

\begin{proposition}\label{branchProp} Let $\mu_{A}$ be an $(\alpha,C_{\alpha})$-regular measure on $\R$, $C_{\alpha} \geq 1$, and let $\delta = 2^{-mN}$ be a dyadic number, for some $m,N \geq 1$. Assume that $U \subset [0,1)$ is a Borel set with $\mu_{A}(U) \geq \delta^{\omega}$ for some positive parameter $\omega \in (0,1)$. Fix $\rho \in (0,1)$, and consider those \emph{good} scales $s \in \{0,\ldots,N - 1\}$ with the property
\begin{equation}\label{form227} \max_{I \in \mathcal{D}_{2^{-ms}}} N_{2^{-m(s + 1)}}(U \cap I) \geq 2^{(1 - \rho)\alpha m + 3}. \end{equation}
Then, the number of good scales is no smaller than
\begin{equation}\label{form228} \left(1 - \left[\tfrac{\omega}{\alpha \rho} + \tfrac{2C_{\alpha}}{\alpha \rho m}\right] \right)N. \end{equation}  \end{proposition} 

\begin{proof} We start by covering $U$ by a collection $\mathcal{U} \subset \mathcal{D}_{\delta}$ of dyadic intervals, and we discard those intervals which do not intersect $\spt \mu_{A}$. The union of the remaining intervals has the same $\mu_{A}$ measure as $U$, and it also suffices to prove \eqref{form227} for this smaller set. So, we assume with no loss of generality that every interval $I \in \mathcal{U}$ intersects $\spt \mu_{A}$. This allows us to pick one point $x_{I} \in I \cap \spt \mu_{A} \subset U \cap \spt \mu_{A}$ for each $I \in \mathcal{U}$, and write $\bar{\mu}$ for the uniformly distributed probability measure on $\{x_{I} : I \in \mathcal{U}\}$. Since $\mu_{A}$ was assumed to be $(\alpha,C_{\alpha})$-regular, and $\mu(U) \geq \delta^{\omega}$, we have $|\mathcal{U}| \geq C_{\alpha}^{-1} \cdot \delta^{\omega - \alpha}$. Therefore
\begin{equation}\label{form261} H(\bar{\mu},\mathcal{D}_{\delta}) \geq \log (C_{\alpha}^{-1} \cdot \delta^{\omega - \alpha}) = (\alpha - \omega)mN - \log C_{\alpha} \geq (\alpha - \omega)mN - C_{\alpha}. \end{equation}
We then decompose the $\delta$-entropy of $\bar{\mu}$ by repeated applications of \eqref{entropy} (note that $H(\bar{\mu},\mathcal{D}_{0}) = 0$, since $\spt \bar{\mu} \subset [0,1)$):
\begin{align*} (\alpha - \omega)mN - C_{\alpha} \leq H(\bar{\mu},\mathcal{D}_{\delta}) & = \sum_{s = 0}^{N - 1} H(\bar{\mu},\mathcal{D}_{2^{-m(s + 1)}} \mid \mathcal{D}_{2^{-ms}})\\
& = \sum_{s = 0}^{N - 1} \sum_{I \in \mathcal{D}_{2^{-ms}}} \bar{\mu}(I) \cdot H(\bar{\mu}_{I},\mathcal{D}_{2^{-m(s + 1)}}).  \end{align*}
Here $\bar{\mu}_{I} = \bar{\mu}(I)^{-1}\bar{\mu}|_{I}$ as usual. In particular, $\bar{\mu}_{I}$ is a probability measure supported on $I \cap \spt \mu_{A}$ (no closure is needed, since $\bar{\mu}$ is a discrete measure). To make further progress, we derive a uniform entropy upper bound from the $(\alpha,C_{\alpha})$-regularity of $\mu_{A}$:
\begin{displaymath} H(\bar{\mu}_{I},\mathcal{D}_{2^{-m(s + 1)}}) \leq \log N_{2^{-m(s + 1)}}(I \cap \spt \mu_{A}) \leq \log (C_{\alpha} \cdot 2^{\alpha m}) \leq \alpha m + C_{\alpha}, \quad I \in \mathcal{D}_{2^{-ms}}. \end{displaymath}
As a consequence of the preceding inequality, we also see that
\begin{equation}\label{form225} H(\bar{\mu},\mathcal{D}_{2^{-m(s + 1)}} \mid \mathcal{D}_{2^{-ms}}) \leq \alpha m + C_{\alpha}, \qquad 0 \leq s \leq N - 1. \end{equation} 
Now, a scale index $s \in \{0,\ldots,N - 1\}$ is called \emph{bad} if 
\begin{displaymath} H(\bar{\mu},\mathcal{D}_{2^{-m(s + 1)}} \mid \mathcal{D}_{2^{-ms}}) \leq (1 - \rho)\alpha m + 3, \end{displaymath}
and \emph{good} otherwise. Note that if an index $s \in \{0,\ldots,N - 1\}$ is good, then in particular $H(\bar{\mu}_{I},\mathcal{D}_{2^{-m(s + 1)}}) \geq (1 - \rho)\alpha m + 3$ for some interval $I \in \mathcal{D}_{2^{-ms}}$, which implies \eqref{form227}. So, it remains to show that there are not too many bad scales. We observe that
\begin{align}\label{form226} (\alpha - \omega)mN - C_{\alpha} \stackrel{\eqref{form261}}{\leq} H(\bar{\mu},\mathcal{D}_{\delta}) & = \sum_{s = 0}^{N - 1} H(\bar{\mu},\mathcal{D}_{2^{-m(s + 1)}} \mid \mathcal{D}_{2^{-ms}})\\
& = \sum_{s \,\mathrm{ bad}} \ldots + \sum_{s \,\mathrm{ good}} \ldots. \notag \end{align}
Let $\lambda N$ be the number of bad scales, so the number of good scales is $(1 - \lambda)N$. With this notation, and using \eqref{form225}, the right hand side is bounded from above by
\begin{displaymath} \lambda N\cdot [(1 - \rho)\alpha m + 3] + (1 - \lambda)N \cdot [\alpha m + C_{\alpha}] \leq (1 - \lambda \rho) \alpha mN + (3 + C_{\alpha})N. \end{displaymath}
Comparing this upper bound with the lower bound in \eqref{form226}, we obtain
\begin{displaymath} (1 - \lambda \rho) \alpha mN \geq (\alpha - \omega)mN - (3 + C_{\alpha})N - C_{\alpha} \geq (\alpha - \omega)mN - 5C_{\alpha}N, \end{displaymath}
or in other words
\begin{displaymath} \lambda \leq \left[\tfrac{\omega}{\alpha \rho} + \tfrac{5C_{\alpha}}{\alpha \rho m}\right]. \end{displaymath}
Since the number of good scales is $(1 - \lambda)N$, we obtain \eqref{form228}. \end{proof}

We can then prove Lemma \ref{lemma11}, or in other words \eqref{form230}-\eqref{form217}. Recall from \eqref{form223} that $\bar{U} \subset \bar{I}$ satisfied $\mu_{A}(\bar{U}) \geq (\bar{\delta}^{1/2})^{\alpha + 27\bar{\zeta} + 4\epsilon} =: (\bar{\delta}^{1/2})^{\alpha + \omega}$, where
\begin{displaymath} \omega := 27\bar{\zeta} + 4\epsilon. \end{displaymath} 
Since $\bar{I}$ is an interval of length $\bar{\delta}^{1/2}$, the rescaled and re-normalised measure 
\begin{displaymath} \mu_{\bar{I}} := (\bar{\delta}^{1/2})^{-\alpha} \cdot T_{\bar{I}}\mu_{A} \end{displaymath}
is an $(\alpha,C_{\alpha})$-regular measure, and $T_{\bar{I}}(\bar{U}) \subset [0,1]$ is a Borel set with $\mu_{\bar{I}}(T_{\bar{I}}(\bar{U})) \geq (\bar{\delta}^{1/2})^{\omega}$. Therefore, Proposition \ref{branchProp} can be applied to $\mu_{\bar{I}}$ and the set $T_{\bar{I}}(\bar{U})$. The scale "$\delta$" at which the proposition is applied is now $\bar{\delta}^{1/2}$, which in our notation (namely $\bar{\delta} = 2^{-mN}$) can be written as
\begin{displaymath} \bar{\delta}^{1/2} = 2^{-m \cdot (N/2)}. \end{displaymath}
Now, the conclusion \eqref{form227} of Proposition \ref{branchProp} would literally say something about the intersections of $T_{\bar{I}}(\bar{U})$ with dyadic intervals of lengths between $\bar{\delta}^{1/2}$ and $1$. In the following, we already translate this information back to $\bar{U}$, and scales between $\bar{\delta}$ and $\bar{\delta}^{1/2}$ (these correspond to indices $s \in \{N/2,\ldots,N - 1\}$): there exists a family of \emph{good indices} $\mathcal{G} \subset \{N/2,\ldots,N - 1\}$ such that
\begin{equation}\label{form229} \max_{I \in \mathcal{D}_{2^{-ms}}} N_{2^{-(s + 1)m}}(\bar{U} \cap I) \geq 2^{(1 - \rho)\alpha m + 3}, \qquad s \in \mathcal{G}, \end{equation} 
and 
\begin{displaymath} |\mathcal{G}| \stackrel{\eqref{form228}}{\geq} \left(1 - \left[\tfrac{27\bar{\zeta} + 4\epsilon}{\alpha \rho} + \tfrac{5C_{\alpha}}{\alpha \rho m}\right] \right) \cdot \tfrac{N}{2}. \end{displaymath} 
This is the lower bound we claimed in \eqref{form230}. Finally, a combination of \eqref{form229} and the inclusion \eqref{form224} finally allows us to estimate from below the branching numbers $R_{s}^{1}$. Assume that $s \in \mathcal{G} \subset \{N/2,\ldots,N - 1\}$, so that \eqref{form229} holds, and let $I \in \mathcal{D}_{sm}$ be some dyadic interval with 
\begin{align*} N_{2^{-m(s + 1)}}(U_{2\bar{\delta}} \cap (I + y_{0}\theta_{0})) & \stackrel{\eqref{form224}}{\geq} N_{2^{-m(s + 1)}}((\bar{U} + y_{0}\theta_{0}) \cap (I + y_{0}\theta_{0}))\\
& \,\,\, = \,\, N_{2^{-m(s + 1)}}(\bar{U} \cap I) \geq 2^{(1 - \rho)\alpha m + 3}. \end{align*} 
Since $2^{-m(s + 1)} \geq \bar{\delta}$, this implies that
\begin{displaymath}  \max_{I' \in \mathcal{D}_{2^{-ms}}} N_{2^{-m(s + 1)}}(U \cap I') \geq 2^{(1 - \rho)\alpha m}. \end{displaymath}
The left hand side is a lower bound for $R_{s}^{1}$, by the definition of these branching numbers. This proves \eqref{form217} for all $s \in \mathcal{G}$, and hence completes the proof of Proposition \ref{mainProp}.

\section{Proof of Proposition \ref{prop3}}\label{s:prop2}

We repeat the statement:

\begin{proposition}\label{prop2} Let $\theta \in [0,1]$, and let $1 \leq M \leq N < \infty$ be constants, let $0 < r \leq R \leq 1$, and let $\mu$ be a $(\gamma,C_{\gamma})$-regular measure with $\gamma \in [0,2]$, $C_{\gamma} > 0$, and $K := \spt \mu \subset \R^{2}$. Abbreviate $\mu_{s} := \mu|_{B(s)}$ for $s > 0$. Then, there exist absolute constants $c,C > 0$ such that
\begin{equation}\label{form27} \mu_{1}(H_{\theta}(CN,[r,1])) \leq \mu_{1}(H_{\theta}(cM,[4R,5])) + CC_{\gamma}^{2} \cdot \mu_{4}(H_{\theta}(c\tfrac{N}{M},[4r,7R])).  \end{equation} 
\end{proposition}
Here we abbreviated $H_{\theta}(K,M,[r,R]) =: H_{\theta}(M,[r,R])$. During the proof, we will also abbreviate $\m_{K,\theta} =: \m_{\theta}$. These notions were introduced in Definitions \ref{def:mult} and \ref{def:highMult}.

\begin{proof}[Proof of Proposition \ref{prop2}] For $t \in \R$, write 
\begin{equation}\label{form143} F_{r}(t) := N_{r}(B(2) \cap K_{r} \cap \pi_{\theta}^{-1}\{t\}) \quad \text{and} \quad F_{R}(t) := N_{R}(B(3) \cap K_{R} \cap \pi_{\theta}^{-1}\{t\}), \end{equation}
where $K_{s}$ refers to the $s$-neighbourhood of $K$. We also define the following variant of $F_{R}$:
\begin{displaymath} \widetilde{F}_{R}(t) := N_{4R}(B(4) \cap K_{4R} \cap \pi_{\theta}^{-1}\{t\}), \qquad t \in \R. \end{displaymath}
The first point to observe about the definition fo $F_{r}$ is that if $x \in B(1)$, then $B(x,1) \subset B(2)$, and hence
\begin{displaymath} \m_{\theta}(x \mid [r,1]) = N_{r}(B(x,1) \cap K_{r} \cap \pi_{\theta}^{-1}\{\pi_{\theta}(x)\}) \leq F_{r}(\pi_{\theta}(x)). \end{displaymath}
In particular,
\begin{equation}\label{form41} x \in B(1) \cap H_{\theta}(CN,[r,1]) \quad \Longrightarrow \quad F_{r}(\pi_{\theta}(x)) \geq CN. \end{equation}
Similarly, 
\begin{equation}\label{form42} x \in B(1) \, \setminus \, H_{\theta}(cM,[4R,5]) \quad \Longrightarrow \quad \widetilde{F}_{R}(\pi_{\theta}(x)) \leq cM, \end{equation}
because if $x \in B(1)$, then $B(4) \subset B(x,5)$, and hence
\begin{displaymath} \widetilde{F}_{R}(\pi_{\theta}(x)) = N_{4R}(B(4) \cap K_{4R} \cap \pi_{\theta}^{-1}\{\pi_{\theta}(x)\}) \leq \m_{\theta}(x \mid [4R,5]). \end{displaymath}
There is also a useful relationship between $F_{r}(t)$ and the push-forward measure $\mu_{\theta} := \pi_{\theta}\mu_{1}$, which reads as follows: if $I \subset \R$ is an interval of length $r$, then 
\begin{equation}\label{form11} \mu_{\theta}(I) \lesssim C_{\gamma}r^{\gamma} \cdot \sup_{t \in 3I} F_{r}(t). \end{equation} 
Indeed, by the $(\gamma,C_{\gamma})$-regularity of $\mu$, an upper bound for $\mu_{\theta}(I)$ is given by $4C_{\gamma}r^{\gamma} \cdot n$, where "$n$" is the largest number of disjoint $r$-discs centred at $B(1) \cap K \cap \pi_{\theta}^{-1}(I)$. All of these $r$-discs are contained in $B(2) \cap K_{r} \cap \pi_{\theta}^{-1}(3I)$. Now, the average line $\pi_{\theta}^{-1}\{t\}$, with $t \in 3I$, meets $\geq n/3$ of these discs, since the probability of hitting each disc individually is $\tfrac{1}{3}$. Hence, for some $t \in 3I$, it holds
\begin{displaymath} F_{r}(t) = N_{r}(B(2) \cap K_{r} \cap \pi_{\theta}^{-1}\{t\}) \gtrsim n. \end{displaymath}
This proves \eqref{form11}. 

Let $\mathcal{K}_{R}$ be a boundedly overlapping cover of $K_{r} \cap B(2)$ by discs of radius $R$, centred at $K$. Thus $B_{R} \subset K_{R} \cap B(3)$ for all $B_{R} \in \mathcal{K}_{R}$. For $B_{R} \in \mathcal{K}_{R}$ fixed, we define
\begin{displaymath} F_{r}(B_{R})(t) := N_{r}(B_{R} \cap K_{r} \cap \pi_{\theta}^{-1}\{t\}) \quad \text{and} \quad \widetilde{F}_{r}(B_{R})(t) := N_{4r}(4B_{R} \cap K_{4r} \cap \pi_{\theta}^{-1}\{t\}). \end{displaymath}
We claim the following inequality for every $t \in \R$:
\begin{equation}\label{form8} F_{r}(t) \leq C_{1}\sum_{B_{R} \in \mathcal{K}_{R}} F_{r}(B_{R})(t), \end{equation} 
where $C_{1} > 0$ is an absolute constant. This inequality means that an upper bound for the $r$-discs intersecting $\pi_{\theta}^{-1}\{t\}$ can be obtained by finding an upper bound on both $R$-discs intersecting $\pi_{\theta}^{-1}\{t\}$, and an upper bound for $r$-discs intersecting $\pi_{\theta}^{-1}\{t\}$ inside any given $R$-disc. To prove \eqref{form8}, fix $t \in \R$, and let $\{x_{1},\ldots,x_{m}\}$ be a maximal $r$-separated subset of $B(2) \cap K_{r} \cap \pi_{\theta}^{-1}\{t\}$, so that $F_{r}(t) \lesssim m$. For every $1 \leq j \leq m$, the point $x_{j} \in K_{r}$ lies in $B_{R} \cap K_{r} \cap \pi_{\theta}^{-1}\{t\}$ for some $B_{R} \in \mathcal{K}_{R}$. Consequently,
\begin{displaymath} F_{r}(t) \lesssim m \leq \sum_{B_{R} \in \mathcal{K}_{R}} |\{1 \leq j \leq m : x_{j} \in B_{R} \cap K_{r} \cap \pi_{\theta}^{-1}\{t\}\}| \lesssim \sum_{B_{R} \in \mathcal{K}_{R}} F_{r}(B_{R})(t), \end{displaymath}
where the final inequality used the $r$-separation of the points $x_{j}$. This proves \eqref{form8}.

Next we claim that, for every $t \in \R$,
\begin{equation}\label{form9} |\{B_{R} \in \mathcal{K}_{R} : F_{r}(B_{R})(t) \neq 0\}| \leq C_{2}F_{R}(t), \end{equation}
where $C_{2} > 0$ is another absolute constant. Indeed, for every $B_{R} \in \mathcal{K}_{R}$ with $F_{r}(B_{R})(t) \neq 0$, there exists $x_{B_{R}} \in B_{R} \cap K_{r} \cap \pi_{\theta}^{-1}\{t\} \subset B(3) \cap K_{R} \cap \pi_{\theta}^{-1}\{t\}$. Since the discs $B_{R}$ have bounded overlap, the points $x_{B_{R}}$ obtained this way are essentially $R$-separated (more precisely: contain an $R$-separated subset of comparable cardinality), and consequently
\begin{displaymath} F_{R}(t) = N_{R}(B(3) \cap K_{R} \cap \pi_{\theta}^{-1}\{t\}) \gtrsim |\{B_{R} \in \mathcal{K}_{R} : F_{r}(B_{R})(t) \neq \emptyset\}|, \end{displaymath}
as stated in \eqref{form9}.

We combine \eqref{form8} and \eqref{form9} to reach the following useful inequality, for any $H \geq 1$:
\begin{displaymath} F_{r}(t) \leq C_{1} \mathop{\sum_{B_{R} \in \mathcal{K}_{R}}}_{F_{r}(B_{R})(t) \geq H} F_{r}(B_{R})(t) + C_{1}C_{2}HF_{R}(t), \qquad t \in \R. \end{displaymath}
In particular, choosing $H := N/M$ and $C \geq 2C_{1}C_{2}$ (this "$C$" is the absolute constant referred to in the statement of the proposition), we find that if $F_{R}(t) \leq M$ for some $t \in \R$, then
\begin{displaymath} F_{r}(t) \leq C_{1} \mathop{\sum_{B_{R} \in \mathcal{K}_{R}}}_{F_{r}(B_{R})(t) \geq N/M} F_{r}(B_{R})(t) + CN/2. \end{displaymath} 
In particular, we derive the following key observation, again valid for any $t \in \R$:
\begin{equation}\label{form10} F_{r}(t) \geq CN \text{ and } F_{R}(t) \leq M \quad \Longrightarrow \quad F_{r}(t) \leq 2C_{1} \mathop{\sum_{B_{R} \in \mathcal{K}_{R}}}_{F_{r}(B_{R})(t) \geq N/M} F_{r}(B_{R})(t). \end{equation} 
To apply \eqref{form10}, we start by making the "trivial" estimate
\begin{displaymath} \mu_{1}(H_{\theta}(CN,[r,1])) \leq \mu_{1}(H_{\theta}(cM,[4R,5])) + \mu_{1}(H_{\theta}(CN,[r,1]) \, \setminus \, H_{\theta}(cM,[4R,5])). \end{displaymath}
Consequently, \eqref{form27} will follow once we manage to prove that
\begin{equation}\label{form28} \mu_{1}(H_{\theta}(CN,[r,1]) \, \setminus \, H_{\theta}(cM,[4R,5])) \lesssim C_{\gamma}^{2} \cdot \mu_{4}(H_{\theta}(c\tfrac{N}{M},[4r,7R])). \end{equation} 
Let $x \in B(1) \cap H_{\theta}(CN,[r,1]) \, \setminus \, H_{\theta}(cM,[4R,5])$ be arbitrary. Then, we infer from \eqref{form41}-\eqref{form42} that
\begin{displaymath} F_{r}(\pi_{\theta}(x)) \geq CN \quad \text{and} \quad \widetilde{F}_{R}(\pi_{\theta}(x)) \leq cM. \end{displaymath}
In particular, the left hand side of \eqref{form28} satisfies
\begin{equation}\label{form45} \mu_{1}(H_{\theta}(CN,[r,1]) \, \setminus \, H_{\theta}(cM,[4R,5])) \leq \int \mathbf{1}_{\{F_{r}(t) \geq CN \text{ and } \widetilde{F}_{R}(t) \leq cM\}}(t) \, d\mu_{\theta}(t). \end{equation} 
Let $\mathcal{D}_{r}(\R)$ be the collection of dyadic intervals of $\R$ of length $r$. We decompose the integral on the right as 
\begin{displaymath} \sum_{I \in \mathcal{D}_{r}(\R)} \int_{I} \mathbf{1}_{\{t : F_{r}(t) \geq CN \text{ and } \widetilde{F}_{R}(t) \leq cM\}}(t) \, d\mu_{\theta}(t). \end{displaymath}
Fix $I \in \mathcal{D}_{r}(\R)$, and assume that there exists at least one point $t_{0} \in I$ with $F_{r}(t_{0}) \geq CN$ and $\widetilde{F}_{R}(t_{0}) \leq cM$ (otherwise the corresponding term is zero). For such an interval $I \in \mathcal{D}_{r}(\R)$, we simply apply the estimate \eqref{form11} to evaluate
\begin{equation}\label{form18} \int_{I} \mathbf{1}_{\{t : F_{r}(t) \geq CN \text{ and } \widetilde{F}_{R}(t) \leq cM\}}(t) \, d\mu_{\theta}(t) \leq \mu_{\theta}(I) \lesssim C_{\gamma}r^{\gamma} \cdot \sup_{t \in 3I} F_{r}(t). \end{equation}
Let $t_{1} \in 3I$ be a point which nearly attains the supremum on the right, say, up to a constant $2$, and moreover $F_{r}(t_{1}) \geq F_{r}(t_{0}) \geq CN$. 

Recall from \eqref{form143} that $F_{R}(t) := N_{R}(B(3) \cap K_{R} \cap \pi_{\theta}^{-1}\{t\})$. We claim that if the absolute constant $c > 0$ is chosen small enough, then $F_{R}(t_{1}) \leq M$. Indeed, let $\{x_{1},\ldots,x_{m}\} \subset B(3) \cap \pi_{\theta}^{-1}\{t_{1}\} \cap K_{R}$ be a maximal $R$-separated set, with $m \sim F_{R}(t_{1})$. Then, since $|t_{1} - t_{0}| \leq 2r$, the line $\pi_{\theta}^{-1}\{t_{0}\}$ intersects $B(5) \cap K_{4R}$ in $\gtrsim m$ points, which are $4R$-separated. Consequently 
\begin{displaymath} cM \geq \widetilde{F}_{R}(t_{0}) = N_{4R}(B(5) \cap K_{4R} \cap \pi_{\theta}^{-1}\{t_{0}\}) \gtrsim m \sim F_{R}(t_{1}), \end{displaymath}
and the claim follows. Therefore, $F_{r}(t_{1}) \geq CN$ and $F_{R}(t_{1}) \leq M$, and we are in a position to apply \eqref{form10} to the point $t_{1}$:
\begin{equation}\label{form17} F_{r}(t_{1}) \leq 2C_{1} \mathop{\sum_{B_{R} \in \mathcal{K}_{R}}}_{F_{r}(B_{R})(t_{1}) \geq N/M} F_{r}(B_{R})(t_{1}). \end{equation} 
We next claim that if $B_{R} \in \mathcal{K}_{R}$ is one of the discs appearing in the sum in \eqref{form17}, that is, $F_{r}(B_{R})(t_{1}) \geq N/M$, then
\begin{equation}\label{form16} F_{r}(B_{R})(t_{1}) \lesssim \frac{C_{\gamma}}{r^{\gamma}}\int_{5I} \mathbf{1}_{\{s : \widetilde{F}_{r}(B_{R})(s) \geq cN/M\}}(s) \, d\pi_{\theta}(\mu|_{3B_{R}})(s). \end{equation} 
To prove this, we first claim that
\begin{equation}\label{form43} \pi_{\theta}(\mu|_{3B_{R}})([t_{1} - 2r,t_{1} + 2r]) \gtrsim \tfrac{r^{\gamma}}{C_{\gamma}} \cdot N_{r}(B_{R} \cap K_{r} \cap \pi_{\theta}^{-1}\{t_{1}\}) = \tfrac{r^{\gamma}}{C_{\gamma}} \cdot F_{r}(B_{R})(t_{1}). \end{equation}
To see this, note that every point $x \in B_{R} \cap K_{r} \cap \pi_{\theta}^{-1}\{t_{1}\}$ lies at distance $\leq r$ from a point in $2B_{R} \cap K \cap \pi_{\theta}^{-1}([t_{1} - r,t_{1} + r])$, and 
\begin{displaymath} B(x,r) \subset 3B_{R} \cap \pi_{\theta}^{-1}([t_{1} - 2r,t_{1} + 2r]), \end{displaymath}
since $r \leq R$. Hence, $\pi_{\theta}(\mu|_{3B_{R}})([t_{1} - 2r,t_{1} + 2r])$ exceeds, by a constant factor, the total $\mu$-measure of discs $B(x,r)$ obtained in this way. This measure is bounded from below by the right hand side of \eqref{form43}.

To deduce \eqref{form16} from \eqref{form43}, it remains to observe that
\begin{equation}\label{form249} [t_{1} - 2r,t_{1} + 2r] \subset \{s : \widetilde{F}_{r}(B_{R})(s) \geq cN/M\}, \end{equation}
assuming that the absolute constant $c > 0$ was chosen small enough. To see this, recall that
\begin{displaymath} \widetilde{F}_{r}(B_{R})(s) = N_{4r}(4B_{R} \cap K_{4r} \cap \pi_{\theta}^{-1}\{s\}). \end{displaymath}
The point is that since $F_{r}(B_{R})(t_{1}) \geq N/M$ (we are only considering these terms in \eqref{form17}), there exist $\gtrsim N/M$ points in $B_{R} \cap K_{r} \cap \pi_{\theta}^{-1}\{t_{1}\}$, which are $r$-separated. Now, if $s \in [t_{1} - 2r,t_{1} + 2r]$, then $\pi_{\theta}^{-1}\{s\}$ intersects $K_{4r}$ whenever $\pi_{\theta}^{-1}\{t_{1}\}$ intersects $K_{r}$, and these intersections occur inside $4B_{R}$. Hence $\widetilde{F}_{r}(B_{R})(s) \gtrsim F_{r}(B_{R})(t_{1}) \geq N/M$ for all $s \in [t_{1} - 2r,t_{1} + 2r]$. This completes the proof of \eqref{form249}, hence \eqref{form16}.

We record at this point that
\begin{equation}\label{form44} x \in 3B_{R} \text{ and } \widetilde{F}_{R}(B_{R})(\pi_{\theta}(x)) \geq cN/M \quad \Longrightarrow \quad x \in H_{\theta}(c\tfrac{N}{M},[4r,7R]). \end{equation}
This is so, because if $x \in 3B_{R}$, then $B(x,7R) \supset 4B_{R}$, and hence
\begin{displaymath} \m_{\theta}(x \mid [4r,7R]) \geq N_{4r}(4B_{R} \cap K_{4r} \cap \pi_{\theta}^{-1}\{\pi_{\theta}(x)\}) = \widetilde{F}_{r}(B_{R})(\pi_{\theta}(x)). \end{displaymath}
Combining \eqref{form18}-\eqref{form16}, we first learn that
\begin{displaymath} \int_{I} \mathbf{1}_{\{t : F_{r}(t) \geq CN \text{ and } \widetilde{F}_{R}(t) \leq cM\}}(t) \, d\mu_{\theta}(t) \lesssim C_{\gamma}^{2} \sum_{B_{R} \in \mathcal{K}_{R}} \int_{5I} \mathbf{1}_{\{s : \widetilde{F}_{r}(B_{R})(s) \geq cN/M\}}(s) \, d\pi_{\theta}(\mu|_{3B_{R}})(s).   \end{displaymath} 
Summing over $I \in \mathcal{D}_{r}(\R)$, using the bounded overlap of the intervals $5I$, applying \eqref{form44}, and finally using the bounded overlap of the discs $3B_{R} \subset B(4)$, $B_{R} \in \mathcal{K}_{R}$, we find that
\begin{align*} \int \mathbf{1}_{\{t : F_{r}(t) \geq CN \text{ and } \widetilde{F}_{R}(t) \leq cM\}}(t) \, d\mu_{\theta}(t) & \lesssim C_{\gamma}^{2} \sum_{B_{R} \in \mathcal{K}_{R}} \int_{\R} \mathbf{1}_{\{s : \widetilde{F}_{r}(B_{R})(s) \geq cN/M\}}(s) \, d\pi_{\theta}(\mu|_{3B_{R}})(s)\\
& \leq C_{\gamma}^{2} \sum_{B_{R} \in \mathcal{K}_{R}} \mu(3B_{R} \cap H_{\theta}(c\tfrac{N}{M},[4r,7R]))\\
& \lesssim C_{\gamma}^{2} \cdot \mu(B(4) \cap H_{\theta}(c\tfrac{N}{M},[4r,7R])). \end{align*} 
Recalling \eqref{form45}, this concludes the proof of \eqref{form28}, and the proof of the proposition. \end{proof}

\bibliographystyle{plain}
\bibliography{references}

\end{document}